\documentclass[a4paper]{amsart}

\usepackage{amsmath,amssymb,url}
\usepackage[latin1]{inputenc}
\usepackage{ifthen}
\usepackage{enumerate}

\numberwithin{equation}{section}

\theoremstyle{plain}
\newtheorem{theorem}{Theorem}[section]
\newtheorem{lemma}[theorem]{Lemma}
\newtheorem{proposition}[theorem]{Proposition}
\newtheorem{corollary}[theorem]{Corollary}

\theoremstyle{definition}

\newtheorem{remark}[theorem]{Remark}
\newtheorem{example}[theorem]{Example}
\newtheorem*{question}{Question}

\DeclareMathOperator{\Pol}{Pol}
\DeclareMathOperator{\Clo}{Clo}
\DeclareMathOperator{\range}{Im}
\DeclareMathOperator{\Aut}{Aut}

\newcommand{\vect}[1]{\ensuremath{\mathbf{#1}}}
\newcommand{\card}[1]{\ensuremath{\lvert{#1}\rvert}}
\newcommand{\subf}[1][]{\ifthenelse{\equal{#1}{}}{\ensuremath{\leq}}{\ensuremath{\leq_{#1}}}}
\newcommand{\fequiv}[1][]{\ifthenelse{\equal{#1}{}}{\ensuremath{\equiv}}{\ensuremath{\equiv_{#1}}}}
\newcommand{\cl}[1]{\ensuremath{\mathcal{#1}}}
\newcommand{\al}[1]{\ensuremath{\mathbf{#1}}}
\newcommand{\GL}{\ensuremath{\mathrm{GL}}}
\newcommand{\AGL}{\ensuremath{\mathrm{AGL}}}
\newcommand{\FF}{\ensuremath{\mathfrak{F}}} 

\newcommand{\FFF}{\ensuremath{\mathbb{F}}}
\newcommand{\hh}{\ensuremath{\text{\sf h}}}

\newcommand{\leaf}{\ensuremath{\mathrm{min}}}
\newcommand{\ttop}{\ensuremath{\mathrm{max}}}
\newcommand{\iso}[1]{\ensuremath{(#1)^{\mathrm{iso}}}}
\newcommand{\zero}{\ensuremath{\text{\bf 0}}}
\newcommand{\one}{\ensuremath{\text{\bf 1}}}
\newcommand{\groupone}{\ensuremath{\texttt{1}}}
\newcommand{\mmod}{\ensuremath{\,\,\mathrm{mod}\,\,}}

\renewcommand{\phi}{\varphi}
\renewcommand{\epsilon}{\varepsilon}

\begin{document}


\title{Clones with finitely many relative ${\mathcal R}$-classes}

\author{Erkko Lehtonen}
\address[E. Lehtonen]{Université du Luxembourg \\
Faculté des Sciences, de la Technologie et de la Communication \\
6, rue Richard Coudenhove-Kalergi \\
L--1359 Luxembourg \\
Luxembourg}
\email{erkko.lehtonen@uni.lu}

\author{\'Agnes Szendrei}
\address[\'A. Szendrei]{Department of Mathematics \\
University of Colorado at Boulder \\
Campus Box 395 \\
Boulder, CO 80309-0395 \\
USA
\and
Bolyai Institute \\
Aradi v\'ertan\'uk tere 1 \\
H--6720 Szeged \\
Hungary }
\email{szendrei@euclid.colorado.edu}

\thanks{This material is based upon work supported by
the Academy of Finland grant no.\ 120307
and by
the Hungarian National Foundation for Scientific Research (OTKA) grants no.\  T~048809, K60148, and K77409.
}



\begin{abstract}
For each clone $\cl{C}$ on a set $A$ there is an associated equivalence relation
analogous to Green's ${\mathcal R}$-relation, which relates two operations on $A$ if and only if each one is a substitution instance of the other using operations from $\cl{C}$. 
We study the clones for which there are only finitely many relative ${\mathcal R}$-classes.
\end{abstract}

\maketitle


\section{Introduction}
\label{intro}

Green's relations play a central role
in semigroup theory. Two elements $a,b$ of a monoid $M$ are 
related by Green's ${\mathcal R}$-relation
if and only if they generate the same right ideal $aM=bM$.
In particular, if $M$ is a transformation monoid on a set $A$, then
two elements $f=f(x)$ and $g=g(x)$ of $M$ are 
${\mathcal R}$-related exactly when
$f\bigl(h_1(x)\bigr)=g(x)$ and
$g\bigl(h_2(x)\bigr)=f(x)$ for some
$h_1,h_2\in M$, that is,
each one of $f,g$ is a substitution instance of the other
by transformations from $M$.
For example, if $M=T_A$ is the full transformation monoid on $A$, then
$f\,{\mathcal R}\,g$ if and only if $f,g$ have the same range.

Henno~\cite{Henno} generalized Green's relations to
Menger algebras (essentially, abstract clones, the multi-variable versions
of monoids), and described Green's relations on the clone
$\cl{O}_A$ of all operations on $A$ for each set $A$.
He proved that two finitary operations on $A$ are ${\mathcal R}$-related
if and only if they have the same range.

Relativized versions of Green's ${\mathcal R}$-relation on the clone 
$\cl{O}_{\{0,1\}}$ of Boolean functions 
have been used in computer science to classify Boolean functions.
In~\cite{Wang} and~\cite{WW} a Boolean function $g$ is defined to be
a {\it minor} of another Boolean function $f$ if and only if
$g$ can be obtained from $f$ by substituting for each variable of $f$ 
a variable, a negated variable,
or one of the constants $0$ or $1$.
A more restrictive notion of Boolean minor, namely
when negated variables are not allowed,
is employed in~\cite{FH} and~\cite{Zverovich}, while in the paper
\cite{Harrison} two $n$-ary Boolean functions are considered
equivalent if they are substitution
instances of each other with respect to
the general linear group $\GL(n,\FFF_2)$
or the affine general linear group $\AGL(n,\FFF_2)$
where $\FFF_2$ is the two-element field.

The notions of `minor' and `${\mathcal R}$-equivalence'
for operations on a set $A$
can be defined relative to any subclone $\cl{C}$ of $\cl{O}_A$
as follows:
for $f,g\in\cl{O}_A$, $g$ is a $\cl{C}$-minor of $f$
if $g$ can be obtained
from $f$ by substituting operations from $\cl{C}$
for the variables of $f$,
and $g$ is $\cl{C}$-equivalent to $f$ if
$f$ and $g$ are $\cl{C}$-minors of each other.
Thus, for example, Henno's ${\mathcal R}$-relation on $\cl{O}_A$
is nothing else than
$\cl{O}_A$-equivalence, and the concepts of Boolean minor
mentioned in the preceding paragraph
are the special cases of the notion of
$\cl{C}$-minor where $\cl{C}$ is the essentially unary clone of Boolean functions generated by negation and the two constants, or by the two constants only.
Further applications of ${\mathcal C}$-minors and ${\mathcal C}$-equivalence
where ${\mathcal C}$ is a clone of essentially unary operations can be found
in~\cite{CP}, \cite{EFHH}, and~\cite{Pippenger}.

The question we are interested in is the following:

\begin{question}
For which clones $\cl{C}$ are there only finitely many relative 
${\mathcal R}$-classes?
\end{question}

\noindent
That is, we want to know for which clones $\cl{C}$ it is the case that
the $\cl{C}$-equivalence relation on $\cl{O}_A$ 
has only finitely many equivalence classes.
Let $\FF_A$ denote the set of all such clones on $A$.
It is easy to see that $\cl C$-equivalent operations have the same range, 
therefore if $A$ is infinite, then there will be infinitely many
$\cl{C}$-equivalence classes for every clone $\cl{C}$, so $\FF_A$ is empty.
If $A$ is finite, then the result of Henno~\cite{Henno} mentioned above
implies that $\cl{O}_A\in\FF_A$.
It is not hard to see that $\FF_A$ is an order filter
(up-closed set) in the lattice of all clones on $A$ 
(Proposition~\ref{basic_props2}). 
Moreover, if $|A|>1$ then the clone $\cl{P}_A$ of projections fails to 
belong to $\FF_A$, because $\cl{P}_A$-equivalent 
operations have the same essential arity (i.e., depend on the same number 
of variables), and on a set with more than one element there exist operations 
of arbitrarily large essential arity.
Thus the order filter $\FF_A$ is proper.

The results of this paper show that the family $\FF_A$ of clones
is quite restricted. Every clone $\cl{C}$ in $\FF_A$ has to be `large'
quantitatively in the sense that it contains a lot of $n$-ary operations 
for each $n$ (Proposition~\ref{BigClones}), 
and it has to be `large' in the sense that
there are strong restrictions on the relations that are invariant
with respect to the operations in $\cl{C}$ (Corollary~\ref{basic_props4}).

There is a rich literature of classification results for `large' subclones
of $\cl{O}_A$ when $A$ is finite (see~\cite{Lau2006} and the references there)
where `large' is usually taken to mean `near the top of the lattice of clones
on $A$'.
Our interest in the order filter $\FF_A$ stems from the fact that
the property of being in $\FF_A$ is a different kind of `largeness'.
Since the family $\FF_A$ is quite restricted, the clones in $\FF_A$
may be classifiable. At the same time, $\FF_A$ contains interesting 
families of clones: e.g., all discriminator clones (\cite{LS}, see 
Theorem~\ref{discr})
and all clones determined by a chain of 
equivalence relations on $A$
together with a set of invariant permutations 
and an arbitrary family of subsets of $A$ 
(Theorem~\ref{mainthm-eqrels}).

Using Rosenberg's description of the maximal clones
$\cl{M}=\Pol\rho$ on a finite set $A$ (see Theorem~\ref{thm-Rosenberg}) 
we determine which maximal clones belong to 
$\FF_A$ (see
Theorem~\ref{thm-max-summary} and
Table~\ref{table-maxcl}). 
Furthermore, for each maximal clone ${\mathcal M}$ 
that belongs to $\FF_A$ we find families of subclones
of ${\mathcal M}$ that also belong to $\FF_A$.
We also investigate which intersections of maximal clones
are in $\FF_A$.

\begin{table}
\begin{tabular}{|l|c|c|}
\hline
\multicolumn{1}{|c|}{$\rho$ ($h$-ary)} & $\Pol \rho \stackrel{?}{\in} \FF_A$ & Proof\\
\hline
bounded partial order & no & Thm~\ref{order_affine}\\
prime permutation & yes & Cor~\ref{discr-max}\\
nontrivial equivalence relation & yes & Thm~\ref{mainthm-eqrels}\\
prime affine relation & no & Thm~\ref{order_affine}\\
central relation & & \\
\quad $h = 1$ & yes & Cor~\ref{discr-max}\\
\quad $2 \leq h \leq |A| - 2$ & no & Thm~\ref{thm-central-r}\\
\quad $h = |A| - 1$ & yes & Thm~\ref{thm-centralk-1}\\
$h$-regular relation &  & \\
\quad $h < |A|$ & no & Thm~\ref{thm-hreg}\\
\quad $h = |A|$ & yes & Thm~\ref{thm-slup}\\
\hline
\end{tabular}

~

\caption{The membership of the maximal clones in $\FF_A$.}
\label{table-maxcl}
\end{table}

\section{Preliminaries}
\label{sec:prelim}

Let $A$ be a fixed nonempty set. If $n$ is a positive integer, then by an $n$-ary \emph{operation} on $A$ we mean a function $A^n\to A$, and we will refer to $n$ as the \emph{arity} of the operation. The set of all $n$-ary operations on $A$ will be denoted by $\cl O_A^{(n)}$, and we will write $\cl O_A$ for the set of all finitary operations on $A$. For $1 \leq i \leq n$ the $i$-th $n$-ary \emph{projection} is the operation $p_i^{(n)} \colon A^n \to A,\ (a_1, \ldots, a_n) \mapsto a_i$. 

For arbitrary positive integers $m$ and $n$ there is a one-to-one correspondence between
the functions $f \colon A^n \to A^m$ and the $m$-tuples  $\vect{f} = (f_1, \ldots, f_m)$ 
of functions $f_i\colon A^n \to A$ ($i = 1, \ldots, m$) via the correspondence
\[
f\mapsto \vect{f} = (f_1, \ldots, f_m)
\quad\text{with}\quad
f_i = p_i^{(m)} \circ f 
\text{ for all $i = 1, \ldots, m$}.
\]
In particular, $\vect{p}^{(n)} = (p_1^{(n)}, \ldots, p_n^{(n)})$ corresponds to the identity function $A^n \to A^n$. {}From now on we will identify each function $f \colon A^n \to A^m$ with the corresponding $m$-tuple $\vect{f} = (f_1, \ldots, f_m) \in (\cl{O}_A^{(n)})^m$ of $n$-ary operations. Using this convention the \emph{composition} of two functions $\vect{f} = (f_1, \ldots, f_m) \colon A^n \to A^m$ and $\vect{g} = (g_1, \ldots, g_k) \colon A^m \to A^k$ can be 
described as follows:
\[
\vect{g} \circ \vect{f} =
(g_1 \circ \vect{f}, \ldots, g_k \circ \vect{f}) =
\bigl(g_1(f_1, \ldots, f_m), \ldots, g_k(f_1, \ldots, f_m)\bigr)
\]
where
\[
g_i(f_1, \ldots, f_m)(\vect{a}) = 
g_i \bigl(f_1(\vect{a}), \ldots, f_m(\vect{a})\bigr)
\qquad
\text{for all $\vect{a} \in A^n$ and for all $i$}.
\]
A \emph{clone} on $A$ is a subset $\cl C$ of $\cl{O}_A$ that contains the projections and is closed under composition; that is, $p_i^{(n)}\in\cl C$  for all $1 \leq i \leq n$ and
$g \circ \vect{f} \in \cl{C}^{(n)}$ whenever $g \in \cl C^{(m)}$ and $\vect{f} \in (\cl{C}^{(n)})^m$  ($m,n\ge1$). 
The clones on $A$ form a complete lattice under inclusion. 
Therefore for each set $F \subseteq \cl{O}_A$ of operations there exists a smallest clone that contains $F$, which will be denoted by $\langle F \rangle$ and will be referred to as the \emph{clone generated by} $F$.

Clones can also be described via invariant relations.
For an $n$-ary operation $f \in \cl{O}_A^{(n)}$ and an $r$-ary relation $\rho$ on $A$ we say that
$f$ \emph{preserves}  $\rho$ (or $\rho$ is \emph{invariant} under $f$, or $f$ is a \emph{polymorphism} of $\rho$), if whenever $f$ is applied coordinatewise to $r$-tuples from $\rho$, the resulting $r$-tuple belongs to $\rho$.
If $\rho$ is an $r$-ary relation on $A$ and $n$ is a positive integer,
$\rho^n$ will denote the $r$-ary relation ``coordinatewise $\rho$-related''
on $A^n$; more formally, for arbitrary $n$-tuples
$\vect{a}_i=(a_{i1},\ldots,a_{in})\in A^n$ ($1\le i\le r$)
\[
(\vect{a}_1,\ldots,\vect{a}_r)\in\rho^n
\quad\iff\quad
(a_{1j},\ldots,a_{rj})\in\rho
\text{ for all $j$ ($1\le j\le n$)}.
\] 
We will say that $\vect{f}=(f_1,\ldots,f_m)\in(\cl{O}_A^{(n)})^m$
\emph{preserves} an $r$-ary relation $\rho$ on $A$ if each
$f_i$ ($1\le i\le m$) does; that is,
\[
(\vect{a}_1,\ldots,\vect{a}_r)\in\rho^n
\quad\implies\quad
\bigl(\vect{f}(\vect{a}_1),\ldots,\vect{f}(\vect{a}_r)\bigr)\in\rho^m
\text{ for all $\vect{a}_1,\ldots,\vect{a}_r\in A^n$}.
\]
For any family $R$ of (finitary) relations on $A$, the set $\Pol R$ of all operations $f\in\cl{O}_A$
that preserve every relation in $R$ is easily seen to be a clone on $A$.
Moreover, if $A$ is finite, then it is a well-known fact that every clone on $A$ is of the form
$\Pol R$ for some family of relations on $A$ (see, e.g., 
\cite{BKKR,Geiger,Lau2006,PK,Szendrei}).
If $R=\{\rho\}$, we will write $\Pol\rho$ for $\Pol\,\{\rho\}$.  

Throughout the paper 
we will use the following additional notation 
concerning operations and relations.
The constant tuple $(a, \dotsc, a)$ 
of any length is denoted by $\bar a$
(the length will be clear from the context). 
If $\theta$ is an equivalence relation on $A$,
then the equivalence class containing $a\in A$ is denoted by
$a/\theta$.
For any operation $f$ on $A$ that preserves $\theta$, 
$f^\theta$ denotes the natural action of $f$ on the set $A/\theta$
of $\theta$-classes.
Furthermore, for any set $F$ of operations contained in $\Pol \theta$ we write $F^\theta$ for the set $\{f^\theta : f \in F\}$. 
The range of an arbitrary function $\phi$ will be denoted by $\range\phi$.

Now let $\cl{C}$ be a fixed clone on a set $A$ of any cardinality. 
For arbitrary operations $f \in \cl{O}_A^{(n)}$ and $g \in \cl{O}_A^{(m)}$ we say that 
\begin{itemize}
\item
$f$ is a \emph{$\cl{C}$-minor} of $g$, in symbols $f \subf[\cl{C}] g$, if $f = g \circ \vect{h}$ for some $\vect{h} \in (\cl C^{(n)})^m$; 
\item
$f$ and $g$ are \emph{$\cl{C}$-equivalent,} in symbols $f \fequiv[\cl{C}] g$, if $f \subf[\cl{C}] g$ and $g \subf[\cl{C}] f$.
\end{itemize}
It is easy to verify (see~\cite{LS}) that $\subf[\cl{C}]$ is a quasiorder on $\cl{O}_A$, and
hence $\fequiv[\cl{C}]$, the intersection of $\subf[\cl{C}]$ with its converse, is an 
equivalence relation on $\cl{O}_A$.

$\FF_A$ will denote the collection of all clones $\cl C$ on $A$ such that the equivalence relation $\fequiv[\cl{C}]$ has only finitely many equivalence classes. As we discussed in the Introduction, if $A$ is infinite, then $\FF_A=\emptyset$, while if $A$ is finite and $|A|>1$, then the clone $\cl{O}_A$ of all operations is in $\FF_A$, and the clone $\cl{P}_A$ of projections is not.

{}From now on we will assume that $A$ is finite.
The next proposition contains some useful basic facts about $\FF_A$.

\begin{proposition}[\cite{LS}]
\label{basic_props2}
Let $\cl{C}$ be a clone on a finite set $A$.
\begin{enumerate}[\rm(i)]
\item
$\cl{C}\in\FF_A$  if and only if there exists an integer $d > 0$ such that every operation on $A$ is $\cl{C}$-equivalent to a $d$-ary operation on $A$.
\item
$\FF_A$ is an order filter in the lattice of all clones on $A$; that is, if
$\cl{C}\in\FF_A$, then $\cl{C}'\in\FF_A$ for every clone $\cl{C}'$ that contains $\cl{C}$.
\end{enumerate}
\end{proposition}

It is well known that every clone on $A$ other than $\cl{O}_A$ is contained in a maximal clone.
Since $\cl{O}_A\in\FF_A$ and $\FF_A$ is an order filter of clones on $A$, it is natural to ask
which maximal clones belong to $\FF_A$. 
To answer this question we will use Rosenberg's description of the maximal clones.

\begin{theorem}[Rosenberg~\cite{Rosenberg}]
\label{thm-Rosenberg}
For each finite set $A$ with $\card{A} \geq 2$ the maximal clones on $A$ are
the clones of the form $\Pol \rho$ where $\rho$ is a relation of one of the following six types:
\begin{enumerate}[\indent\rm(1)]
\item a bounded partial order on $A$,
\item a prime permutation on $A$,
\item a prime affine relation on $A$,
\item a nontrivial equivalence relation on $A$,
\item a central relation on $A$,
\item an $h$-regular relation on $A$.
\end{enumerate}
\end{theorem}

Here a partial order on $A$ is called 
\emph{bounded} if it has both a least and a greatest element. 
A \emph{prime permutation} on $A$ is (the graph of) a fixed point free permutation on $A$ in which all cycles are of the same prime length, and 
a \emph{prime affine relation} on $A$ is the graph of the ternary operation $x-y+z$  for some elementary abelian $p$-group $(A;+,-,0)$ on $A$ ($p$ prime).
An equivalence relation on $A$ is called \emph{nontrivial} 
if it is neither the equality relation ${\zero}_A$ on $A$
nor the full relation ${\one}_A$ on $A$. 

To describe central relations and $h$-regular relations 
we call an $h$-ary relation $\rho$ on $A$ \emph{totally reflexive} if 
$\rho$ contains all 
$h$-tuples from $A^h$ 
whose coordinates are not pairwise distinct, and
\emph{totally symmetric} if $\rho$ is invariant under any permutation of its coordinates.
We say that $\rho$ is a \emph{central relation} on $A$ if $\emptyset
\neq \rho \neq A^h$, $\rho$ is totally reflexive and totally
symmetric, and there exists an element $c\in A$ such that
$\{c\}\times A^{h-1}\subseteq\rho$. The elements $c$ with this property are called the
\emph{central elements} of $\rho$.
Note that the arity $h$ of a central relation on $A$ has to satisfy $1\le h\le|A|-1$, 
and the unary central relations are just the nonempty proper subsets of $A$.
 
For an integer $h \geq 3$ a family $T = \{\theta_1, \dotsc,
\theta_r\}$ ($r\ge1$) of equivalence relations on $A$ is called
\emph{$h$-regular} if each $\theta_i$ ($1 \leq i \leq r$) has exactly
$h$ blocks, and for arbitrary blocks $B_i$ of $\theta_i$ ($1 \leq i
\leq r$) the intersection $\bigcap_{i=1}^r B_i$ is nonempty. 
To each  $h$-regular family $T = \{\theta_1, \dotsc, \theta_r\}$  of equivalence relations on $A$
we associate  an $h$-ary relation $\lambda_T$ on $A$ as follows:
\begin{multline*}
\lambda_T = \{(a_1, \dotsc, a_h) \in A^h :
\text{for each $i$, $a_1, \dotsc, a_h$ is not a transversal}\\
\text{for the blocks of $\theta_i$}\}.
\end{multline*}
Relations of the form $\lambda_T$ are called \emph{$h$-regular} (or \emph{$h$-regularly generated}) \emph{relations}.
It is clear from the definition that $h$-regular relations are totally reflexive and totally symmetric, their arity $h$ satisfies $3\le h\le|A|$, and
$h=|A|$ holds if and only if $T$ is the one-element family consisting of the
equality relation.

We conclude this section by summarizing earlier known results
proving 
some of the maximal clones from Theorem~\ref{thm-Rosenberg} 
to belong or not to belong to $\FF_A$.

\begin{theorem}[\cite{LS}]
\label{discr}
Let $A$ be a finite set with $|A|\ge2$.
\begin{enumerate}[\indent\rm(i)]
\item
The clone on $A$ generated by the ternary discriminator function
\[
t_A(x,y,z) =
\begin{cases}
z, & \text{if $x = y$,} \\
x, & \text{otherwise}
\end{cases}
\quad
(x, y, z \in A)
\]
is a minimal member of $\FF_A$. 
Hence every clone containing $t_A$ belongs to $\FF_A$.
\item
If $|A|=2$, then a clone is in $\FF_A$ if and only if it contains $t_A$.
\end{enumerate}
\end{theorem}

It is well known and easy to check that every maximal clone determined by a prime permutation on $A$ or by a  
proper subset of $A$ contains $t_A$.
Therefore we get the following corollary.

\begin{corollary}
\label{discr-max}
Every maximal clone determined by a prime permutation on $A$ or by a  
proper subset of $A$ (i.e., a unary central relation on $A$) belongs to $\FF_A$.\end{corollary}

\begin{theorem}[\cite{ULM}]
\label{order_affine}
If $A$ is a finite set with $|A|\ge2$, then
the maximal clones determined by bounded partial orders or by prime affine relations do not belong to $\FF_A$.
\end{theorem}

\section{Two necessary conditions}
\label{sec:2neccond}

In this section we establish some necessary conditions for a clone
$\cl{C}$ on a finite set $A$ to belong to $\FF_A$.
The first condition shows that for $\cl{C}\in\FF_A$ it is necessary that
for each subset $B$ of $A$, the operations from $\cl{C}$ restrict to
$B$ so that the restrictions that are operations on $B$
form a clone belonging to $\FF_B$.

\begin{proposition}
\label{basic_props3}
Let $\cl{C}$ be a clone on a finite set $A$, let $B$ be a nonempty subset of $A$, and let
$\cl{C}_B$ be the clone on $B$ defined as follows:
\[
\cl{C}_B = \{ f|_B : f \in \cl{C} \cap \Pol B \}.
\]
If $\cl{C}\in\FF_A$, then $\cl{C}_B\in\FF_B$.
\end{proposition}
\begin{proof}
We will prove the contrapositive,
so suppose that $\cl{C}_B\notin\FF_B$.
Our goal is to show that $\cl{C}\notin\FF_A$.  
Since $\cl{C}_A=\cl{C}$,
there is nothing to prove if $B = A$.
Therefore let us assume
that $B$ is a proper subset of $A$, and let $0 \in A \setminus B$. 
Using the assumption $\cl{C}_B\notin\FF_B$, 
select representatives $g_i$ ($i = 1, 2, \ldots$)
of infinitely many different $\fequiv[\cl{C}_B]$-classes. Define $f_i$ on $A$ such that $f_i(\vect{x}) = g_i(\vect{x})$ if all coordinates of the tuple $\vect{x}$ are in $B$ and $f_i(\vect{x}) = 0$ otherwise. 
We will prove $\cl{C}\notin\FF_A$ by showing that the operations $f_i$ ($i = 1, 2, \ldots$) belong to pairwise different $\fequiv[\cl{C}]$-classes.

Suppose that there exist operations $f_i,f_j$ ($i\not=j$) such that 
$f_i \fequiv[\cl{C}] f_j$,
that is, $f_i = f_j \circ \vect{h}$ and $f_j = f_i \circ \vect{h}'$ 
for some $\vect{h} \in (\cl{C}^{(m)})^{n}$ and 
$\vect{h}' \in (\cl{C}^{(n)})^{m}$ 
where $m$ is the arity of $f_i$ and $n$ is the arity of $f_j$. 
Let $\vect{x} \in B^m$. 
Then $f_i(\vect{x}) = g_i(\vect{x}) \in B$, so 
$f_j(\vect{h}(\vect{x})) = (f_j \circ \vect{h})(\vect{x})=f_i(\vect{x}) 
\in B$. 
Since $f_j(\vect{y}) = 0 \notin B$ if $\vect{y} \notin B^n$, we get that $\vect{h}(\vect{x}) \in B^n$. This shows that $\vect{h}$ preserves $B$, hence
$\vect{h}|_B \in (\cl{C}_B^{(m)})^n$.
Similarly, by interchanging the roles of $f_i$ and $f_j$
we conclude that $\vect{h}'$ also preserves $B$, and
$\vect{h}'|_B \in (\cl{C}_B^{(n)})^m$.
By construction, $f_i$, $f_j$ preserve $B$ as well, 
therefore $f_i|_B = f_j|_B \circ \vect{h}|_B$
and
$f_j|_B = f_i|_B \circ \vect{h}'|_B$.    
This implies that
$g_i = g_j \circ \vect{h}|_B$ and $g_j = g_i \circ \vect{h}'|_B$.
Hence $g_i\fequiv[\cl{C}_B] g_j$, which
contradicts the choice of the operations $g_i,g_j$.
\end{proof}

\begin{corollary}
\label{basic_props4}
Let $\rho$ be a relation on a finite set $A$. 
If $A$ has a nonempty subset $B$ such that for the clone determined by the
restriction $\rho|_B$ of $\rho$ to $B$ we have 
that $\Pol \rho|_B\notin\FF_B$, then
$\Pol \rho\notin\FF_A$.
\end{corollary}
\begin{proof}
Let $\cl{C} = \Pol \rho$, and let
$\cl{C}_B$ be the clone defined in Proposition~\ref{basic_props3}.
First we will show that $\cl{C}_B\subseteq\Pol\rho|_B$. 
Indeed,
every operation in $\cl{C}_B$ is of the form $f|_B$ for some
$f\in\cl{C}\cap\Pol{B}=\Pol\{\rho,B\}$.
Since $f$ preserves $\rho$ and $B$, it also preserves
$\rho\cap B^2=\rho|_B$.
Thus $f|_B$ also preserves $\rho|_B$, that is,
$f|_B\in\Pol\rho|_B$.

If $\Pol \rho|_B\notin\FF_B$, then
the fact that $\cl{C}_B$ is a subclone of $\Pol \rho|_B$
implies by Proposition~\ref{basic_props2}~(ii) that
$\cl{C}_B\notin\FF_B$.
Therefore it follows from
Proposition~\ref{basic_props3}
that $\Pol\rho=\cl{C}\notin\FF_A$, as claimed.
\end{proof}

The second necessary condition for $\cl{C}\in\FF_A$ is a quantitative condition 
indicating that the clones in $\FF_A$ are large in the sense that they must have a lot of
$n$-ary operations for each $n$. 

\begin{proposition}
\label{BigClones}
Let $A$ be a $k$-element set.
If $\cl{C}\in\FF_A$, then there exists a positive constant $c$
such that $\card{\cl{C}^{(n)}} \geq c k^{k^n \!/ n}$ for all $n \geq 1$.
\end{proposition}
\begin{proof}
Denote the number of $\fequiv[\cl{C}]$-classes by $\mu$. For every $n\geq 1$
the number of $n$-ary operations on $A$ is $k^{k^n}$, therefore 
there must be a $\fequiv[\cl{C}]$-class $B$ such that $\card{B^{(n)}} \geq k^{k^n} / \mu$. Any $f \in B^{(n)}$ has at most $\card{\cl{C}^{(n)}}^n$ $n$-ary $\cl{C}$-minors, so we have that
\[
\card{\cl{C}^{(n)}}^n \geq 
\card{ \{f \circ \vect{g} : \vect{g} \in (\cl{C}^{(n)})^n\} }
\geq \card{B^{(n)}} \geq k^{k^n} / \mu.
\]
It follows that
\[
\card{\cl{C}^{(n)}} \geq (k^{k^n} / \mu)^{1/n} = k^{k^n \!/ n} / \mu^{1/n} \geq k^{k^n \!/ n} / \mu.
\]
The claim now follows by letting $c = 1 / \mu$.
\end{proof}

\begin{corollary}
\label{SmallClones}
Let $\cl{C}$ be a clone on a $k$-element set $A$. If $\card{\cl{C}^{(n)}} \leq k^{p(n)}$ for all $n$, where $p \colon \mathbb{N} \to \mathbb{N}$ is a function such that 
$\displaystyle \lim_{n \to \infty} \frac{p(n)}{k^n/n} = 0$,
then $\cl C\notin\FF_A$.
\end{corollary}

\begin{proof}
Suppose that the assumptions of the corollary hold, but  $\cl C\in\FF_A$. Proposition~\ref{BigClones} implies then that for some positive constant $c$ we have $ck^{k^n/n} \leq \card{\cl{C}^{(n)}}$ for all $n$. Hence $ck^{k^n/n} \leq k^{p(n)}$ for all $n$. Since $k > 1$, we get that $\log_k c + k^n/n \leq p(n)$ for all $n$, or equivalently,
$\displaystyle \frac{\log_k c}{k^n/n} + 1 \leq \frac{p(n)}{k^n/n}$ for all $n$.
Taking the limit of both sides as 
$n \to \infty$ 
we get that $1 \leq 0$, a contradiction.
\end{proof}

Every polynomial function $p$ satisfies the condition $\displaystyle \lim_{n \to \infty} \frac{p(n)}{k^n/n} = 0$, hence the following statement
is a special case of Corollary~\ref{SmallClones}.

\begin{corollary}
\label{PolyClones}
If $\cl{C}$ is a clone on a $k$-element set $A$ such that 
for some polynomial function $p$ we have $\card{\cl{C}^{(n)}} \leq k^{p(n)}$ 
for all $n$, 
then $\cl C\notin\FF_A$.
\end{corollary}

\begin{remark}
\label{notconverse}
The converse of Proposition~\ref{BigClones} is not true, that is, there
exist clones $\cl C\notin\FF_A$ that 
satisfy the conclusion of Proposition~\ref{BigClones}. 
For example, if $A=\{0,1\}$ is a $2$-element set and $\le$ is the 
natural order $0\le1$ on $A$, then it follows from part (ii) of 
Theorem~\ref{discr} that the clone $\cl{M}:=\Pol\le$ 
of all monotone Boolean functions is not in $\FF_A$. However, 
Gilbert~\cite{gilbert} proved that 
\[|\cl{M}^{(n)}|\ge 2^{\binom{n}{\lfloor{n/2}\rfloor}}
\qquad\text{for all $n\ge1$.}
\]
If $n\ge2$, then $\binom{n}{\lfloor{n/2}\rfloor}\ge
\binom{n}{k}$ for all $0<k<n$ and
$\binom{n}{\lfloor{n/2}\rfloor}\ge\binom{n}{0}+\binom{n}{n}$, therefore
$\binom{n}{\lfloor{n/2}\rfloor}\ge2^n/n$. 
Hence $|\cl{M}^{(n)}|\ge2^{2^n/n}$ holds for all $n\ge2$.
For $n=1$ we have $|\cl{M}^{(1)}|=3\ge \frac{1}{2}\cdot 2^{2^1/1}$.
Thus $|\cl{M}^{(n)}|\ge \frac{1}{2}\cdot 2^{2^n/n}$ for all $n\ge1$,
which shows that the clone $\cl{C}=\cl{M}$ satisfies
the conclusion of Proposition~\ref{BigClones}.

Another example is the clone $\cl{B}_{k-2}$ 
on a $k$-element set $A$ with $k\ge4$ that
consists
of all essentially at most unary operations and 
all operations whose range has at most $k-2$ 
elements (see Section~\ref{sec:h-reg}). 
We will show in Theorem~\ref{thm-slup} that 
$\cl{B}_{k-2}\notin\FF_A$.
On the other hand,
\[
\card{\cl{B}_{k-2}^{(n)}} \geq (k-2)^{k^n}
= k^{{k^n}\log_k (k-2)} = k^{d k^n},
\]
where $1/2 \leq d = \log_k (k-2) < 1$.
Now if we choose $c = k^{-k/2}$ ($0<c<1$)
then for $n=1$ we have that
\[
k^{d k^1} \geq k^{k/2} = k^{-k/2} k^k = c k^{k^1 \!/ 1},
\]
and for $n \geq 2$ we have that $dn \geq 1$ and so
\[
k^{d k^n} \geq k^{d k^n \!/ dn} = k^{k^n \!/ n} \geq c k^{k^n \!/n}.
\]
Thus, $\card{\cl{B}_{k-2}^{(n)}} \geq c k^{k^n \!/n}$ for all $n \geq 1$,
proving that the conclusion of Proposition~\ref{BigClones} 
holds for $\cl{B}_{k-2}$.
\end{remark}

\begin{remark}
\label{derive2.5}
Theorem~\ref{order_affine} can be derived from Corollaries~\ref{basic_props4},
\ref{PolyClones}, and the fact that 
$\cl{M}\notin\FF_{\{0,1\}}$ holds
for the clone $\cl{M}$ of monotone Boolean functions 
(see Remark~\ref{notconverse}).
Indeed, let first $\Pol{\le}$ be a maximal clone on $A$ determined by 
a bounded partial order $\le$.
We may assume without loss of generality that $A=\{0,1,2,\ldots,k-1\}$
where $k\ge2$ and the least and greatest elements of $\le$ are $0$ and $1$.
Thus $\Pol{\le}|_{\{0,1\}}=\cl{M}\notin\FF_{\{0,1\}}$, 
so Corollary~\ref{basic_props4} implies that $\Pol{\le}\notin\FF_A$.

Next let $\cl{C}$ be a maximal clone on $A$ determined by a
prime affine relation. 
In this case $|A|=q^r$ for some prime $q$ and some positive integer $r$. 
Moreover, there exists an elementary abelian $q$-group $(A;+)$ such that
the $n$-ary operations in $\cl{C}$ are exactly the operations
$\sum_{i=1}^n M_ix_i+a$ where $a\in A$ and each $M_i$ is an $r\times r$ 
matrix over the $q$-element field. 
Thus, using the notation $k:=|A|=q^r$ 
we get that 
$|\cl{C}^{(n)}|\le(q^{r^2})^n q^r=k^{rn+1}$.
Hence Corollary~\ref{PolyClones} implies that $\cl{C}\notin\FF_A$
\end{remark}

We conclude this section by two further applications of 
Propositions~\ref{basic_props3}, \ref{BigClones} and their corollaries.
Recall that \emph{Burle's clone} on a finite set $A$ is the subclone 
of $\cl{O}_A$ that consists
of all essentially at most unary operations and all 
{\it quasilinear operations}, i.e., all 
operations of the form $g \bigl(h_1(x_1) \oplus \dotsb \oplus h_n(x_n)\bigr)$ where $h_1, \dotsc, h_n \colon A \to \{0,1\}$, $g \colon \{0,1\} \to A$ are arbitrary mappings and $\oplus$ denotes addition modulo $2$.
We will denote Burle's clone by $\cl{B}_1$ (see Section~\ref{sec:h-reg}). 

\begin{corollary}
\label{burle}
If $A$ is a finite set with at least two elements, 
then $\cl{B}_1\notin\FF_A$.
\end{corollary}

\begin{proof}
If $\card{A}=2$, then
Burle's clone is the unique maximal clone determined by a prime 
affine relation. 
As discussed in Remark~\ref{derive2.5}, 
in this case $\cl{B}_1\notin\FF_A$ can be proved using 
Corollary~\ref{PolyClones}. 
{}From now on let $\card{A}=k\ge3$, and assume without loss of generality that 
$A=\{0,1,2,\ldots,k-1\}$.
In this case we can employ either one of Corollaries~\ref{basic_props4} and
\ref{PolyClones} to prove that $\cl{B}_1\notin\FF_A$.

First we will discuss the proof that 
relies on Corollary~\ref{basic_props4}.
It is well known that $\cl{B}_1=\Pol\beta$ where $\beta$ is the $4$-ary
relation on $A$ that consists of all tuples of the form 
$(x,x,y,y)$, $(x,y,x,y)$, and $(x,y,y,x)$ with $x,y\in A$.
Since $\beta|_{\{0,1\}}$ is the unique prime affine relation on $\{0,1\}$,
our argument in Remark~\ref{derive2.5} shows that
$\Pol\beta|_{\{0,1\}}\notin\FF_{\{0,1\}}$.
Thus Corollary~\ref{basic_props4} yields that 
$\cl{B}_1=\Pol\beta\notin\FF_A$.

To get the same conclusion using Corollary~\ref{PolyClones}
we have to estimate the number of $n$-ary operations in $\cl{B}_1$.
The number of functions $A \to \{0, 1\}$ is $2^k$, 
and the number of functions $\{0, 1\} \to A$ is $k^2$, 
so the number of $n$-ary quasilinear operations on $A$ is at most 
$k^2 (2^k)^n$. The number of functions $A \to A$ is $k^k$, 
so the number of $n$-ary, essentially at most unary operations on $A$ 
is at most $n k^k$. Thus,
\[
\card{\cl{B}_1^{(n)}} \leq
k^2 (2^k)^n + n k^k \leq
k^k (k^k)^n + (k^k)^n k^k \leq
k k^k (k^k)^n = k^{kn + k + 1},
\]
where the second inequality holds because $k > 2$ and hence 
$n \leq (k^k)^n$ for all $n \geq 1$. 
It follows from Corollary~\ref{PolyClones} that $\cl{B}_1\notin\FF_A$.
\end{proof}

In the proof of Theorem~\ref{mainthm-eqrels} we will see an application of
Corollary~\ref{SmallClones} where the function $p$ is not a polynomial.

Our last application answers a question on minimal clones raised by P.~Mayr.
Recall that a clone $\cl{C}$ on $A$ is called \emph{minimal} if 
$\cl{C}$ is not the clone $\cl{P}_A$ of projections, and 
$\cl{P}_A$ is the only proper subclone of $\cl{C}$.
Equivalently, $\cl{C}$ is a minimal clone on $A$ if and only if
$\cl{C}\setminus\cl{P}_A\not=\emptyset$ and $\langle f\rangle=\cl{C}$
for all $f\in\cl{C}\setminus\cl{P}_A$.

\begin{corollary}
\label{min-clones}
If $A$ is a finite set with at least two elements, then no minimal clone 
on $A$ belongs to $\FF_A$.
\end{corollary}

\begin{proof}
Assume that the statement is false, and let $A$ be a finite 
set of minimum size $\card{A}\ge2$ such that $\FF_A$ contains a minimal
clone $\cl{C}$. 
Let $B$ be any $2$-element subset of $A$.
Since $\cl{C}$ is a minimal clone, the clone $\cl C\cap\Pol{B}$ 
is either $\cl{P}_A$ or $\cl{C}$. 
Hence the clone $\cl C_B=\{f|_B:f\in\cl{C}\cap\Pol{B}\}$ 
defined in Proposition~\ref{basic_props3} 
is either $\cl{P}_B$ or a minimal clone on $B$.
By Proposition~\ref{basic_props3}, 
the assumption $\cl{C}\in\FF_A$ implies that
$\cl{C}_B\in\FF_B$.
However, as we discussed in the introduction, 
$\cl{P}_B\notin\FF_B$. 
Therefore $\cl{C}_B$ is a minimal clone on $B$
that is a member of $\FF_B$. 
The minimality of $A$ implies that $B=A$ and hence $\card{A}=2$.
It is well known from~\cite{Post} 
that there are seven minimal clones on a $2$-element set,
and each one of them is either a subclone of the maximal clone $\cl{M}$ of
all monotone Boolean functions, or a subclone of
the maximal clone $\cl{B}_1$ of all linear Boolean functions.
Therefore Theorem~\ref{discr} or Theorem~\ref{order_affine} (see also
Remark~\ref{derive2.5}) implies that $\cl{C}\notin\FF_A$. 
This contradicts our assumption on $\cl{C}$, and hence proves 
Corollary~\ref{min-clones}.
\end{proof}

\section{Equivalence relations}
\label{sec:eqrels}

Let $E$ be a set of equivalence relations on a finite set $A$. 
Our aim in this section is to show that $\Pol E \in \FF_A$ 
if and only if $E$ is a chain (with respect to inclusion).
We will in fact prove the following stronger theorem.

\begin{theorem}
\label{mainthm-eqrels}
Let $A$ be a finite set, and let
$E$ be a set of equivalence relations on $A$,
$\Gamma$ a set of permutations of $A$, and 
$\Sigma$ a set of nonempty subsets of $A$.
The clone $\Pol(E,\Gamma,\Sigma)$ is a member of $\FF_A$
if and only if
\begin{enumerate}[\indent\rm(a)]
\item
$E$ is a chain (i.e., any two members of $E$ are comparable), and
\item
$\Gamma\subseteq\Pol{E}$.
\end{enumerate}
\end{theorem}

For any set $E$ of equivalence relations on $A$
we call a permutation $\gamma$ of $A$ \emph{$E$-invariant}
if $\gamma\in\Pol{E}$, 
that is, if $\gamma$ is an
automorphism of the relational structure $(A;E)$.
Therefore we denote the group of
$E$-invariant permutations of $A$ by $\Aut{E}$.
Furthermore, we denote the 
set of all nonempty subsets of $A$ by ${\mathcal P}^+(A)$.
Thus, in Theorem~\ref{mainthm-eqrels}, $\Sigma$ is an arbitrary
subset of ${\mathcal P}^+(A)$ and (b) requires that 
$\Gamma\subseteq\Aut{E}$.

\begin{proof}[Proof of Theorem~\ref{mainthm-eqrels}]
{\it Necessity.}
Let $\cl C=\Pol(E,\Gamma,\Sigma)$ and $k=|A|$.
We want to show that if (a) or (b) fails, then $\cl C\notin\FF_A$.
Assume first that (a) fails, that is, $E$ contains
equivalence relations
$\alpha$ and $\beta$ such that $\alpha \not\subseteq \beta$ and $\beta \not\subseteq \alpha$. Clearly, $\cl C\subseteq \Pol(\alpha, \beta)$, therefore
in view of Proposition~\ref{basic_props2}~(ii) it suffices to prove 
that the clone $\cl{E} = \Pol(\alpha, \beta)$ fails to belong to $\FF_A$.
Let $\al{A}$ denote the algebra $(A; \cl{E})$.
Since $\cl{E} = \Pol(\alpha, \beta)$, it follows that $\alpha$ and
$\beta$ are congruences of $\al{A}$, and the clones of the
corresponding quotient algebras are $\Clo(\al{A}/\alpha) =
\cl{E}^\alpha$ and $\Clo(\al{A}/\beta) = \cl{E}^\beta$,
the natural actions of $\cl{E}$ on $A/\alpha$ and $A/\beta$.

First we will consider the case when $\alpha \wedge \beta = {\zero}_A$. Then the embedding $\al{A} \to \al{A}/\alpha \times \al{A}/\beta$, $a \mapsto (a/\alpha, a/\beta)$ represents $\al{A}$ as a subdirect product of $\al{A}/\alpha$ and $\al{A}/\beta$. Hence $\cl{E} \to \cl{E}^\alpha \times \cl{E}^\beta$, $h \mapsto (h^\alpha, h^\beta)$ is a clone embedding. This implies that for each $n$,
\[
\card{\cl{E}^{(n)}} \leq \card{{(\cl{E}^\alpha)}^{(n)}} \cdot \card{{(\cl{E}^\beta)}^{(n)}}.
\]
The assumption that $\alpha$ and $\beta$ are incomparable ensures that $\card{\al{A}/\alpha} \leq k - 1$ and $\card{\al{A}/\beta} \leq k - 1$. Thus 
\[
\card{\cl{E}^{(n)}}
\leq (k - 1)^{(k - 1)^n} \cdot (k - 1)^{(k - 1)^n}
= (k - 1)^{2(k - 1)^n}
< k^{2(k - 1)^n}.
\]
Since $\displaystyle \lim_{n \to \infty} \frac{2(k - 1)^n}{k^n/n} = 0$, Corollary~\ref{SmallClones} implies that $\cl{E}\notin\FF_A$.

To prove the statement in the general case let 
$\theta = \alpha \wedge \beta$, 
and consider the algebra $\al{A}/\theta$ and its congruences $\alpha/\theta$ and $\beta/\theta$. Clearly, the clone of $\al{A}/\theta$ is $\Clo(\al{A}/\theta) = \cl{E}^\theta$, and the assumptions ensure that $\alpha/\theta \not\subseteq \beta/\theta$ and $\beta/\theta \not\subseteq \alpha/\theta$. Since $\alpha/\theta \wedge \beta/\theta = {\zero}_{A/\theta}$, the special case established in the preceding paragraph shows that $\cl{E}^\theta\notin\FF_{A/\theta}$.
Hence
there exists an infinite sequence of operations $g_n$ ($n \geq 1$) on $A/\theta$ such that $g_i \not\fequiv[\cl{E}^\theta] g_j$ for all $i \neq j$. Now choose and fix operations $f_n$ ($n \geq 1$) on $A$ such that $g_n = f_n^\theta$ for each $n$. Then $f_n \in \Pol \theta$ ($n \geq 1$) and $f_i^\theta \not\fequiv[\cl{E}^\theta] f_j^\theta$ whenever $i \neq j$. We claim that $f_i \not\fequiv[\cl{E}] f_j$ whenever $i \neq j$. Suppose otherwise, and let $i \neq j$ be such that $f_i \fequiv[\cl{E}] f_j$. Then there exist tuples of operations $\vect{h}$ and $\vect{h}'$ in $\cl{E}$ such that $f_i = f_j \circ \vect{h}$ and $f_j = f_i \circ \vect{h}'$. Since all operations in $\vect{h}$ and $\vect{h}'$ belong to $\cl{E} = \Pol(\alpha, \beta)$, they preserve $\theta = \alpha \wedge \beta$. Hence we get that $f_i^\theta = f_j^\theta \circ \vect{h}^\theta$ and $f_j^\theta = f_i^\theta \circ (\vect{h}')^\theta$, which contradicts the choice of the operations $g_n = f_n^\theta$. Thus there are infinitely many 
$\fequiv[\cl{E}]$-classes, and hence $\cl{E}\notin\FF_A$.
This proves the necessity of condition (a).

Now assume that condition (b) fails, and let 
$\gamma\in\Gamma$ be such that $\gamma\notin\Pol E$, that is,
$\gamma\notin\Pol\rho$ for some $\rho\in E$.
Let $\gamma(\rho)=\{\bigl(\gamma(a),\gamma(b)\bigr):(a,b)\in\rho\}$, and
let $E'=E\cup\{\gamma(\rho)\}$.
Clearly, $\gamma(\rho)$ is an equivalence relation on $A$,
and $\gamma(\rho)\not=\rho$, since $\gamma\notin\Pol\rho$.
As $A$ is finite, and $\rho$ and $\gamma(\rho)$ 
have the same system of block sizes, it follows that 
$\rho$ and $\gamma(\rho)$ are incomparable. 
Hence
$E'$ is a set of equivalence relations that is not a chain.
It is easy to verify that
every operation that preserves both $\gamma$ and $\rho$
also preserves $\gamma(\rho)$. 
Therefore $\cl C\subseteq\Pol E'$, and the 
failure of condition (a) shows that
$\Pol E'\notin\FF_A$. Thus 
Proposition~\ref{basic_props2}~(ii) implies that $\cl C\notin\FF_A$,
establishing the necessity of condition (b).

{\it Sufficiency.}
Given a chain $E$ of equivalence relations, there is a smallest
clone of the form $\Pol(E,\Gamma,\Sigma)$ satisfying the assumptions
of the theorem and also condition (b), namely the clone
$\Pol\bigl(E,\Aut{E},\cl P^+(A)\bigr)$.
Therefore, by Proposition~\ref{basic_props2}~(ii), it suffices to prove
that this clone belongs to $\FF_A$.
This claim, which is the hardest part of Theorem~\ref{mainthm-eqrels},
is stated below as Theorem~\ref{chains-of-eqrels}, and will be
proved separately.
\end{proof}

\begin{theorem}
\label{chains-of-eqrels}
If $E$ is a chain of equivalence relations on a finite set $A$, 
then $\Pol\bigl(E,\Aut{E},{\mathcal P}^+(A)\bigr)\in\FF_A$.
\end{theorem}

\begin{remark}
\label{2/3-minority}
For every chain $E$ of equivalence relations on $A$,
the clone $\Pol\bigl(E,\linebreak[1]\Aut{E},{\mathcal P}^+(A)\bigr)$
contains a $2/3$-minority operation, i.e., a ternary operation $m$
such that
\begin{equation}
\label{eq:2/3-minority}
m(x,x,y)=y,\quad m(x,y,x)=x,\quad \text{and}\quad m(x,y,y)=x
\end{equation}
for all $x,y\in A$.
To define such an operation let $\theta(a,b)$ denote the least
equivalence relation $\epsilon\in E\cup\{{\zero}_A,{\one}_A\}$ such that
$a\,\epsilon\,b$ ($a,b\in A$).
It is clear that if $a,b,c\in A$ and 
$\theta(a,b)\le\theta(a,c),\theta(b,c)$, then
$\theta(a,c)=\theta(b,c)$.
We will write $a\sim b\sim c$ to denote that 
$\theta(a,b)=\theta(a,c)=\theta(b,c)$, and
$a\sim b\not\sim c$ to denote that
$\theta(a,b)<\theta(a,c)=\theta(b,c)$.
Since $E\cup\{{\zero}_A,{\one}_A\}$ is a chain, 
it follows that exactly one of the following
conditions holds for any triple $(a,b,c)\in A^3$:

\begin{center}
(i)~$a\sim b\sim c$,\quad 
(ii)~$a\sim b\not\sim c$,\quad
(iii)~$a\sim c\not\sim b$,\quad
(iv)~$b\sim c\not\sim a$.
\end{center}

We define a ternary operation $m$ on $A$ as follows:
\[
m(x,y,z)=
\begin{cases}
z & \text{if $x\sim y\sim z$ or $x\sim y\not\sim z$,}\\
x & \text{if $x\sim z\not\sim y$ or $y\sim z\not\sim x$}
\end{cases}
\qquad
(x,y,z\in A).
\]
For any $x,y\in A$ we have 
$x\sim x\sim y$ if $x=y$ and $x\sim x\not\sim y$
if $x\not=y$. 
Hence, in either case, 
the definition of $m$ shows that
the equalities in \eqref{eq:2/3-minority}
hold, which proves that
$m$ is a 2/3-minority operation.
Since on any input triple the value of $m$ equals one of the
inputs, it follows that 
$m$ preserves all nonempty subsets of $A$.
If $\gamma\in\Aut{E}$, then 
$\theta(a,b)=\theta\bigl(\gamma(a),\gamma(b)\bigr)$
holds for all $a,b\in A$. 
Consequently, for each one of conditions (i)--(iv),
a triple $(a,b,c)\in A^3$
satisfies this condition if and only if 
the triple $\bigl(\gamma(a),\gamma(b),\gamma(c)\bigr)$ does.
This implies that $m$ preserves all permutations $\gamma\in\Aut{E}$.

Finally, to see that $m$ preserves all equivalence relations in $E$ let
$\rho\in E$, and let $(a,b,c)\,\rho^3\,(a',b',c')$.
As we will now show, the
latter assumption implies that 
\begin{align}
\theta(a,b)\vee\rho & {}=\theta(a',b')\vee\rho,\notag\\
\theta(a,c)\vee\rho & {}=\theta(a',c')\vee\rho,\label{xyz}\\
\theta(b,c)\vee\rho & {}=\theta(b',c')\vee\rho.\notag
\end{align}
Indeed, by our assumption 
we have that $a\,\rho\,a'$ and $b\,\rho\,b'$, therefore
$$
(a',b')\in\rho\circ\theta(a,b)\circ\rho\subseteq\theta(a,b)\vee\rho.
$$
Here $\theta(a,b)\vee\rho$ is the larger one of $\theta(a,b)$ and
$\rho$ in the chain $E$, so $(a',b')\in\theta(a,b)\vee\rho$
implies that
the least equivalence relation
$\theta(a',b')$ in $E$ containing the pair $(a',b')$ satisfies 
$\theta(a',b')\le\theta(a,b)\vee\rho$.
Hence $\theta(a',b')\vee\rho\le\theta(a,b)\vee\rho$.
By interchanging the roles of $a,b$ and $a',b'$ we get
the reverse inclusion $\theta(a,b)\vee\rho\le\theta(a',b')\vee\rho$,
which proves the first equality in (\ref{xyz}).
The second and third equalities can be proved similarly.

Our goal is to verify that the assumption
$(a,b,c)\,\rho^3\,(a',b',c')$ implies that\linebreak[4] $m(a,b,c)\,\rho\,m(a',b',c')$.
If $a\,\rho\,b\,\rho\,c$ or ${a'}\,\rho\,\,{b'}\,\rho\,\,{c'}$, then
by the assumption\linebreak[4]
$(a,b,c)\,\rho^3\,(a',b',c')$
all six elements $a,b,c,a',b',c'$ lie in the same $\rho$-class, so
$m(a,b,c)$ and $m(a',b',c')$, too, lie in that $\rho$-class,
because $m(a,b,c)\in\{a,c\}$ and $m(a',b',c')\in\{a',c'\}$.
Thus $m(a,b,c)\,\rho\,m(a',b',c')$ holds in this case.

Now assume for the rest of the proof that 
\begin{equation}
\label{eq:assumption}
\text{$a,b,c$ are not
all $\rho$-related, and $a',b',c'$ are not all $\rho$-related.}
\end{equation}
We want to prove that 
\begin{enumerate}
\item[$(*)$]
for each one of conditions (i)--(iv),
$(a,b,c)$ satisfies this condition if and only if 
$(a',b',c')$ does.
\end{enumerate}
By the definition of $m$, this will imply that 
$\bigl(m(a,b,c),m(a',b',c')\bigr)=(a,a')$ or $(c,c')$,
hence $m(a,b,c)\,\rho\,m(a',b',c')$.
Since statement $(*)$ is invariant under 
performing the same permutation on the coordinates of the two triples,
and since the roles of the two triples are symmetric, $(*)$ will follow
if we show that $a\sim b\sim c$ implies $a'\sim b'\sim c'$,
and  $a\sim b\not\sim c$ implies $a'\sim b'\not\sim c'$.
So, let us assume first that $a\sim b\sim c$, that is,
$\theta(a,b)=\theta(a,c)=\theta(b,c)$.
Since $E\cup\{{\zero}_A,{\one}_A\}$ is a chain, 
our assumption \eqref{eq:assumption}
forces that
$\theta(a,b)=\theta(a,c)=\theta(b,c)>\rho$.
Therefore \eqref{xyz} implies that 
$$
\rho<\theta(a,b)=\theta(a,c)=\theta(b,c)=
\theta(a',b')\vee\rho=\theta(a',c')\vee\rho=\theta(b',c')\vee\rho.
$$
The inequality $\rho<\theta(a',b')\vee\rho$ shows that $\rho<\theta(a',b')$.   
Similarly, $\rho<\theta(a',c')$ and $\rho<\theta(b',c')$.
Now the displayed equalities imply that    
$\theta(a',b')=\theta(a',c')=\theta(b',c')\bigl(=\theta(a,b)\bigr)$,
and hence $a'\sim b'\sim c'$.
Next let us assume that $a\sim b\not\sim c$.
Thus, $\theta(a,c)=\theta(b,c)>\theta(a,b)$, and
since $E\cup\{{\zero}_A,{\one}_A\}$ is a chain, we get 
from our assumption \eqref{eq:assumption} that
$\theta(a,c)=\theta(b,c)>\rho$.
This inequality, combined with the second and third equalities in
\eqref{xyz} yields, as before, that
$$
\theta(a',c')=\theta(b',c')=\theta(a,c)=\theta(b,c)>\rho.
$$
The same holds with $\rho$ replaced by $\theta(a,b)$, since
$\theta(a,c)>\theta(a,b)$.
Therefore $\rho$ can also be replaced by $\alpha:=\theta(a,b)\vee\rho$, 
the larger one of $\theta(a,b)$ and $\rho$.
Hence
$$
\theta(a',c')=\theta(b',c')>\alpha.
$$
Making use of \eqref{xyz} again 
we also get that $\alpha\ge\theta(a',b')$, because
$$
\alpha=\theta(a,b)\vee\rho=\theta(a',b')\vee\rho\ge\theta(a',b').
$$
Thus $a'\sim b'\not\sim c'$, which completes the proof
of $(*)$, and thereby establishes the existence of a $2/3$-minority
operation in 
the clone $\Pol\bigl(E,\Aut{E},{\mathcal P}^+(A)\bigr)$
for every chain $E$ of equivalence relations on $A$.
\end{remark}

\begin{remark}
If $E=\emptyset$ (or $E\subseteq\{{\zero}_A,{\one}_A\}$), 
then $\Aut{E}$ is the full symmetric group on $A$, 
the $2/3$-minority operation $m$ defined in Remark~\ref{2/3-minority}
is the ternary descriminator $t_A$ on $A$, and
$\Pol\bigl(E,\Aut{E},{\mathcal P}^+(A)\bigr)$ is the clone
generated by $t_A$.
Therefore Theorem~\ref{chains-of-eqrels} includes 
the statement $\langle t_A\rangle\in\FF_A$ from
Theorem~\ref{discr}~(i) as a special case.
\end{remark}

Let $E$ be a chain of equivalence relations on $A$, 
let $\Gamma=\Aut{E}$, and
let $\cl C=\Pol\bigl(E,\Aut{E}, \linebreak[0] {\mathcal P}^+(A)\bigr)$.
We will prove Theorem~\ref{chains-of-eqrels}
by associating to each operation on $A$ a 
finite structure 
of bounded size in such a way that if two operations have isomorphic 
structures associated to them, 
then they are in the same $\fequiv[\cl C]$-class.
This finite structure, to be defined in detail below,
 will be a $\Gamma$-set with 
a tree structure on it, and the leaves of the tree will have a labeling
that is compatible with the action of $\Gamma$.

Let $G$ be an arbitrary group.
A \emph{$G$-set} is a unary algebra $(U;G)$
such that each $g\in G$ acts on $U$ by a permutation
$U\to U$, $u\mapsto g\cdot u$, and for any $g,g'\in G$ and
$u\in U$, we have $gg'\cdot u=g\cdot (g'\cdot u)$.
Since each $g\in G$ acts by a permutation of $U$, it follows that
the neutral element $\groupone$ of $G$ acts by the identity
permutation, that is, $\groupone\cdot u=u$ holds for all $u\in U$.
Consequently, for any $g\in G$, 
the actions of $g$ and $g^{-1}$ are inverses of each other.
If there is no danger of confusion, we will write $gu$ instead of
$g\cdot u$.
For any element $u\in U$, \emph{the stabilizer of $u$ in $G$}
is the subgroup $G_u:=\{g\in G:gu=u\}$ of $G$.
For $u\in U$
the subalgebra $Gu:=\{gu:g\in G\}$ of $(U;G)$ generated by
$u\in U$ is called the \emph{$G$-orbit} of $u$.
It is well known and easy to check that the $G$-orbits
of $(U;G)$ are minimal subalgebras, and therefore they
partition $U$.
If $(U;G)$ and $(V;G)$ are $G$-sets, then a mapping 
$\phi\colon U\to V$ is a \emph{homomorphism} 
$(U;G)\to(V;G)$ of $G$-sets, if $\phi(gu)=g\cdot\phi(u)$
holds for all $u\in U$ and $g\in G$.
By a \emph{pointed $G$-set} $(U;u,G)$ we mean a $G$-set $(U;G)$ with a 
distinguished element $u\in U$. 
If $U=Gu$ is a $G$-orbit, we will call the pointed $G$-set
$(U;u,G)$ as well as the pointed set $(U;u)$ (if the $G$-set structure 
is irrelevant) a \emph{pointed $G$-orbit}.
A homomorphism $(U;u,G)\to(V;v,G)$ between pointed $G$-sets 
is a homomorphism $\phi\colon (U;G)\to(V;G)$ between the underlying $G$-sets
such that $\phi(u)=v$.
If $(U;u,G)$ and $(V;v,G)$ are pointed $G$-orbits, 
that is, $U=Gu$ and $V=Gv$, then
a homomorphism $\phi\colon(U;u,G)\to(V;v,G)$ 
exists between them if and only if $G_u\subseteq G_v$; moreover, $\phi$
is uniquely determined: $\phi\colon U=Gu\to Gv=V$,
$gu\mapsto gv$ for all $g\in G$. 
We will denote this homomorphism (if it exists) by $\chi_{u,v}$.
Clearly, $\chi_{u,v}=\chi_{gu,gv}$ for all $g\in G$, and
$\chi_{u,v}$ is an isomorphism if and only if $G_u=G_v$.

By a \emph{tree} we mean a finite partial algebra $\al P=(P;{}^*,1_P)$ 
where ${}^*\colon P\setminus\{1_P\}\to P$ is a function, called
the \emph{successor function},
such that the distinguished element $1_P$ can be obtained from
any other element $a\in P\setminus\{1_P\}$ by repeated application of 
${}^*$.
Denoting the $i$-th power of ${}^*$ by ${}^{*^i}$ we get that
for each $a\in P$ there is a unique integer
$d\ge0$ such that $a^{*^d}=1_P$, which will be called the \emph{depth} of $a$.
The only element of depth $0$ is $1_P$.
An element $a$ of $\al P$ will be called a \emph{leaf} if
it is not in the range of the successor function.
We will denote the set of leaves of $\al P$ by $\al P_\leaf$. 
If every leaf of $\al P$ has the same depth $d$, 
we will say that \emph{the tree $\al P$ has uniform depth $d$}.

If $\al P=(P;{}^*,1_P)$ and $\al Q=(Q;{}^*,1_Q)$ are trees, 
we will call a function
$\phi\colon P\to Q$ a \emph{homomorphism} 
$\al P\to\al Q$ of trees if
\begin{enumerate}
\item[(H0)]
$\phi(1_P)=1_Q$,
\item[(H1)]
$\phi$ maps leaves to leaves, that is, 
$\phi(\al P_\leaf)\subseteq \al Q_\leaf$, and
\item[(H2)]
$\phi(a^*)=\phi(a)^*$ for all $a\in P\setminus\{1_P\}$. 
\end{enumerate}
An \emph{automorphism} of $\al P$
is a bijective homomorphism 
$\al P\to\al P$.

A tree $\al Q=(Q;{}^*,1_Q)$ is a \emph{subtree} of another tree
$\al P=(P;{}^*,1_P)$ if $Q\subseteq P$ and the identity function
$Q\to P$, $q\mapsto q$ is a homomorphism 
$\al Q\to\al P$.
Thus $\al Q$ is a subtree of 
$\al P$ if and only if $Q\subseteq P$, $1_Q=1_P$,
$\al Q_\leaf\subseteq \al P_\leaf$, and the successor function
of $\al Q$ is the restriction to $Q\setminus\{1_Q\}$
of the successor function of $\al P$.

Let $G$ be a group. We define a
$G$-tree to be a tree 
on which $G$ acts by automorphisms; more precisely,
a \emph{$G$-tree} is a structure 
$\al P=(P;{}^*,1_P,G)$ such that
$(P;{}^*,1_P)$ is a tree, 
$(P;G)$ is a $G$-set, and for each
$g\in G$ the permutation $a\mapsto ga$ of $P$
is an automorphism of the tree $(P;{}^*,1_P)$.
The assumption that $G$ acts by tree automorphisms
implies that in every $G$-tree $\al P=(P;{}^*,1_P,G)$,
\[
g\cdot1_P=1_P
\quad \text{for all $g\in G$,}
\]
and
\[
a^{*^d}=1_P
\,\,\,\iff\,\,\,
(ga)^{*^d}\,\bigl({}=g a^{*^d}\bigr){}=1_P
\quad
\text{for all $a\in P\setminus\{1_P\}$ and $g\in G$.}
\]
Therefore each $G$-orbit
$Ga$ of $\al P$
consists of elements of the same depth.
Similarly, if $a$ is a leaf, then so are all elements in 
the $G$-orbit $Ga$ of $a$.
Thus the leaves of $\al P$ form a $G$-set $(\al P_\leaf;G)$.

For arbitrary $G$-trees
$\al P=(P;{}^*,1_P,G)$ and $\al Q=(Q;{}^*,1_Q,G)$ 
a \emph{$G$-ho\-mo\-mor\-phism} $\al P\to\al Q$
is a mapping $\phi\colon P\to Q$ that is a 
homomorphism $(P;{}^*,1_P)\to(Q;{}^*,1_Q)$
of trees and also a homomorphism $(P;G)\to(Q;G)$ of $G$-sets;
that is, in addition to (H0), (H1), and (H2), $\phi$ also satisfies
\begin{enumerate}
\item[(H3)]
$\phi(g a)=g\cdot\phi(a)$ for all $a\in P$ and $g\in G$.
\end{enumerate}
A $G$-tree $\al Q=(Q;{}^*,1_Q,G)$ is a \emph{$G$-subtree} of
$\al P=(P;{}^*,1_P,G)$ if 
$Q\subseteq P$ and the identity function
$Q\to P$, $q\mapsto q$ is a $G$-homomorphism 
$\al Q\to\al P$. Thus
$\al Q$ is a $G$-subtree of
$\al P$ if and only if
$(Q;{}^*,1_Q)$ is a subtree of $(P;{}^*,1_P)$ and 
the action of each $g\in G$ on $Q$ is the restriction to $Q$
of the action of $g$ on $P$.
Hence, if $\al P=(P;{}^*,1_P,G)$ is a $G$-tree,
then a subtree $(Q;{}^*,1_P)$ of $(P;{}^*,1_P)$   
is (the underlying tree of) a $G$-subtree of
$\al P$ if and only if $Q$ is a union of $G$-orbits of 
$\al P$.

Next we will introduce the concept of a labeled $G$-tree.
The labels will come from a structure $(S;\le,G)$ 
where $(S;\le)$ is a partially ordered set on which $G$ acts
by automorphisms; more precisely, 
$(S;\le,G)$ is a structure such that
$(S;\le)$ is a partially ordered set,
$(S;G)$ is a $G$-set, and for each
$g\in G$, the permutation $s\mapsto gs$ of $S$
is an automorphism of $(S;\le)$.
If $\al P=(P;{}^*,1_P,G)$ is a $G$-tree, then an
\emph{$S$-labeling of the leaves of $\al P$} is a 
homomorphism $\ell\colon (\al P_\leaf;G)\to (S;G)$ 
of $G$-sets.
An \emph{$S$-labeled $G$-tree} is a structure
$(\al P;\ell)=(P;{}^*,1_P,G;\ell)$ where
$\al P=(P;{}^*,1_P,G)$ is a $G$-tree and $\ell$ is an $S$-labeling of the
leaves of $\al P$.
If the labeling $\ell$ is understood, we will write $\al P$
instead of $(\al P;\ell)$.

For arbitrary $S$-labeled $G$-trees $\al P=(P;{}^*,1_P,G;\ell_P)$ 
and $\al Q=(Q;{}^*,1_Q,G; \linebreak[0] \ell_Q)$ 
a \emph{label-preserving $G$-homomorphism} $\al P\to\al Q$
is a $G$-homomorphism $\phi\colon(P;\linebreak[0]{}^*,1_P,G) \to (Q;{}^*,1_Q,G)$
with the additional property that
\begin{enumerate}
\item[(H4)]
$\ell_P(a)=\ell_Q\bigl(\phi(a)\bigr)$ for all $a\in \al P_\leaf$,
\end{enumerate}
and 
a \emph{label-increasing $G$-homomorphism} $\al P\nearrow\al Q$
is a $G$-homomorphism $\phi\colon(P;\linebreak[0]{}^*,1_P,G)\to(Q;{}^*,1_Q,G)$
with the additional property that
\begin{enumerate}
\item[(H5)]
$\ell_P(a)\le\ell_Q\bigl(\phi(a)\bigr)$ for all $a\in \al P_\leaf$.
\end{enumerate}
Clearly, every label-preserving $G$-homomorphism is a
label-increasing $G$-ho\-mo\-mor\-phism.
Moreover, the composition of label-preserving
$G$-homomorphisms is a la\-bel-pre\-serv\-ing
$G$-homomorphism, and the same holds for label-increasing
$G$-ho\-mo\-mor\-phisms.
An {\it isomorphism between $S$-labeled $G$-trees}
is a bijective, la\-bel-pre\-serv\-ing $G$-ho\-mo\-mor\-phism. 
As usual, if there exists an isomorphism $\al P\to\al Q$ 
between two $S$-labeled $G$-trees $\al P$ and $\al Q$, then
$\al P$ and $\al Q$ are said to be {\it isomorphic}; is symbols:
$\al P\cong\al Q$.

An $S$-labeled $G$-tree $\al Q=(Q;{}^*,1_Q,G;\ell_Q)$ 
is an \emph{$S$-labeled $G$-subtree} of
$\al P=(P;{}^*,1_P,G;\ell_P)$ if 
$Q\subseteq P$ and the identity function
$Q\to P$, $q\mapsto q$ is a label-preserving $G$-homomorphism 
$\al Q\to\al P$; or equivalently,
if 
$(Q;{}^*,1_Q,G)$ is a $G$-subtree of $(P;{}^*,1_P,G)$ and 
$\ell_Q$ is the restriction of $\ell_P$ to $\al Q_\leaf$.

The main examples of labeled trees we will be concerned with
are obtained from chains $E$ of equivalence relations as follows.
Let $E=\{\rho_i:1\le i\le r\}$, say, 
$\rho_0:={\zero}_A<\rho_1<\cdots<\rho_{r-1}<\rho_{r}<{\one}_A=:\rho_{r+1}$, 
and let $\Gamma:=\Aut{E}$.
Since $\Gamma$ is a group of permutations on $A$, 
$(A;\Gamma)$ becomes a $\Gamma$-set with the natural action
defined by $\gamma a=\gamma(a)$ for all $a\in A$ and
$\gamma\in\Gamma$.
For each integer $n\ge1$, the $n$-th power of $(A;\Gamma)$
is the $\Gamma$-set $(A^n;\Gamma)$
where $\Gamma$ acts coordinatewise on $n$-tuples
in $A^n$; that is, 
$\gamma\vect{a}=\bigl(\gamma(a_1),\ldots,\gamma(a_n)\bigr)$
for all $\vect{a}=(a_1,\ldots,a_n)\in A^n$.
Since each permutation $\gamma\in\Gamma$ is $\rho_i$-invariant for all
$i$ ($0\le i\le r+1$), these equivalence relations are congruences
of $(A;\Gamma)$, and for each $n\ge1$,
the equivalence relations $(\rho_i)^n$ are congruences of
$(A^n;\Gamma)$.
Hence we get quotient $\Gamma$-sets 
$(A^n;\Gamma)/(\rho_i)^n=(A^n/(\rho_i)^n;\Gamma)$
whose elements are the blocks of $(\rho_i)^n$, and $\Gamma$
acts on them the natural way: if $B$ is a block of $(\rho_i)^n$
and $\gamma\in\Gamma$,
then 
$\gamma B$ is the block $\{\gamma\vect{x}:\vect{x}\in B\}$
of $(\rho_i)^n$. 
Thus the
$\Gamma$-orbit of any block $B$ of $(\rho_i)^n$ 
is the set
$\Gamma B=\{\gamma B:\gamma\in\Gamma\}$.
For $i=0$ we will identify $A^n/(\rho_0)^n=A^n/{\zero}_{A^n}$ with
$A^n$, and accordingly, if $B=\{\vect{x}\}$, then we will write
$\Gamma\vect{x}$ for $\Gamma\{\vect{x}\}$.

For each integer $n\ge1$ we define a $\Gamma$-tree 
$(P_n(E);{}^*,1_{P_n(E)},\Gamma)$ of uniform depth $r+1$ 
associated to $E$ as follows: 
\begin{itemize}
\item
$P_n(E):=\{(i,B): 0\le i\le r+1,\ 
\text{$B$ is a block of $(\rho_i)^n$ on $A^n$}\}$,
\item
$1_{P_n(E)}:=(r+1,A^n)$,
\item
the successor of each element $(i,B)$ ($0\le i\le r$)
is defined by
$(i,B)^*:=(i+1,C)$ where $C$ is the unique block of $\rho_{i+1}$ with
$B\subseteq C$, and
\item
$\gamma\cdot(i,B):=(i,\gamma B)$ for all $(i,B)\in P_n(E)$.
\end{itemize}
It is clear that $(P_n(E);{}^*,1_{P_n(E)},\Gamma)$ 
is indeed a $\Gamma$-tree of uniform depth $r+1$.
 
\begin{example}
\label{ex-tree}
Figure~1 depicts the $\Gamma$-tree
$(P_n(E);{}^*,1_{P_n(E)},\Gamma)$ for the case when $n=1$, 
$A=\{1,2,3,4\}$, $E=\{\rho_1,\rho_2\}$, and $\rho_1$ has blocks 
$\{1\}$, $\{2\}$, $\{3,4\}$, while $\rho_2$ has blocks 
$\{1,2\}$, $\{3,4\}$. 
\begin{figure}
\setlength{\unitlength}{1mm}
\begin{picture}(110,60)
\thicklines
\put(55,51){\oval(18,10)}
\put(49,52.5){$1_{P_1(E)}=$}
\put(50.7,48){$(3,A)$}
\put(25,36){\oval(18,8)}
\put(17.8,35){$(2,\{1,2\})$}
\put(85,36){\oval(18,8)}
\put(77.8,35){$(2,\{3,4\})$}
\put(10,20){\oval(18,8)}
\put(4.4,19){$(1,\{1\})$}
\put(40,20){\oval(18,8)}
\put(34.4,19){$(1,\{2\})$}
\put(85,20){\oval(18,8)}
\put(77.8,19){$(1,\{3,4\})$}
\put(10,4){\oval(18,8)}
\put(6,3){$(0,1)$}
\put(40,4){\oval(18,8)}
\put(36,3){$(0,2)$}
\put(70,4){\oval(18,8)}
\put(66,3){$(0,3)$}
\put(100,4){\oval(18,8)}
\put(96,3){$(0,4)$}
\put(25,40){\line(4,1){24.3}}
\put(85,40){\line(-4,1){24.3}}
\put(10,24){\line(3,2){11.9}}
\put(40,24){\line(-3,2){11.9}}
\put(85,24){\line(0,1){8}}
\put(10,8){\line(0,1){8}}
\put(40,8){\line(0,1){8}}
\put(70,8){\line(3,2){11.9}}
\put(100,8){\line(-3,2){11.9}}
\put(19,3){$\leftarrow$}
\put(27.5,3){$\rightarrow$}
\put(22.5,4){\tiny\text{$(1\,2)$}}
\put(19,19){$\leftarrow$}
\put(27.5,19){$\rightarrow$}
\put(22.5,20){\tiny\text{$(1\,2)$}}
\put(79,3){$\leftarrow$}
\put(87.5,3){$\rightarrow$}
\put(82.5,2.3){\tiny\text{$(3\,4)$}}
\end{picture}
\caption{}
\end{figure}
It is easy to see that
$\Gamma=\Aut{E}$ is the $4$-element group generated by the
transpositions $(1\ 2)$ and $(3\ 4)$.
The transposition $(1\ 2)$ acts by switching
$(0,1)$ with $(0,2)$, $(1,\{1\})$ with $(1,\{2\})$, and fixing all
other vertices of the tree, while the transposition
$(3\ 4)$ acts by switching
$(0,3)$ with $(0,4)$ and fixing all other vertices.
\end{example}

We return to the discussion of the
$\Gamma$-trees $(P_n(E);{}^*,1_{P_n(E)},\Gamma)$ 
introduced before the example, where $E$ is an arbitrary 
chain of equivalence relations on a finite set $A$, 
$\Gamma=\Aut{E}$, and $n\ge 1$.
To describe the labelings of the leaves of 
$\bigl(P_n(E);{}^*,1_{P_n(E)},\Gamma\bigr)$
that we will need later on,
we have to first define the appropriate 
partially ordered $\Gamma$-set of labels.
To this end let ${\mathbb S}$ denote the set of all functions
$(\Gamma\vect{y},\vect{y})\to A$ 
whose domains are pointed $\Gamma$-orbits in $A^m$ for some $m\ge1$.
We define an action of $\Gamma$ on ${\mathbb S}$ as follows:
\begin{itemize}
\item
for arbitrary element 
$\mu\colon(U,\vect{y})\to A$
of ${\mathbb S}$ with $U=\Gamma\vect{y}$ 
and for any $\gamma\in\Gamma$, the function $\gamma\mu$
is $\mu$ considered as a function 
$(U,\gamma\vect{y})\to A$.
\end{itemize}
That is, the only difference between $\mu$ and $\gamma\mu$ is
in the distinguished element of the orbit $U$.
Clearly, $\gamma\mu\in{\mathbb S}$ and 
$(\gamma\gamma')\mu=\gamma(\gamma'\mu)$ hold for
all $\gamma,\gamma'\in\Gamma$ and $\mu\in{\mathbb S}$,
so we have obtained a $\Gamma$-set $({\mathbb S};\Gamma)$.

Now we define a quasiorder $\preceq$ on ${\mathbb S}$.
Let $\mu\colon(U,\vect{y})\to A$ and $\nu\colon(V,\vect{z})\to A$
be arbitrary elements of ${\mathbb S}$ where $U=\Gamma\vect{y}$,
$V=\Gamma\vect{z}$, and $\vect{y}\in A^m$, $\vect{z}\in A^n$.
For any tuple $\vect{x}\in A^m$ let 
$\vect{x}^\flat$ denote the set of coordinates of $\vect{x}$.
We define $\mu\preceq\nu$ by the following condition:
\begin{itemize}
\item
$\mu\preceq\nu$ if and only if $\Gamma_\vect{y}\subseteq\Gamma_\vect{z}$,
$\vect{y}^\flat\supseteq\vect{z}^\flat$, and 
$\mu=\nu\circ\chi_{\vect{y},\vect{z}}$ where
$\chi_{\vect{y},\vect{z}}$       
is the unique homomorphism
$(\Gamma\vect{y};\vect{y},\Gamma)\to(\Gamma\vect{z};\vect{z},\Gamma)$,
$\gamma\vect{y}\mapsto\gamma\vect{z}$
of pointed $\Gamma$-sets.
\end{itemize}
$\sim$ will denote the intersection of $\preceq$ with its converse.
It follows from the definitions of $\sim$ and $\preceq$ that
$\mu\sim\nu$ if and only if
$\Gamma_\vect{y}=\Gamma_\vect{z}$, 
$\vect{y}^\flat=\vect{z}^\flat$, and 
$\mu=\nu\circ\chi_{\vect{y},\vect{z}}$,
$\nu=\mu\circ\chi_{\vect{z},\vect{y}}$.
The equality $\Gamma_\vect{y}=\Gamma_\vect{z}$
implies that $\chi_{\vect{y},\vect{z}}$ and
$\chi_{\vect{z},\vect{y}}$ are mutually inverse isomorphisms
between the pointed $\Gamma$-sets $(\Gamma\vect{y};\vect{y},\Gamma)$
and $(\Gamma\vect{z};\vect{z},\Gamma)$.
Therefore
\begin{itemize}
\item
$\mu\sim\nu$ if and only if $\Gamma_\vect{y}=\Gamma_\vect{z}$,
$\vect{y}^\flat=\vect{z}^\flat$, and 
$\mu=\nu\circ\chi_{\vect{y},\vect{z}}$ where
$\chi_{\vect{y},\vect{z}}$       
is the unique isomorphism
$(\Gamma\vect{y};\vect{y},\Gamma)\to(\Gamma\vect{z};\vect{z},\Gamma)$,
$\gamma\vect{y}\mapsto\gamma\vect{z}$
of pointed $\Gamma$-sets.
\end{itemize}
 
The next lemma summarizes some elementary
consequences of these definitions that
we will need later on.

\begin{lemma}
\label{basic_prec}
Let $({\mathbb S};\Gamma)$ be the $\Gamma$-set, and let
$\preceq$ and $\sim$ be the relations on ${\mathbb S}$
defined above.
\begin{enumerate}[\indent\rm(1)]
\item
$\preceq$ is a quasiorder, i.e., it is reflexive and transitive.
\item
$\sim$ is an equivalence relation, and $\preceq$ induces a partial order
$\le$ on the quotient set ${\mathbb S}/{\sim}$ by
\[
\mu/{\sim}\le\nu/{\sim}
\quad\iff\quad
\mu\preceq\nu
\qquad
\text{for all $\mu,\nu\in{\mathbb S}$}.
\]
\item
$\Gamma$ acts on ${\mathbb S}$ by automorphisms of the relational 
structure $({\mathbb S};\preceq,\sim)$.
\item
The quotient structure $({\mathbb S}/{\sim};\le,\Gamma)$ is a 
partially ordered set
on which $\Gamma$ acts by automorphisms of 
$({\mathbb S}/{\sim};\le)$.
\item
The number of $\sim$-classes of ${\mathbb S}$ is at most 
$|A|^{|A|+2|\Gamma|}$, hence ${\mathbb S}/{\sim}$ is finite.
\end{enumerate}
\end{lemma}

\begin{proof}
Let $\lambda\colon(T,\vect{x})\to A$,
$\mu\colon(U,\vect{y})\to A$, and $\nu\colon(V,\vect{z})\to A$
be arbitrary elements of ${\mathbb S}$ where 
$T=\Gamma\vect{x}$, $U=\Gamma\vect{y}$, $V=\Gamma\vect{z}$, 
and $\vect{x}\in A^l$, $\vect{y}\in A^m$, $\vect{z}\in A^n$.

(1) $\mu\preceq\mu$, since 
$\Gamma_\vect{y}=\Gamma_\vect{y}$,
$\vect{y}^\flat=\vect{y}^\flat$, and 
$\chi_{\vect{y},\vect{y}}$ is the identity function $U\to U$,
so $\mu=\mu\circ\chi_{\vect{y},\vect{y}}$.
Thus $\preceq$ is reflexive. 
To verify that $\preceq$ is transitive, assume that
$\lambda\preceq\mu\preceq\nu$, that is,
$\Gamma_\vect{x}\subseteq\Gamma_\vect{y}\subseteq\Gamma_\vect{z}$,
$\vect{x}^\flat\supseteq\vect{y}^\flat\supseteq\vect{z}^\flat$, and 
$\lambda=\mu\circ\chi_{\vect{x},\vect{y}}$, 
$\mu=\nu\circ\chi_{\vect{y},\vect{z}}$.
Then $\Gamma_\vect{x}\subseteq\Gamma_\vect{z}$,
$\vect{x}^\flat\supseteq\vect{z}^\flat$, and 
$\lambda=\nu\circ(\chi_{\vect{y},\vect{z}}\circ\chi_{\vect{x},\vect{y}})$.
Since $\chi_{\vect{y},\vect{z}}\circ\chi_{\vect{x},\vect{y}}=
\chi_{\vect{x},\vect{z}}$, we get that
$\lambda=\nu\circ\chi_{\vect{x},\vect{z}}$, proving that
$\lambda\preceq\nu$.

(2) is an immediate consequence of (1).

(3) Since $\sim$ is the intersection of $\preceq$ and its converse,
it is enough to prove that $\Gamma$ acts by automorphisms of
$({\mathbb S};\preceq)$. 
To this end we need to show that $\mu\preceq\nu$ implies
$\gamma\mu\preceq\gamma\nu$ for all $\gamma\in\Gamma$.
Let $\mu\preceq\nu$, that is, 
$\Gamma_\vect{y}\subseteq\Gamma_\vect{z}$,
$\vect{y}^\flat\supseteq\vect{z}^\flat$, and 
$\mu=\nu\circ\chi_{\vect{y},\vect{z}}$.
Then 
\begin{gather}
\Gamma_{\gamma\vect{y}}=
\gamma\Gamma_\vect{y}\gamma^{-1}\subseteq\gamma\Gamma_\vect{z}\gamma^{-1}
=\Gamma_{\gamma\vect{z}},\notag\\
(\gamma\vect{y})^\flat=
\gamma(\vect{y}^\flat)\supseteq\gamma(\vect{z}^\flat)= 
(\gamma\vect{z})^\flat,\notag
\end{gather}
and $\gamma\mu=\gamma\nu\circ\chi_{\gamma\vect{y},\gamma\vect{z}}$, 
because $\mu=\nu\circ\chi_{\vect{y},\vect{z}}$,
$\chi_{\vect{y},\vect{z}}=\chi_{\gamma\vect{y},\gamma\vect{z}}$,
and $\mu$, $\gamma\mu$ are the same function $U\to A$ and
$\nu$, $\gamma\nu$ are the same function $V\to A$.
This proves that $\gamma\mu\preceq\gamma\nu$.

(4) is an immediate consequence of (2) and (3).

(5) We saw earlier that
$\mu\sim\nu$ if and only if
$\Gamma_\vect{y}=\Gamma_\vect{z}$, 
$\vect{y}^\flat=\vect{z}^\flat$, and 
$\mu=\nu\circ\chi_{\vect{y},\vect{z}}$
for the unique ismomorphism $\chi_{\vect{y},\vect{z}}$
between the pointed $\Gamma$-sets $(U;\vect{y},\Gamma)$
and $(V;\vect{z},\Gamma)$.
The equality $\Gamma_\vect{y}=\Gamma_\vect{z}$ also
implies that
$(U;\vect{y},\Gamma)$
and $(V;\vect{z},\Gamma)$ are isomorphic
to the pointed $\Gamma$-set $(\Gamma/\Gamma_\vect{y};\Gamma_\vect{y},\Gamma)$
of the left cosets of $\Gamma_\vect{y}$ under 
the natural action of $\Gamma$ by left multiplication.
Therefore
the number of $\sim$-classes in ${\mathbb S}$
is at most the number of triples 
$(\vect{y}^\flat,\Gamma_\vect{y},\sigma)$
where 
$\vect{y}^\flat$ is a subset of $A$,
$\Gamma_\vect{y}$ is a subgroup of $\Gamma$,
and
$\sigma$ is a function $(\Gamma/\Gamma_\vect{y};\Gamma_\vect{y})\to A$.
Hence  the number of $\sim$-classes is at most
$2^{|A|}2^{|\Gamma|}|A|^{|\Gamma|}\le|A|^{|A|+2|\Gamma|}$, as claimed.
\end{proof}

If $g$ is an $n$-ary operation on $A$, we define an 
${\mathbb S}/{\sim}$-labeling $\ell_g$ of the leaves of 
the $\Gamma$-tree $\bigl(P_n(E);{}^*, 1_{P_n(E)},\Gamma\bigr)$
by
\begin{itemize}
\item
$\ell_g\bigl((0,\vect{x})\bigr)=g|_{(\Gamma\vect{x},\vect{x})}/{\sim}$
for all $\vect{x}\in A^n$
\end{itemize}
where $g|_{(\Gamma\vect{x},\vect{x})}$ denotes the restriction
of $g$ to the pointed $\Gamma$-orbit $(\Gamma\vect{x},\vect{x})$;
thus $g|_{(\Gamma\vect{x},\vect{x})}$ is an element of ${\mathbb S}$.
This labeling yields an ${\mathbb S}/{\sim}$-labeled $\Gamma$-tree
$\bigl(P_n(E);{}^*,\linebreak[0] 1_{P_n(E)},\Gamma;\ell_g\bigr)$,
which we will denote by $\al P_g(E)$, 
and will call
\emph{the ${\mathbb S}/{\sim}$-labeled $\Gamma$-tree 
associated to $f$.}

\begin{example}
\label{ex-treelabel}
Let $A$, $E$, and $\Gamma$ be as in Example~\ref{ex-tree}, and let
$g$ be the unary operation on $A$ defined by $g(1)=2$, $g(2)=4$,
$g(3)=4$, and $g(4)=3$.
The ${\mathbb S}/{\sim}$-labeled $\Gamma$-tree $\al P_g(E)$
is obtained from the $\Gamma$-tree 
$\bigl(P_1(E);{}^*,\linebreak[0] 1_{P_1(E)},\Gamma\bigr)$
in Example~\ref{ex-tree} by labeling the leaves via $\ell_g$.
For each leaf $(0,x)$ ($x\in A=\{1,2,3,4\}$), the label of $(0,x)$ is the
equivalence class $\mu_x/{\sim}$ where 
$\mu_x\colon (\Gamma x,x)\to A$ is the restriction of $g$ to the pointed 
$\Gamma$-orbit of $x$; i.e.,
\begin{align*}
&\mu_1\colon (\{1,2\},1)\to A,\quad 1\mapsto 2,\ 2\mapsto 4;\\
&\mu_2\colon (\{1,2\},2)\to A,\quad 1\mapsto 2,\ 2\mapsto 4;\\
&\mu_3\colon (\{3,4\},3)\to A,\quad 3\mapsto 4,\ 4\mapsto 3;\\
&\mu_4\colon (\{3,4\},4)\to A,\quad 3\mapsto 4,\ 4\mapsto 3.
\end{align*}
The functions $\mu_x$ ($x\in A$) belong to pairwise different
$\sim$-classes,  because $x^\flat=\{x\}\not=\{y\}=y^\flat$ 
for distinct elements $x,y\in A$. Therefore 
the labeling $\ell_g$ assigns four distinct labels
to the four leaves.
\end{example}

The next lemma shows the relevance of the 
${\mathbb S}/{\sim}$-labeled $\Gamma$-trees
$\al P_f(E)$ and $\al P_g(E)$
to the problem of determining whether $f\subf[\cl C]g$ holds
for two operations $f,g$ on $A$.

\begin{lemma}
\label{equiv-lessthan}
Let $E$ be a chain of equivalence relations on a finite set $A$, and
let $\cl C=\Pol\bigl(E,\Aut{E},\cl{P}^+(A)\bigr)$.
For arbitrary operations $f,g$ on $A$, 
$f \subf[\cl{C}] g$ if and only if
there exists a label-increasing $\Gamma$-homomorphism 
$\al P_f(E)\nearrow\al P_g(E)$
between the  ${\mathbb S}/{\sim}$-labeled $\Gamma$-trees 
associated to $f$ and $g$. 
\end{lemma}

\begin{proof}
Let $f$ be $m$-ary and $g$ be $n$-ary.
To prove the forward implication
assume that $f \subf[\cl{C}] g$,
and let $\vect{h}\in(\cl C^{(m)})^n$ be such that $f=g\circ\vect{h}$.
Since $\vect{h}$ preserves the equivalence relations in $E$,
$\vect{h}$ maps each block $B$ of $(\rho_i)^m$ into a block of $(\rho_i)^n$.
Thus $\vect{h}$ induces a map 
$$
\psi\colon P_m(E)\to P_n(E),
\quad
(i,B)\mapsto\bigl(i,\overline{\vect h(B)}\bigr)
$$
where $\overline{\vect{h}(B)}$ denotes the block of $(\rho_i)^n$
containing $\vect{h}(B)$.
We claim that $\psi$ is a label-increasing $\Gamma$-homomorphism 
$\al P_f(E)\nearrow\al P_g(E)$.
Clearly, $\psi$ maps 
$1_{P_m(E)}=(r+1,A^m)$ to $1_{P_n(E)}=(r+1,A^n)$, and  
it maps leaves to leaves.
Furthermore, if $(i,B)\in P_m(E)$
with $0\le i\le r$, then $(i,B)^*=(i+1,C)$
for the unique block $C$ of $(\rho_{i+1})^m$ satisfying $B\subseteq C$.
Therefore $\vect{h}(B)\subseteq \vect{h}(C)$,
so
$\overline{\vect{h}(B)}\subseteq \overline{\vect{h}(C)}$,
which shows that 
$$
\psi\bigl((i,B)^*\bigr)
=\psi\bigl((i+1,C)\bigr)
=\bigl(i+1,\overline{\vect{h}(C)}\bigr)
=\bigl(i,\overline{\vect{h}(B)}\bigr)^*
=\psi((i,B))^*.
$$
Thus $\psi$ is a homomorphism of trees.
Next, if $(i,B)\in P_m(E)$ and $\gamma\in\Gamma$, then
$$
\psi\bigl(\gamma((i,B))\bigr)
=\psi\bigl((i,\gamma(B))\bigr)=\bigl(i,\overline{\vect{h}(\gamma(B))}\bigr)
$$
and
$$
\gamma\bigl(\psi((i,B))\bigr)=
\gamma\bigl((i,\overline{\vect{h}(B)})\bigr)=
\bigl(i,\gamma(\overline{\vect{h}(B)})\bigr)=
\bigl(i,\overline{\gamma(\vect{h}(B))}\bigr).
$$
Since $\vect{h}$ preserves $\gamma$, we have
$\vect{h}\bigl(\gamma(B)\bigr)=\gamma\bigl(\vect{h}(B)\bigr)$,
proving
$\psi\bigl(\gamma((i,B))\bigr)=\linebreak[4]\gamma\bigl(\psi((i,B))\bigr)$.
Hence $\psi$ is a $\Gamma$-homomorphism 
$\bigl(P_m(E);{}^*,1_{P_m(E)},\Gamma\bigr)
\to \bigl(P_n(E);\linebreak[0]{}^*,1_{P_n(E)},\Gamma\bigr)$.

Finally, if $(0,\vect{x})$ is a leaf of $P_m(E)$, 
then using the definition of the labelings
$\ell_f$ and $\ell_g$ and the relationship $f=g\circ\vect{h}$ we
get that 
\begin{align*}
\ell_f\bigl((0,\vect{x})\bigr)
&{}= f|_{(\Gamma\vect{x},\vect{x})}/{\sim},\\
\ell_g\bigl(\psi((0,\vect{x}))\bigr)
&{}= \ell_g\bigl((0,\vect{h}(\vect{x}))\bigr)=
g|_{(\Gamma\vect{h}(\vect{x}),\vect{h}(\vect{x}))}/{\sim},
\end{align*}
and
\[
f|_{(\Gamma\vect{x},\vect{x})}
=(g\circ\vect{h})|_{(\Gamma\vect{x},\vect{x})}
=g|_{(\Gamma\vect{h}(\vect{x}),\vect{h}(\vect{x}))}
           \circ\vect{h}|_{(\Gamma\vect{x},\vect{x})}.
\]
Here $\vect{h}|_{(\Gamma\vect{x},\vect{x})}\colon
  \bigl(\Gamma\vect{x},\vect{x}\bigr)\to
    \bigl(\Gamma\vect{h}(\vect{x}),\vect{h}(\vect{x})\bigr)$
is a homomorphism of pointed $\Gamma$-orbits, since $\vect{h}$ preserves
all permutations $\gamma\in\Gamma$.
Thus $\Gamma_\vect{x}\subseteq\Gamma_{\vect{h}(\vect{x})}$
and  $\vect{h}|_{(\Gamma\vect{x},\vect{x})}=
\chi_{\vect{x},\vect{h}(\vect{x})}$.
In addition, we have 
$\vect{x}^\flat\supseteq\vect{h}(\vect{x})^\flat$, since
$\vect{h}$ preserves all subsets of $A$.
Thus 
\[
f|_{(\Gamma\vect{x},\vect{x})}
=g|_{(\Gamma\vect{h}(\vect{x}),\vect{h}(\vect{x}))}
           \circ \chi_{\vect{x},\vect{h}(\vect{x})}
\preceq g|_{(\Gamma\vect{h}(\vect{x}),\vect{h}(\vect{x}))},
\]
implying that 
$\ell_f\bigl((0,\vect{x})\bigr)\le\ell_g\bigl(\psi((0,\vect{x}))\bigr)$.
This proves that $\psi$ is a label-increasing $\Gamma$-homomorphism
$\al P_f(E)\nearrow\al P_g(E)$, and hence 
concludes the proof of the forward implication.

For the converse, assume that there exists 
a label-increasing $\Gamma$-ho\-mo\-mor\-phism
$\psi\colon \al P_f(E)\nearrow \al P_g(E)$. 
Our goal is to show that $f\subf[\cl C]g$. 
Since $\psi$ is a homomorphism of trees, therefore it 
maps each leaf of $\al P_f(E)$ into a leaf of $\al P_g(E)$.
Hence $\psi$ yields a function $\vect{h}\colon A^m\to A^n$ 
such that $\psi((0,\vect{x}))=(0,\vect{h}(\vect{x}))$ for all
$\vect{x}\in A^m$.
We will establish $f\subf[\cl C]g$ by proving that 
$\vect h\in(\cl C^{(m)})^n$ and $f=g\circ\vect{h}$.

First we will show that $\vect{h}$ preserves all equivalence relations
$\rho_i$ ($1\le i\le r$).
Let $\vect{x},\vect{y}\in A^m$ be such that
$\vect{x}\,\,(\rho_i)^m\,\,\vect{y}$. Then
$\vect{x},\vect{y}$ are in the same block $B$ of $(\rho_i)^m$,
i.e.,
$(0,\vect{x})^{\ast^i}=(i,B)=(0,\vect{y})^{\ast^i}$.
Since $\psi$ is a homomorphism of trees, we get that
\begin{multline*}
\bigl(0,\vect{h}(\vect{x})\bigr)^{\ast^i}
=\psi\bigl((0,\vect{x})\bigr)^{\ast^i}
=\psi\bigl((0,\vect{x})^{\ast^i}\bigr) \\
=\psi\bigl((0,\vect{y})^{\ast^i}\bigr)
=\psi\bigl((0,\vect{y})\bigr)^{\ast^i}
=\bigl(0,\vect{h}(\vect{y})\bigr)^{\ast^i}.
\end{multline*}
Hence $\vect{h}(\vect{x})$ and 
$\vect{h}(\vect{y})$ are in the same block of
$(\rho_i)^n$, that is,
$\vect{h}(\vect{x})\,\,(\rho_i)^n\,\,\vect{h}(\vect{y})$.

Next we show that $\vect{h}$ preserves all permutations $\gamma\in\Gamma$.
Since $\psi$ is a $\Gamma$-ho\-mo\-mor\-phism and $\Gamma$ acts on the leaves
of $\al P_f(E)$ and $\al P_g(E)$
by $\gamma\cdot(0,\vect{u})=(0,\gamma\vect{u})$ for all $\vect{u}$ and 
$\gamma$, we get that
\begin{multline*}
\bigl(0,\vect{h}(\gamma\vect{x})\bigr)
=\psi\bigl((0,\gamma\vect{x})\bigr)
=\psi\bigl(\gamma\cdot(0,\vect{x})\bigr) \\
=\gamma\cdot\bigl(\psi(0,\vect{x})\bigr)
=\gamma\cdot\bigl(0,\vect{h}(\vect{x})\bigr)
=\bigl(0,\gamma\vect{h}(\vect{x})\bigr)
\end{multline*}
for all $\vect{x}\in A^m$ and $\gamma\in\Gamma$.
Hence 
$\vect{h}\bigl(\gamma(\vect{x})\bigr)=\gamma\bigl(\vect{h}(\vect{x})\bigr)$
for all $\vect{x}\in A^m$ and $\gamma\in\Gamma$, 
as claimed.

This also proves that $\vect{h}$ restricts to
every pointed $\Gamma$-orbit $(\Gamma\vect{u},\vect{u})$ in $A^m$ 
as a homomorphism 
$\vect{h}|_{(\Gamma\vect{u},\vect{u})}\colon (\Gamma\vect{u},\vect{u})
\to \bigl(\Gamma\vect{h}(\vect{u}),\vect{h}(\vect{u})\bigr)$
between two pointed $\Gamma$-orbits.
Since such a homomorphism exists only if 
$\Gamma_\vect{u}\subseteq\Gamma_{\vect{h}(\vect{u})}$,
and when it exists, it is uniquely determined, we get that
$\vect{h}|_{(\Gamma\vect{u},\vect{u})}=\chi_{\vect{u},\vect{h}(\vect{u})}$.

Since 
\begin{gather}
\ell_f\bigl((0,\vect{u})\bigr)=f|_{(\Gamma\vect{u},\vect{u})}/{\sim},\notag\\
\ell_g\bigl(\psi((0,\vect{u}))\bigr)=\ell_g\bigl((0,\vect{h}(\vect{u}))\bigr)=
g|_{(\Gamma\vect{h}(\vect{u}),\vect{h}(\vect{u}))}/{\sim},\notag
\end{gather}
and $\psi$ is label-increasing,
we get that  
$f|_{(\Gamma\vect{u},\vect{u})}\preceq 
g|_{(\Gamma\vect{h}(\vect{u}),\vect{h}(\vect{u}))}$.
By the definition of $\preceq$ this means that
$\Gamma_\vect{u}\subseteq\Gamma_{\vect{h}(\vect{u})}$,
$\vect{u}^\flat\supseteq\vect{h}(\vect{u})^\flat$,
and 
$f|_{(\Gamma\vect{u},\vect{u})}=
g|_{(\Gamma\vect{h}(\vect{u}),\vect{h}(\vect{u}))}\circ
\chi_{\vect{u},\vect{h}(\vect{u})}$.
Combining this with the equality
$\vect{h}|_{(\Gamma\vect{u},\vect{u})}=\chi_{\vect{u},\vect{h}(\vect{u})}$
we get that
\[
f|_{(\Gamma\vect{u},\vect{u})}=
g|_{(\Gamma\vect{h}(\vect{u}),\vect{h}(\vect{u}))}\circ
\vect{h}|_{(\Gamma\vect{u},\vect{u})}.
\]
Since $A^m$ is the union of all $\Gamma$-orbits $\Gamma\vect{u}$,
we obtain from the last displayed equality that $f=g\circ\vect{h}$.
The property that $\vect{u}^\flat\supseteq\vect{h}(\vect{u})^\flat$
for all $\vect{u}\in A^m$ shows that $\vect{h}$ preserves all
subsets of $A$.
Thus $\vect h\in(\cl C^{(m)})^n$ and $f=g\circ\vect{h}$,
which proves that $f\subf[\cl C]g$. 
\end{proof}

It follows from Lemma~\ref{equiv-lessthan} that
$f \fequiv[\cl{C}] g$ holds for two operations $f,g$ on $A$
if and only if
there exist label-increasing $\Gamma$-homomorphisms 
$\al P_f(E)\nearrow\al P_g(E)$ and $\al P_g(E)\nearrow\al P_f(E)$
between the ${\mathbb S}/{\sim}$-labeled $\Gamma$-trees 
associated to $f$ and $g$. 
Since the size of $\al P_f(E)$ increases with the arity of $f$,
this lemma alone is not enough to conclude that 
the number of $\fequiv[\cl{C}]$-classes is finite.
We want to replace each $\al P_f(E)$ by an 
${\mathbb S}/{\sim}$-labeled $\Gamma$-tree
$\widehat{\al P}_f$ that is
\begin{itemize}
\item
homomorphically equivalent to $\al P_f(E)$, that is, there exist
label-preserving $\Gamma$-homomorphisms $\al P_f(E)\to\widehat{\al P}_f$ 
and $\widehat{\al P}_f\to\al P_f(E)$, and
\item
as small as possible with this property.
\end{itemize}
The first condition is to ensure that the analog of 
Lemma~\ref{equiv-lessthan} remains true if, instead of $\al P_f(E)$, 
we associate $\widehat{\al P}_f$ to each operation $f$.
The second condition will allow us to prove that, up to isomorphism, 
there are only finitely many $\widehat{\al P}_f$'s, and 
hence it will follow that 
the number of $\fequiv[\cl{C}]$-classes is finite.

The intended relationship between $\widehat{\al P}_f$ and $\al P_f(E)$ 
is captured by the concept of a \emph{core}, which applies to arbitrary
finite structures. For our purposes it will be enough to
discuss cores of $S$-labeled $G$-trees.

Let $\al P=(P;{}^*,1_P,G; \linebreak[0] \ell_P)$ and
$\al Q=(Q;{}^*,1_Q,G; \linebreak[0] \ell_Q)$ 
be $S$-labeled $G$-trees.
We say that 
\begin{enumerate}[\indent\rm(1)]
\item
\emph{$\al Q$ is a core} if every
label-preserving $G$-homomorphism $\al Q\to\al Q$ is onto;
\item
\emph{$\al Q$ is a core of $\al P$} if 
\begin{itemize}
\item
$\al Q$ is \emph{homomorphically equivalent} to $\al P$, that is,
there exist label-pre\-serv\-ing $G$-homomorphisms $\al P\to\al Q$ and
$\al Q\to\al P$, and 
\item
$\al Q$ is minimal with this property (i.e., no proper labeled
$G$-subtree of $\al Q$ is homomorphically equivalent to $\al P$).
\end{itemize}
\end{enumerate}

For the reader's convenience we will state and prove the basic properties
of cores for $S$-labeled $G$-trees. The first one of these properties
is that the two uses of the word `core' in the definitions above 
are compatible:
every core of an $S$-labeled $G$-tree [in the sense of (2)] 
is actually a core [in the sense of (1)].
We will use this propery later on without further reference.
The second and third properties show that every $S$-labeled $G$-tree
has a core (in fact, it has one among its $S$-labeled $G$-subtrees),
and the core is uniquely determined, up to isomorphism.

\begin{lemma}
\label{reduction}
Let $\al P$
be an $S$-labeled $G$-tree.
\begin{enumerate}[\indent\rm(1)]
\item
Every core of $\al P$ is a core.
\item
If $\widehat{\al P}$ is minimal, with respect to inclusion, 
among all $S$-labeled $G$-subtrees $\al P'$ of $\al P$ 
for which there exists a label-preserving $G$-ho\-mo\-mor\-phism
$\al P\to\al P'$, then $\widehat{\al P}$ is a core of $\al P$. 
\item
Any two cores of $\al P$ are isomorphic.
\end{enumerate}
\end{lemma}

\begin{proof}
(1)
Let $\al Q$ be a core of $\al P$.
It follows that 
there exist label-preserving $G$-homomorphisms 
$\phi\colon\al P\to\al Q$ and $\psi\colon\al Q\to\al P$. 
To prove that $\al Q$ is a core, we need to show that
every label-preserving $G$-homomorphism $\tau\colon\al Q\to\al Q$
is onto. The range $\al R$ of $\tau$ is an $S$-labeled $G$-subtree of
$\al Q$, therefore the identity embedding $\iota\colon\al R\to\al Q$
is a label-preserving $G$-homomorphism.
Thus $\tau=\iota\circ\tilde\tau$ for some label-preserving
$G$-homomorphism $\tilde\tau\colon\al Q\to\al R$. Hence we have
label-preserving $G$-homomorphisms
$$
\al P\overset{\phi}{\underset{\psi}{\rightleftarrows}}\al Q
\overset{\tilde\tau}{\underset{\iota}{\rightleftarrows}}\al R,
$$
which implies that $\al R$ is homomorphically equivalent to $\al P$,
as witnessed by $\tilde\tau\circ\phi\colon\al P\to\al R$ and
$\psi\circ\iota\colon\al R\to\al P$.
Since $\al Q$ is a core of $\al P$, the $S$-labeled $G$-subtree $\al
R$ of $\al Q$ cannot be proper. Thus $\al R=\al Q$ and $\tau$ is onto.

(2)
Let $\widehat{\al P}$ be minimal, with respect to inclusion, 
among all $S$-labeled $G$-subtrees $\al P'$ of $\al P$ 
for which there exists a label-preserving $G$-homomorphism
$\al P\to\al P'$.
Such a $\widehat{\al P}$ exists, since $\al P$ is finite. Moreover, the
identity embedding $\widehat{\al P}\to\al P$ is a label-preserving
$G$-homomorphism, because $\widehat{\al P}$ is an $S$-labeled $G$-subtree of
$\al P$. Thus $\widehat{\al P}$ is homomorphically equivalent to $\al P$.
The choice that $\widehat{\al P}$ is minimal among the 
$S$-labeled $G$-subtrees $\al P'$
of $\al P$ for which there exists a label-preserving $G$-homomorphism
$\al P\to\al P'$ ensures that 
$\widehat{\al P}$ is also minimal among the $S$-labeled $G$-subtrees
of $\al P$ that are homomorphically equivalent to $\al P$.
This proves that $\widehat{\al P}$ is a core of $\al P$, as claimed.

(3)
Let $\al Q$ and $\al Q'$ be cores of $\al P$. 
Then $\al Q$ and $\al Q'$ are homomorphically equivalent to $\al P$, so
we can choose label-preserving $G$-homomorphisms
$$
\al Q\overset{\phi}{\underset{\psi}{\leftrightarrows}}\al P
\overset{\phi'}{\underset{\psi'}{\rightleftarrows}}\al Q'
$$
witnessing this fact.
Thus we have  label-preserving $G$-homomorphisms
$$
\al Q\overset{\sigma}{\longrightarrow}\al Q',\quad
\al Q\overset{\,\sigma'}{\longleftarrow}\al Q',\quad
\al Q\overset{\sigma'\circ\sigma}{\longrightarrow}\al Q,\quad
\text{and}\quad
\al Q'\overset{\,\sigma\circ\sigma'}{\longleftarrow}\al Q'
$$
where $\sigma=\phi'\circ\psi$ and $\sigma':=\phi\circ\psi'$.
Since $\al Q$ and $\al Q'$  are cores by part (1), the
latter two label-preserving $G$-homomorphisms are onto.
Since $\al Q$ and $\al Q'$ are finite,
they are also one-to-one.
This implies that $\sigma$ and $\sigma'$ are 
both onto and one-to-one, hence they are isomorphisms.
\end{proof}

To prove that for each $d$ there are, 
up to isomorphism, 
only finitely many
$S$-labeled trees of uniform depth $d$ that are cores
(Lemma~\ref{bound-on-reduced}), we need some 
necessary conditions for an $S$-labeled $G$-tree
to be a core
(Corollary~\ref{reduced-trees}). 
These necessary conditions will be derived
from a general lemma on label-preserving
$G$-homomorphisms between $S$-labeled $G$-trees (Lemma~\ref{homs}).

We start with some preparation.
Let $\al P=(P;{}^*,1_P,G;\ell)$ 
be an $S$-labeled $G$-tree.
The set of elements of depth $1$ in $\al P$, that is, the set
of elements $a\in P$ such that $a^*=1_P$, will be denoted
by $\al P_\ttop$.
For any $a\in\al P_\ttop$ the set of all elements
$b\in P$ such that $b^{*^i}=a$ for some 
integer $i\ge0$ will be denoted by $(a]$.
The next lemma summarizes some basic facts on $\al P_\ttop$ and $(a]$
($a\in\al P_\ttop$) that we will need later on.

\begin{lemma}
\label{basic_tree}
Let $\al P=(P;{}^*,1_P,G;\ell)$ 
be an $S$-labeled $G$-tree.
\begin{enumerate}[\indent\rm(1)]
\item
$(a]\cap(b]=\emptyset$ if $a,b$ are distinct elements of $\al P_\ttop$, and
\[
P=\{1_P\}\cup\bigcup\bigl((a]:a\in\al P_\ttop\bigr).
\]
\item
$\al P_\ttop$ and $\al P_\leaf$ are unions of $G$-orbits.
\item
For each $a\in\al P_\ttop$,
\begin{enumerate}[{\rm(i)}]
\item
if $c\in(a]$, $c\not=a$, then $c^*\in(a]$;
\item
if $g\in G$, then $g\cdot(a]=(ga]$;
hence $g\cdot(a]=(a]$ if and only if $g\in G_a$.
\end{enumerate}
\item
For each $a\in\al P_\ttop$, $(a]$ is the underlying set of an
$S$-labeled $G_a$-tree
\[
(a]_{\al P}:=\bigl((a];{}^*,1_{(a]},G_a;\ell\bigr)
\]
where $1_{(a]}=a$, ${}^*$ is the restriction of the successor
function of $\al P$ to the set $(a]\setminus\{a\}$,
the action of each $g\in G_a$ on $(a]$ is obtained by restricting
the action of $g$ to $(a]$, and 
$\ell$ is the restriction of the labeling of 
the leaves of $\al P$ to the leaves of $(a]$.
\item
$\bigl((a]_\al P\bigr)_\leaf=\al P_\leaf\cap(a]$ for each $a\in\al P_\ttop$,
so
if $|P|>1$, then
\[
\al P_\leaf=\bigcup\bigl(((a]_{\al P})_\leaf:a\in\al P_\ttop\bigr).
\]
\end{enumerate}
\end{lemma}

\begin{proof}
Recall that for each element
$u\in P\setminus\{1_P\}$ there exists a unique positive integer
$d$, the depth of $u$, such that $u^{*^d}=1_P$. 
Thus $u^{*^{d-1}}\in\al P_\ttop$ and $u\in(u^{*^{d-1}}]$,
which proves the displayed equality in (1).
Moreover, if $u\in(a]$ for some $a\in\al P_\ttop$, then
the definitions of $\al P_\ttop$ and $(a]$ yield that $a^*=1_P$
and $u^{*^i}=a$ for some integer $i\ge0$.
Thus $u^{*^{i+1}}=1_P$, and the uniqueness of the depth of $u$
implies that $d=i+1$.
Hence $a=u^{*^{d-1}}$, showing that $u\in(a]$ for a unique $a\in\al P_\ttop$. 
This completes the proof of (1).

(2) and (3) are immediate consequences of the definitions, using also
the fact that
each $g\in G$ acts by automorphisms of the tree $(P; 1_P,{}^*)$.
(3) ensures that ${}^*$ and $g\in G_a$ restrict to $(a]$ as claimed.
The properties of the operations of $\bigl((a];{}^*,1_{(a]},G_a\bigr)$ 
that make it a $G_a$-tree are inherited from $\al P$.
Furthermore,
it follows from the definition of $(a]$ that
the leaves of the tree $((a];1_{(a]},{}^*)$ are exactly 
the leaves of $\al P$ that are in $(a]$.
This establishes the first equality in (5), and also implies that
the restriction of $\ell$ to $(a]$ (also denoted by $\ell$) yields 
an $S$-labeling of the leaves of the $G_a$-tree
$\bigl((a];{}^*,1_{(a]},G_a\bigr)$. 
This proves (4).
Finally, the displayed equality in (5) follows from the equality 
$\bigl((a]_\al P\bigr)_\leaf=\al P_\leaf\cap(a]$
proved earlier and
the displayed equality in (1).
\end{proof}

It follows from the preceding lemma that every $S$-labeled $G$-tree
is the disjoint union of the $S$-labeled $G_a$-trees $(a]_\al P$
($a\in\al P_\ttop$)
with a new top element $1_P$ added.
In the next lemma 
we will use this structure of $S$-labeled $G$-trees to analyze
the label-preserving $G$-homomorphisms between them. 

\begin{lemma}
\label{homs}
Let $\al P=(P;{}^*,1_P,G;\ell_P)$ 
and $\al Q=(Q;{}^*,1_Q,G;\ell_Q)$ 
be $S$-labeled $G$-trees, and let
$\{a_i:1\le i\le t\}$ be a transversal for the $G$-orbits of
$\al P_\ttop$. 
If $b_i$ $(1\le i\le t)$ are elements of 
$\al Q_\ttop$ such that $G_{a_i}=G_{b_i}$ 
for each $i$, then 
\begin{enumerate}
\item[{\rm(1)}]
every family
$\{\psi_i:1\le i\le t\}$ of 
label-preserving $G_{a_i}$-homomorphisms
$\psi_i\colon (a_i]_{\al P}\to(b_i]_{\al Q}$ 
has a unique extension to a label-preserving $G$-ho\-mo\-mor\-phism
$\phi\colon\al P\to\al Q$.
\item[{\rm(2)}]
$\phi$ is onto if and only if
every $G$-orbit of 
$\al Q_\ttop$ contains at least one $b_i$, and
\[
(b_i]=\bigcup(h\cdot\range{\psi_j}: 1\le j\le t,\ h\in G,\ hb_j=b_i)
\]
for each $i$ $(1\le i\le t)$.
\item[{\rm(2)$'$}]
In particular, if 
$\{b_i:1\le i\le t\}$ is a transversal for the $G$-orbits of
$\al Q_\ttop$, then $\phi$ is onto if and only if
each $\psi_i$ is onto. 
\item[{\rm(3)}]
$\phi$ is bijective if and only if 
$\{b_i:1\le i\le t\}$ is a transversal for the $G$-orbits of
$\al Q_\ttop$ and 
each $\psi_i$ is bijective.
\end{enumerate}
\end{lemma}

\begin{proof}
Let $a_i,b_i$ ($1\le i\le t$) satisfy the assumptions of the lemma.
Fix an $i$ ($1\le i\le t$), and consider the $G$-orbit
$Ga_i=\{ha_i:h\in G\}$ of $a_i$. 
As we noticed in Lemma~\ref{basic_tree}~(2), $Ga_i\subseteq\al P_\ttop$.
We claim that the subset 
$P_i=\{1_P\}\cup\bigcup_{h\in G}(ha_i]$ 
of $P$ is the underlying set of an $S$-labeled 
$G$-subtree $\al P_i$ of $\al P$.
Indeed, the definition of $\al P_\ttop$ and Lemma~\ref{basic_tree}~(3)
shows that 
the successor of every element
of $P_i\setminus\{1_P\}$ is in $P_i$, and that
$P_i$ is closed under the action of $G$. 
Furthermore, it follows from the first equality in
Lemma~\ref{basic_tree}~(5)
that $(\al P_i)_\leaf=\al P_\leaf\cap\al P_i$.
This proves that 
$P_i$ is the underlying set of an $S$-labeled 
$G$-subtree $\al P_i$ of $\al P$.
In fact, $\al P_i$ is the smallest $S$-labeled $G$-subtree of 
$\al P$ that contains $(a_i]$. 
For, if $\al P'_i$ is an $S$-labeled $G$-subtree of 
$\al P$ such that $(a_i]\subseteq P_i'$, then
$1_P\in P'_i$ by the definition of a subtree, and
$(ha_i]=h\cdot(a_i]\subseteq P'_i$, since
$\al P'_i$ is closed under the action of $G$.
Thus $P_i\subseteq P'_i$.

Similarly, for each $i$ ($1\le i\le t$), the subset 
$Q_i=\{1_Q\}\cup\bigcup_{h\in G}(hb_i]_{\al Q}$ 
of $Q$ is the underlying set of an $S$-labeled 
$G$-subtree $\al Q_i$ of $\al Q$, and
$\al Q_i$ is the smallest $S$-labeled $G$-subtree of 
$\al P$ that contains $(b_i]$. 

(1)
Now assume that $\{\psi_i:1\le i\le t\}$ is a family of
label-preserving $G_{a_i}$-homomorphisms
$\psi_i\colon (a_i]_{\al P}\to(b_i]_{\al Q}$.
First we prove the uniqueness of the extension $\phi$ 
claimed in (1).
Assume $\phi\colon\al P\to\al Q$ is 
a label-preserving $G$-homomorphism that extends all 
$\psi_i$. 
Then $\phi(1_P)=1_P$ and, by Lemma~\ref{basic_tree} (3)(ii), 
for each $h\in G$ and $c\in(ha_i]$ we have $h^{-1}c\in(a_i]$, so
\[
\phi(c)=\phi\bigl(h(h^{-1}c)\bigr)=h\phi(h^{-1}c)=h\psi_i(h^{-1}c).
\]
This proves that $\phi$ is uniquely determined by the $\psi_i$'s.

To prove the existence of $\phi$ we will verify that 
under the assumptions of the lemma,
for each $i$ ($1\le i\le t$),
\begin{enumerate}
\item[(I)$_i$]
the rule
\[
\phi_i(c)=\begin{cases}
        1_Q & \text{if $c=1_P$}\\
        h\psi_i(h^{-1}c) & \text{if $c\in(ha_i]$ ($h\in G$)}
        \end{cases}
\]
defines a label-preserving $G$-homomorphism 
$\phi_i\colon \al P_i\to\al Q_i$
that extends $\psi_i$,
\end{enumerate}
and 
\begin{enumerate}
\item[(II)]
for any family $\{\phi_i:1\le i\le t\}$
of label-preserving $G$-homomorphisms 
$\phi_i\colon\al P_i\to\al Q_i$, the union
$\phi$ of the $\phi_i$'s is a label-preserving 
$G$-ho\-mo\-morphism $\al P\to\al Q$.
\end{enumerate}
We will start with (II).
By Lemma~\ref{basic_tree}~(1)
every element $c$ of $\al P$ other than $1_P$ 
belongs to a subset of the form
$(a]$ for a unique $a\in\al P_\ttop$.
Since $\{a_i:1\le i\le t\}$ is a transversal for the $G$-orbits of
$\al P_\ttop$, the $G$-orbits $Ga_i$ partition $\al P_\ttop$.
Moreover, $(\al P_i)_\ttop=Ga_i$ for each $i$, therefore 
it follows that
every element $c$ of $\al P$ other than $1_P$  
belongs to exactly one of the $G$-subtrees $\al P_i$ of $\al P$.
As for $1_P$, we have $\phi_i(1_P)=1_Q$ 
for each $i$, since $\phi_i$ is a homomorphism
of trees. 
Thus we get that 
$\phi:=\bigcup_{i=1}^t\phi_i$ is a well-defined function $P\to Q$.

To prove that $\phi$ is a label-preserving $G$-homomorphism
$\al P\to\al Q$
we have to verify that it satisfies conditions (H1)--(H4)
((H0) is established already).
Since all $\phi_i$ are label-preserving $G$-homomorphisms,
they satisfy conditions (H1)--(H4).
In particular, $\phi_i$ maps $(\al P_i)_\leaf$ into $(\al Q_i)_\leaf$.
But the displayed equality in Lemma~\ref{basic_tree}~(5) applied to
$\al P$ and each $\al P_i$ shows that
$\al P_\leaf=\bigcup_{i=1}^t (\al P_i)_\leaf$, so (H1) follows for
$\phi$.
Since each $\al P_i$ is an $S$-labeled $G$-subtree of $\al P$,
conditions (H2)--(H4) immediately follow from the corresponding
conditions for the $\phi_i$'s.
This completes the proof of (II).

For each $i$, statement (I)$_i$ is a special case of the general 
statement about the existence of $\phi$, 
namely the special case when 
$\al P_\ttop$ is a single
$G$-orbit $Ga$.
Therefore all (I)$_i$ will be proved if we show the existence of $\phi$
for the case when $t=1$ holds for $\al P$.
To simplify notation, we will omit subscripts; that is, 
we let $a$ be an element of $\al P_\ttop$, and assume that 
$\al P=\{1_P\}\cup\bigcup_{h\in G}(ha]_{\al P}$.
Furthermore, we let $b$ be an element of 
$\al Q_\ttop$ with $G_a = G_b$, and let 
$\psi\colon(a]_{\al P}\to(b]_{\al Q}$ be a label-preserving
$G_a$-homomorphism.
Our goal is to show that 
\[
\phi(c)=\begin{cases}
        1_Q & \text{if $c=1_P$}\\
        h\psi(h^{-1}c) & \text{if $c\in(ha]$ ($h\in G$)}
        \end{cases}
\]
defines a label-preserving $G$-homomorphism 
$\phi\colon \al P\to\al Q$
that extends $\psi$.

First we show that $\phi$ is a well-defined function $P\to Q$.
If $c\in(ha]$, then $h^{-1}c\in(a]$, therefore $\psi(h^{-1}c)$
is defined, and hence so is $h\psi(h^{-1}c)$.
Suppose now that $c\in(ha]$ and $c\in(ga]$. 
Since $ha,ga\in\al P_\ttop$, 
Lemma~\ref{basic_tree}~(1) shows that
$ha=ga$. Thus
$g^{-1}ha=a$, that is, $g^{-1}h\in G_a$.
Hence
\[
g\psi(g^{-1}c)=g\psi\bigl((g^{-1}h)(h^{-1}c)\bigr)
=g(g^{-1}h)\psi(h^{-1}c)=h\psi(h^{-1}c),
\]
where the middle equality holds, because $\psi$ is a $G_a$-homomorphism.
This shows that $\phi$ is well-defined. 
Clearly, $\phi$ is an extension of $\psi$, for if
$c\in(a]$, 
then an application of the definition of $\phi$ 
to $h=\groupone$, the neutral element of $G$,
yields that $\phi(c)=\psi(c)$.  

To prove that $\phi$ is a label-preserving $G$-homomorphism
$\al P\to\al Q$,
we need to check that conditions (H0)--(H4) hold for $\phi$.
(H0) is obvious from the definition of $\phi$, and
(H1) holds, because $\psi$ as well as the actions of 
$h\in G$ map leaves to leaves. 
To show that (H2) holds let $c\in P\setminus\{1_P\}$.
As $c\not=1_P$, we have that $c\in(ha]$ for some $h\in G$.
Assume first that $c=ha$.
Since $\psi$ is a $G_a$-homomorphism 
$(a]_{\al P}\to(b]_{\al Q}$
and $a=1_{(a]}$, $b=1_{(b]}$, therefore we get that $\psi(a)=b$,
Hence 
if $c=ha$, then 
$\phi(c)=h\psi(h^{-1}c)=
h\psi(a)=hb\in\al Q_\ttop$, so $\phi(c)^*=1_Q=\phi(1_P)=\phi(c^*)$.
Now assume that $c\in(ha]$ but $c\not=ha$.
Then $h^{-1}c\in(a]$ and $h^{-1}c\not=a$.
Hence $(h^{-1}c)^*$ in $(a]_{\al P}$ is the same as 
$(h^{-1}c)^*$ in $\al P$, which is equal to $h^{-1}c^*$.
Using this (in the fourth equality below) we get that
\[
\phi(c)^*
=\bigl(h\psi(h^{-1}c)\bigr)^*
=h\bigl(\psi(h^{-1}c)\bigr)^*
=h\psi\bigl((h^{-1}c)^*\bigr)
=h\psi(h^{-1}c^*)
=\phi(c^*),
\]
which completes the proof of (H2).
Next we prove (H3).
Every $g\in G$ acts by tree automorphisms, therefore
$g\cdot 1_P=1_P$ and $g\cdot 1_Q=1_Q$, whence
$\phi(g\cdot 1_P)=\phi(1_P)=1_Q=g\cdot 1_Q=g\phi(1_P)$.
To prove (H3) for elements $c\not=1_P$
let $c\in(ha]$ and $g\in G$.
Then $gc\in(gha]$, hence
$\phi(gc)=gh\psi((gh)^{-1}gc)=gh\psi(h^{-1}c)=g\phi(c)$.
Thus (H3) holds for $\phi$.
Finally, we verify (H4).
Let $c\in\al P_\leaf$. Then $c$ is a leaf in $(ha]_{\al P}$ 
for some $h\in G$,
and hence
$h^{-1}c$ is a leaf in $(a]_{\al P}$.
Since $\psi\colon(a]_{\al P}\to(b]_{\al Q}$ is a label-preserving
$G_a$-homomorphism, 
$\psi(h^{-1}c)$ is a leaf in $(b]_{\al Q}$.
Using the facts that $\ell_P$, $\ell_Q$ are labelings of $\al P$ and $\al Q$,
and their restrictions are the labelings of
$(a]_{\al P}$ and $(b]_{\al Q}$, 
we get that
\begin{multline*}
\ell_Q(\phi(c))
=\ell_Q\bigl(h\psi(h^{-1}c)\bigr)
=h\ell_Q\bigl(\psi(h^{-1}c)\bigr) \\
=h\ell_P(h^{-1}c)
=hh^{-1}\ell_P(c)
=\ell_P(c),
\end{multline*}
proving (H4).
This finishes the proof of statement (1) of the lemma.

(2) We return to the general case; that is, 
$\{a_i:1\le i\le t\}$ is a transversal for the $G$-orbits of
$\al P_\ttop$,  $\{b_i:1\le i\le t\}$ is a subset of 
$\al Q_\ttop$ such that $G_{a_i}=G_{b_i}$ 
for each $i$, $\{\psi_i:1\le i\le t\}$ is a family of
label-preserving $G_{a_i}$-homomorphisms
$\psi_i\colon (a_i]_{\al P}\to(b_i]_{\al Q}$, and
$\phi$ is the unique extension of all $\psi_i$'s
to a label-preserving homomorphism $\phi\colon\al P\to\al Q$
constructed in part (1).
The two-step construction of $\phi$ described in (I)${}_i$ and (II)
above shows that
\[
\phi(c)=\begin{cases}
        1_Q & \text{if $c=1_P$},\\
        h\psi_i(h^{-1}c) & \text{if $c\in(ha_i]$ for some 
                                $1\le i\le t$ and some $h\in G$.}
        \end{cases}
\]
We claim that an element $u$ of $\al Q$ is in the range of 
$\phi$ if and only if either $u=1_Q$ or $u\in h\cdot\range{\psi_i}$
for some $i$ and some $h\in G$.
The necessity of this condition is clear from the description of $\phi$
above. For the sufficiency, let $u=1_Q$ or $u\in h\cdot\range{\psi_i}$.
In the first case, clearly, $u$ is in the range of $\phi$.
In the second case $u=h\psi_i(v)$ for some $v\in(a_i]$,
so for $c=hv$ we have $c\in(ha_i]$ and 
$u=h\psi_i(v)=h\psi_i(h^{-1}c)=\phi(c)$. Thus $u$ is in the range of
$\phi$, as claimed.
This proves that
\[
\range{\phi}=\{1_Q\}\cup\bigcup(G\cdot\range{\psi_j}:1\le j\le t).
\]
Since $\psi_j$ maps $a_j=1_{(a_j]}$ to $1_{(b_j]}=b_j$, we have
$b_j\in\range{\psi_j}\subseteq(b_j]$. It follows that
$\al Q_\ttop\cap G\cdot\range{\psi_j}=Gb_j$ holds for all $j$.
Thus
\[
\al Q_\ttop\cap\range{\phi}
=\bigcup(\al Q_\ttop\cap G\cdot\range{\psi_j}:1\le j\le t)
=\bigcup(Gb_j:1\le j\le t).
\]
For any two distinct elements $b,b'\in\al Q_\ttop$, the subsets
$(b]$ and $(b']$ are disjoint by Lemma~\ref{basic_tree}~(1). Therefore 
$(b_i]$ is disjoint from $h\cdot\range{\psi_j}\bigl(\subseteq(hb_j]\bigr)$
unless $b_i=hb_j$, and hence
$h\cdot\range{\psi_j}\subseteq(b_i]$.
Thus, for each $i$ ($1\le i\le t$),
\begin{align*}
(b_i]\cap\range{\phi}
&{}=(b_i]\cap \bigcup(G\cdot\range{\psi_j}:1\le j\le t)\\
&{}=\bigcup(h\cdot\range{\psi_j}:1\le j\le t,\ h\in G,\ hb_j=b_i).
\end{align*}

Hence, for $\phi$ to map onto $Q$, it is necessary that
$\al Q_\ttop=\bigcup(Gb_j:1\le j\le t)$ and
$(b_i]=\bigcup(h\cdot\range{\psi_j}:1\le j\le t,\ h\in G,\ hb_j=b_i)$
for each $i$ ($1\le i\le t$).
This shows that the conditions in (2) are necessary.
Conversely, assume that $\phi$ satisfies these conditions.
The second one of these conditions implies that $(b_i]\subseteq\range{\phi}$
for all $i$ $(1\le i\le t)$.
Since $\phi$ is a $G$-homomorphism $\al P\to\al Q$, its range
is closed under the actions of all $g\in G$.
Combining this with the condition $\al Q_\ttop=\bigcup(Gb_j:1\le j\le t)$
we obtain that
\begin{align*}
Q\setminus\{1_Q\}
&{}=\bigcup\bigl((b]:b\in\al Q_\ttop\bigr)\\
&{}=\bigcup\bigl((gb_i]:1\le i\le t,\ g\in G\bigr)\\
&{}=\bigcup\bigl(G\cdot(b_i]:1\le i\le t\bigr)
\subseteq \range{\phi}.
\end{align*}
Since $1_Q\in\range{\phi}$, we get that $\phi$ is surjective.
This proves statement (2).

(2)$'$
Now assume that $\{b_i:1\le i\le t\}$
is a transversal for the $G$-orbits of $\al Q_\ttop$.
Then $hb_j=b_i$ holds for some $1\le i,j\le t$ and $h\in G$
only if $i=j$ and $h\in G_{b_i}$.
Therefore 
\[
\bigcup(h\cdot\range{\psi_j}: 1\le j\le t,\ h\in G,\ hb_j=b_i)
=\bigcup(h\cdot\range{\psi_i}: h\in G_{b_i})
=\range{\psi_i}.
\]
Hence the criterion in (2) implies that in this special case
$\phi$ is onto
if and only if all $\psi_i$ are onto.

(3) To prove the necessity of the conditions in (3)
suppose that $\phi$ is bijective. 
First we show that $\{b_i:1\le i\le t\}$
is a transversal for the $G$-orbits of $\al Q_\ttop$.
Since $\phi$ is onto, we get from part (2) of the lemma 
that every $G$-orbit of $\al Q_\ttop$ is represented
by at least one element in $\{b_i:1\le i\le t\}$.
If $\{b_i:1\le i\le t\}$ was not a transversal,
there would exist $1\le j<l\le t$ such that
$b_j=hb_l$ for some $h\in G$.
Hence $\phi(a_j)=b_j=hb_l=\phi(ha_l)$, but $a_j\not=ha_l$
as $\{a_i:1\le i\le t\}$ is a transversal for the $G$-orbits of
$\al P_\ttop$.
This shows that $\{b_i:1\le i\le t\}$
is a transversal for the $G$-orbits of $\al Q_\ttop$.
Since $\phi$ is onto, it follows from part (2)$'$ of the lemma
that each $\psi_i$ is onto; 
$\psi_i$ is also one-to-one, since
$\phi$ extends $\psi_i$ and $\phi$ is one-to-one.
This proves that the conditions in (3) are indeed necessary
for $\phi$ to be bijective.

Conversely, if $\phi$ satisfies the conditions in (3), then
it is clearly onto by the criterion in part (2)$'$.
To verify that $\phi$ is one-to-one, let $c,c'$ be elements in
$\al P$ such that $\phi(c)=\phi(c')$.
It is clear from the description of $\phi$ that
the only element whose $\phi$-image is $1_Q$ is $1_P$.
Therefore if $1_P\in\{c,c'\}$, say 
$c=1_P$, then $\phi(c')=\phi(c)=1_Q$, so $c'=1_P$ and hence $c=c'$.
Assume from now on that $c,c'\not=1_P$.
Then $c\in (ha_i]$ and $c'\in(h'a_{i'}]$ for some $i,i'$
($1\le i,i'\le t$) and $h,h'\in G$.
Since $\phi(c)=h\psi_i(h^{-1}c)\in h\cdot(b_i]=(hb_i]$ and similarly
$\phi(c')\in (h'b_{i'}]$, the assumption
that $\phi(c)=\phi(c')$, combined with Lemma~\ref{basic_tree}~(1),
implies that $hb_i=h'b_{i'}$.
Our assumption that $\{b_i:1\le i\le t\}$
is a transversal for the $G$-orbits of $\al Q_\ttop$
forces that $i=i'$ and $hG_{b_i}=h'G_{b_i}$.
Throughout the lemma we assume $G_{b_i}=G_{a_i}$ for each $i$,
therefore $ha_i=h'a_i$ and both of $c,c'$ belong to
$(ha_i]$.
Thus the equality $\phi(c)=\phi(c')$ can be rewritten as
$h\psi_i(h^{-1}c)=h\psi_i(h^{-1}c')$.
Hence $\psi_i(h^{-1}c)=\psi_i(h^{-1}c')$, and since $\psi_i$
is bijective, $h^{-1}c=h^{-1}c'$, which implies that $c=c'$.
This proves the sufficiency of the conditions in (3), and completes
the proof of the lemma.
\end{proof}

\begin{corollary}
\label{reduced-trees}
If $\al Q=(Q;{}^*,1_Q,G;\ell)$ is
an $S$-labeled $G$-tree that is a core, then
the following hold for arbitrary elements $a,b$ of $\al Q_\ttop$: 
\begin{enumerate}[\indent\rm(i)]
\item
$(a]_{\al Q}$, as an $S$-labeled $G_a$-tree, is a core.
\item
If $G_a=G_b$ and $(a]_{\al Q}\cong(b]_{\al Q}$ as $S$-labeled
$G_a$-trees, then $Ga=Gb$.
\end{enumerate} 
\end{corollary}

\begin{proof}
(i) Suppose that $(a]_{\al Q}$, as an $S$-labeled $G_a$-tree, is not a core.
Then there exists a label-preserving $G_a$-homomorphism
$\psi\colon(a]_{\al Q}\to(a]_{\al Q}$ that is not surjective.
Let $a_1=a,a_2,\ldots,a_t$ be a transversal for the $G$-orbits of
$\al Q_\ttop$, let $\psi_1=\psi$, and  for $2\le i\le t$
let $\psi_i$ be the identity isomorphism $(a_i]_{\al Q}\to(a_i]_{\al Q}$.
Applying Lemma~\ref{homs} we get that the family $\{\psi_i:1\le i\le t\}$
can be extended to a unique label-preserving 
$G$-homomorphism $\phi\colon\al Q\to\al Q$.
Moreover, since $\psi_i(a_i)=a_i$ for all $i$, part (2)$'$ of the lemma
applies and yields that $\phi$ is not surjective.
Therefore $\al Q$ is not a core.

(ii)
Assume that $a,b\in\al Q_\ttop$ are in different $G$-orbits
such that $G_a=G_b$ and $(a]_{\al Q}\cong(b]_{\al Q}$ as $S$-labeled
$G_a$-trees. 
Let $\psi$ be a label-preserving 
$G_a$-isomorphism $(a]_{\al Q}\to(b]_{\al Q}$. 
We want to show that $\al Q$ is not a core.
Let $a_1=a,a_2=b,a_3\ldots,a_t$ be a transversal for the $G$-orbits of
$\al Q_\ttop$, let $\psi_1=\psi$, and  for $2\le i\le t$
let $\psi_i$ be the identity isomorphism $(a_i]_{\al Q}\to(a_i]_{\al Q}$.
Applying Lemma~\ref{homs} we get that the family $\{\psi_i:1\le i\le t\}$
can be extended to a unique label-preserving 
$G$-homomorphism $\phi\colon\al Q\to\al Q$.
Since $\{\psi_i(a_i):1\le i\le t\}=\{b,a_3,\ldots,a_t\}$
does not represent all $G$-orbits of $\al Q_\ttop$, it follows from
part (2) of the lemma that $\phi$ is not surjective.
Therefore $\al Q$ is not a core, as claimed.
\end{proof}

\begin{lemma}
\label{bound-on-reduced}
For every group $G$ and $G$-set $(S;G)$ of labels, and for
each natural number $k$ there exists an integer $n_k=n_k(G,S)$ 
depending only on $k$, $G$, and $(S;G)$ 
such that there are at most $n_k$
nonisomorphic 
$S$-labeled $G$-trees
of uniform depth $k$ that are cores.
\end{lemma}

\begin{proof}
Let $\al Q=(Q;{}^*,1_Q,G;\ell)$ 
be an $S$-labeled $G$-tree
of uniform depth $k$ that is a core.
We want to find an upper bound on the number of possibilities for $\al Q$,
up to isomorphism. 

If $k=0$, then $Q=\{1_Q\}$, and the unique element (which is a leaf)
can be labeled in $|S|$ different ways. 
Therefore in this case
there are $n_0=|S|$ possibilities for $\al Q$, up to isomorphism.

Now let $k\ge1$, and assume that $n_{k-1}=n_{k-1}(G,S)$ has been found for all
$G$ and $(S;G)$. 
Choose a transversal $\{a_i:1\le i\le t\}$ for the $G$-orbits of 
$\al Q_\ttop$, and for each transversal element $a_i$ 
consider the pair $\bigl(G_{a_i},\iso{(a_i]_{\al Q}}\bigr)$
where $\iso{(a_i]_{\al Q}}$ denotes the isomorphism type of
$(a_i]_{\al Q}$, as an $S$-labeled $G_{a_i}$-tree.
Each $(a_i]_{\al Q}$ has uniform depth $k-1$, since $\al Q$
has uniform depth $k$.
Since $\al Q$ is a core, we get from
Corollary~\ref{reduced-trees} that
the $S$-labeled $G_{a_i}$-tree $(a_i]_{\al Q}$ is 
a core for every
$i$ ($1\le i\le t$). 
Moreover,
if $1\le i<j\le t$, then
$G_{a_i}\not=G_{a_j}$ or $(a_i]_\al Q\not\cong(a_j]_\al Q$.
Thus the pairs $\bigl(G_{a_i},\iso{(a_i]_{\al Q}}\bigr)$ ($1\le i\le t$)
are pairwise distinct.
By part (3) of Lemma~\ref{homs} the set
\[
\bigl\{\bigl(G_{a_i},\iso{(a_i]_{\al Q}}\bigr):1\le i\le t\bigr\}
\]
determines $\al Q$, up to isomorphism.
Therefore the number of possible isomorphism types for $\al Q$
is at most
\[
n_k(G,S)=2^s
\quad\text{where}\quad
s=\sum\bigl(n_{k-1}(H,S):H \text{ is a subgroup of } G\bigr).
\] 
This completes the proof.
\end{proof}

We return to the proof of Theorem~\ref{chains-of-eqrels}. 
As before, let $A$ be a finite set, and let
$E=\{\rho_i:1\le i\le r\}$ be a chain of equivalence relations, say, 
$\rho_0:={\zero}_A<\rho_1<\cdots<\rho_{r-1}<\rho_{r}<{\one}_A=:\rho_{r+1}$, 
and let $\Gamma:=\Aut{E}$.
Earlier in this section we defined for each operation $f$ on $A$
an ${\mathbb S}/{\sim}$-labeled 
$\Gamma$-tree $\al P_f(E)$ of uniform depth $r+1$.
By Lemma~\ref{reduction} $\al P_f(E)$ has a core
$\widehat{\al P}_f=
(\widehat P_f,\widehat\le,\widehat\ell_f)$ 
that is an ${\mathbb S}/{\sim}$-labeled $\Gamma$-subtree 
of $\al P_f(E)$.
Thus $\widehat{\al P}_f$ is of uniform depth $r+1$, and 
there exists a label-preserving $\Gamma$-homomorphism
$\phi_f\colon \al P_f(E)\to\widehat{\al P}_f$.
Moreover,
$\widehat{\al P}_f$ is uniquely determined up to isomorphism.
We will refer to $\widehat{\al P}_f$ as 
{\it the core of the ${\mathbb S}/{\sim}$-labeled 
$\Gamma$-tree associated to $f$}.

The following statement is an easy consequence of 
Lemma~\ref{equiv-lessthan}.

\begin{corollary}
\label{equiv-equiv}
Let $E$ be a chain of equivalence relations on a finite set $A$, 
let $\Gamma=\Aut{E}$,
and
let $\cl C=\Pol\bigl(E,\Aut{E},\cl{P}^+(A)\bigr)$.
For arbitrary operations $f,g$ on $A$, 
\begin{enumerate}[\indent\rm(1)]
\item
$f \subf[\cl{C}] g$ if and only if
there exists a label-increasing homomorphism 
$\widehat{\al P}_f\nearrow\widehat{\al P}_g$ 
between the cores of the
${\mathbb S}/{\sim}$-labeled $\Gamma$-trees associated to $f$ and $g$. 
\item
$f \fequiv[\cl{C}] g$ if and only if
there exist label-increasing homomorphisms 
$\widehat{\al P}_f\nearrow\widehat{\al P}_g$ and
$\widehat{\al P}_g\nearrow\widehat{\al P}_f$
between the cores of the
${\mathbb S}/{\sim}$-labeled $\Gamma$-trees associated to $f$ and $g$. 
\end{enumerate}
\end{corollary}

\begin{proof}
Let $f,g$ be arbitrary operations on $A$.
By construction, there exist 
label-preserving $\Gamma$-homomorphisms
$\phi_f\colon \al P_f(E)\to\widehat{\al P}_f$ and
$\phi_g\colon \al P_g(E)\to\widehat{\al P}_g$.
Since $\widehat{\al P}_f$ is an 
${\mathbb S}/{\sim}$-labeled $\Gamma$-subtree of 
$\al P_f(E)$, and $\widehat{\al P}_g$ is an 
${\mathbb S}/{\sim}$-labeled $\Gamma$-subtree of 
$\al P_g(E)$, the identity mappings 
$\iota_f\colon \widehat{\al P}_f \to \al P_f(E)$ and
$\iota_g\colon \widehat{\al P}_g \to \al P_g(E)$ 
are also label-preserving $\Gamma$-homomorphisms.

By Lemma~\ref{equiv-lessthan}, $f \subf[\cl{C}] g$
if and only if 
there exists a label-increasing $\Gamma$-ho\-mo\-mor\-phism 
$\al P_f(E)\nearrow\al P_g(E)$. 
We claim that
there exists a label-increasing $\Gamma$-ho\-mo\-mor\-phism 
$\al P_f(E)\nearrow\al P_g(E)$
if and only if there exists a label-increasing $\Gamma$-ho\-mo\-mor\-phism
$\widehat{\al P}_f\nearrow\widehat{\al P}_g$.
Indeed, 
if
$\psi\colon\al P_f(E)\nearrow\al P_g(E)$, then
$\phi_g\circ\psi\circ\iota_f\colon\widehat{\al P}_f\nearrow\widehat{\al P}_g$,
and conversely, 
if
$\psi'\colon\widehat{\al P}_f\nearrow\widehat{\al P}_g$,
then 
$\iota_g\circ\psi'\circ\phi_f\colon\al P_f(E)\nearrow\al P_g(E)$,
since the composition of label-increasing (or label-preserving)
$\Gamma$-ho\-mo\-mor\-phisms is a label-increasing $\Gamma$-homomorphism.
This proves (1).

The relation $\fequiv[\cl{C}]$ is the intersection of 
$\subf[\cl{C}]$ with its converse, therefore (2) is an immediate
consequence of (1).
\end{proof}

\begin{proof}[Proof of Theorem~\ref{chains-of-eqrels}]
Let $E$ be a chain of equivalence relations on a finite set $A$,
let $\Gamma=\Aut{E}$, and let
$\cl C=\Pol\bigl(E,\Aut{E},\cl{P}^+(A)\bigr)$.
Corollary~\ref{equiv-equiv} implies that $f \fequiv[\cl C] g$ 
holds for two operations $f$ and $g$ on $A$ if and only if 
for the cores $\widehat{\al P}_f$ and $\widehat{\al P}_g$
of the associated ${\mathbb S}/{\sim}$-labeled $\Gamma$-trees 
there exist 
label-increasing $\Gamma$-homomorphisms
$\widehat{\al P}_f\nearrow \widehat{\al P}_g$
and
$\widehat{\al P}_g\nearrow \widehat{\al P}_f$.
In particular, it follows that $f \fequiv[\cl C] g$ if
$\widehat{\al P}_f\cong \widehat{\al P}_g$.
By Lemma~\ref{bound-on-reduced} there exist only finitely many 
isomorphism classes of trees $\widehat{\al P}_f$ as $f$ runs over 
all operations on $A$.
Therefore there exist only finitely many
$\fequiv[\cl C]$-classes.
\end{proof}

\section{Central relations}
\label{sec:central}

Let $A$ be a $k$-element finite set, $k\ge3$.
In this section, our aim is to find all maximal clones $\Pol \rho$ on 
$A$ 
that are determined by central relations 
and are members of $\FF_A$. 
Note that the arity $r$ of a central relation on $A$
satisfies $1\le r\le k-1$, and the case of unary central relations
is settled in Corollary~\ref{discr-max}.
Therefore in this section we consider only central relations
of arity $r\ge2$.

We will show that if $\rho$ has arity $r\le k-2$, then
$\Pol\rho\notin\FF_A$ (Theorem~\ref{thm-central-r}), while if 
$\rho$ has arity $r=k-1$, then $\Pol\rho\in\FF_A$ 
(Theorem~\ref{thm-centralk-1}).
Note that for each element $c\in A$ there is a unique central relation
$\sigma_c$ of arity $k-1$ with central element $c$, namely
\begin{align*}
\sigma_c=
\{(a_1,\ldots,a_{k-1})\in A^{k-1}:{}&
\text{$a_i=a_j$ for some $1\le i<j\le k-1$, or}\\
{}&
\text{$a_i=c$ for some $1\le i\le k-1$}\}.
\end{align*}
Therefore all central relations of arity $k - 1$ are of the form
$\sigma_c$ for some $c \in A$.
In Theorem~\ref{thm-centralk-1} we will, in fact, 
prove that $\Pol(\sigma_c,\{c\})\in\FF_A$ for all $c\in A$, which implies by
Proposition~\ref{basic_props2}(ii) that all maximal clones 
$\Pol\sigma_c$ ($c\in A$) also belong to $\FF_A$.

We start by stating Jablonski\u\i's Lemma which we will need
in the proof of Theorem~\ref{thm-centralk-1}.

\begin{lemma}
\label{lm-jablonskii}
{\rm(Jablonski\u{\i}~\cite{Jablonskii})}
Let $f$ be an $n$-ary operation on a finite set $A$ such that
$f$ depends on at least two of its variables.
If the range $\range f$ of $f$ has $r\ge3$ elements, then
there exist $D_1,\ldots,D_n\subseteq A$ such that 
$|D_i|<r$ for all $1\le i\le n$ and $f[D_1\times\cdots\times D_n]=\range f$.
\end{lemma}

\begin{theorem}
\label{thm-centralk-1}
If $\sigma_c$ is the $(k-1)$-ary central relation with central element $c$
on a $k$-element set $A$ $(k \geq 3)$, then $\Pol(\sigma_c, \{c\})\in\FF_A$.
\end{theorem}

\begin{proof}
We may assume without loss of generality that $c=0$, and we will write
$\sigma$ for $\sigma_0$. So, let $\cl{C} = \Pol(\sigma, \{0\})$. 
Since $\sigma$ contains all $(k-1)$-tuples whose coordinates
are not pairwise distinct or include $0$, 
it follows that $\sigma$ is preserved by
\begin{itemize}
\item
all operations $f \colon A^n \to A$ with $\card{\range f} \leq k - 2$, and also
\item
all operations $f \colon A^n \to A$ with $\card{\range f} = k - 1$ and 
$0 \in \range f$.
\end{itemize}

To prove that $\cl{C}\in\FF_A$ we partition $\cl{O}_A$ into two subsets, $O_0$
and $O_1=\cl{O}_A\setminus O_0$, as follows: an operation $f$ belongs to $O_0$
if and only if its domain $A^n$ where $n$ is the arity of $f$ contains
a subset $C_1 \times \dotsb \times C_n$ such that
$0\in C_i\not=A$ for all $i$ ($1\le i\le n$) and 
$f[C_1 \times \dotsb \times C_n] = \range f$.
First we will show that all nonsurjective operations $f$
on $A$ belong to $O_0$.
Indeed, assume first that $f$ is nonsurjective and essentially unary, say
it depends on its first variable only. 
Then there exists a nonsurjective
unary operation $f_1$ such that
$f(\vect{x})=f_1(x_1)$ for all $\vect{x}=(x_1,\ldots,x_n)\in A^n$.
Therefore $f_1(a)=f_1(b)$ for some distinct $a,b\in A$ such that $a\not=0$.
Hence the choice $C_1=A\setminus\{a\}$, $C_2=\cdots=C_n=\{0\}$ shows that
$f\in O_0$.
Now assume that $f$ is nonsurjective and depends on at least two of its
variables. If $|\range{f}|=2$, then there exists $\vect{a}=(a_1,\ldots,a_n)$
distinct from $\bar{0}$
such that $f(\vect{a})\not= f(\bar{0})$. Hence the choice $C_i=\{a_i,0\}$
($1\le i\le n$) shows that $f\in O_0$ ($C_i\not=A$, since $k=|A|\ge3$).
Finally, if $|\range{f}|>2$ but $f$ is nonsurjective, then
by Jablonski\u{\i}'s Lemma (Lemma~\ref{lm-jablonskii}) 
there exist $(k-2)$-element subsets
$D_1,\ldots,D_n$ of $A$ such that 
$f[D_1 \times \dotsb \times D_n] = \range f$. 
Hence we can choose $C_i=D_i\cup\{0\}$ ($1\le i\le n$) to show that
$f\in O_0$.

\begin{quote}
\textit{Claim 1.}
If $\range f = \range g$, $f(\bar{0}) = g(\bar{0})$, and $f,g\in O_0$, 
then $f \fequiv_{\cl{C}} g$.
\end{quote}

\textit{Proof of Claim 1.}
Let $f$ be $n$-ary and $g$ be $m$-ary.
Using the assumption $f\in O_0$, we fix 
sets $C_i \subset A$ ($1 \leq i \leq n$)
such that 
$0 \in C_i$ and $f[C_1 \times \dotsb \times C_n] = \range f$.
Furthermore, we
choose a transversal $\{\vect{b}_1, \dotsc, \vect{b}_r\} \subseteq C_1 \times \dotsb \times C_n$ of $\ker f$ where $\vect{b}_1 = \bar{0}$. 
Now we define a function $\vect{h} \colon A^m \to A^n$ 
as follows: for each $\vect{a}\in A^m$ we have 
$g(\vect{a})\in\range{g}=\range{f}$, therefore $g(\vect{a})=f(\vect{b}_j)$
for a unique $j$; we let $\vect{h}(\vect{a}) = \vect{b}_j$.
It is clear from this definition that $g = f \circ \vect{h}$. 
By assumption, $g(\bar{0}) = f(\bar{0})$ and $\vect{b}_1=\bar{0}$,
therefore $g(\bar{0})=f(\vect{b}_1)$.
Hence $\vect{h}(\bar{0})=\vect{b}_1=\bar{0}$, which implies that
$\vect{h}$ preserves $\{0\}$. 
Since $\range{\vect{h}}=\{\bar{0},\vect{b}_2,\ldots,\vect{b}_r\}
\subseteq C_1\times\cdots\times C_n$,
the range of each component $h_i$ of $\vect{h}$ satisfies
$0 \in \range h_i\subseteq C_i\not=A$. 
As was observed at the beginning of the proof,
this implies that each $h_i$ preserves $\sigma$.
Thus $\vect{h}$ preserves $\sigma$. 
This proves that $\vect{h}\in(\cl{C}^{(m)})^n$ and hence
$g \subf[\cl{C}] f$. 
A similar argument shows that $f \subf[\cl{C}] g$.
\quad $\diamond$

\begin{quote}
\textit{Claim 2.}
If $f(\bar{0}) = g(\bar{0})$ and $f,g\notin O_0$, 
then $f \fequiv_{\cl{C}} g$.
\end{quote}

\textit{Proof of Claim 2.}
Again, let $f$ be $n$-ary and $g$ be $m$-ary.
We proved earlier that all nonsurjective operations belong to $O_0$, therefore
$f$ and $g$ are necessarily surjective. 
Let $\{\vect{b}_1, \dotsc, \vect{b}_k\}$ be a transversal of $\ker f$ 
where $\vect{b}_1 = \bar{0}$. 
As before, we define $\vect{h} \colon A^m \to A^n$ such that
for each $\vect{a}\in A^m$, $\vect{h}(\vect{a}) = \vect{b}_j$ 
for the unique $j$ such that $g(\vect{a}) = f(\vect{b}_j)$. 
We get, as before, that $g = f \circ \vect{h}$ and that
$\vect{h}$ preserves $\{0\}$. 
It remains to show that $\vect{h}$ preserves $\sigma$.
Let $\vect{a}_1,\ldots,\vect{a}_{k-1}\in A^m$
be $m$-tuples such that 
$\bigl(\vect{h}(\vect{a}_1),\ldots,\vect{h}(\vect{a}_{k-1})\bigr)
\notin\sigma^n$.
Since the range of $\vect{h}$ is $\{\bar{0},\vect{b}_2,\ldots,\vect{b}_k\}$
and $\sigma^n$ contains every $(k-1)$-tuple which has repeated coordinates
or has $\bar{0}$ as one of its coordinates, we get that
$\{\vect{h}(\vect{a}_1),\ldots,\vect{h}(\vect{a}_{k-1})\}=
\{\vect{b}_2,\ldots,\vect{b}_k\}$.
Hence 
\begin{align*}
\{g(\bar{0}),g(\vect{a}_1),\ldots,g(\vect{a}_{k-1})\}
&{}=\{(f\circ\vect{h})(\bar{0}),(f\circ\vect{h})(\vect{a}_1),\ldots,
                           (f\circ\vect{h})(\vect{a}_{k-1})\}\\
&{}=
\{f(\bar{0}),f(\vect{b}_2),\ldots,f(\vect{b}_k)\}=A,
\end{align*}
implying that $g[D_1\times\cdots\times D_m]=A=\range{g}$ where
$D_i$ is the set of $i$-th coordinates of 
$\bar{0},\vect{a}_1,\ldots,\vect{a}_{k-1}$ for each $i$ ($1\le i\le m$).
Since by assumption $g\notin O_0$, we have $D_i=A$ for at least one $i$.
Thus $(\vect{a}_1,\ldots,\vect{a}_{k-1})\notin\sigma^m$, proving that
$\vect{h}$ preserves $\sigma$.
This completes the proof that $\vect{h}\in(\cl{C}^{(m)})^n$ and hence
$g \subf[\cl{C}] f$. 
A similar argument shows that $f \subf[\cl{C}] g$.
\quad $\diamond$

Now consider the mapping
\[
\Phi\colon
\cl{O}_A
\to\mathcal{P}^+(A)\times A\times\{0,1\},\quad
f\mapsto (\range{f},f(\bar{0}),i_f)
\]
where $i_f=0$ if $f\in O_0$ and $i_f=1$ if $f\in O_1$.
Claims~1 and 2 show that we have $f \fequiv[\cl{C}] g$
whenever $\Phi(f)=\Phi(g)$.
Therefore the number of $\fequiv[\cl{C}]$-classes does not exceed the
number of kernel classes of $\Phi$. The number of kernel classes of
$\Phi$ is finite, since the codomain $\mathcal{P}^+(A)\times A\times\{0,1\}$
of $\Phi$ is finite. Hence the number of $\fequiv[\cl{C}]$-classes is
also finite, which proves that $\cl{C}\in\FF_A$.
\end{proof}

\begin{theorem}
\label{thm-central-r}
If $\rho$ is an $r$-ary central relation on a $k$-element set $A$ such that $2 \leq r \leq k-2$ $(k \geq 4)$, then $\Pol \rho\notin\FF_A$.
\end{theorem}

\begin{proof}
Let $\cl C=\Pol\rho$.
We may assume without loss of generality that $A$ is the set
$\{0,\ldots,k-1\}$,
$0$ is a central element of $\rho$, and $(1, 2, \dotsc, r) \notin \rho$. 
For each $n \geq 2$ and 
$1\le i\le n$ we define $2n$-tuples $\vect{a}_i^n$, $\vect{b}_i^n$, and 
$\vect{c}_i^n$ as follows: 
$\vect{a}_i^n=(0,0,\ldots,0,0,1,1,0,0,\ldots,0,0)$ 
with the two $1$'s occurring in the $(2i-1)$-th and $2i$-th coordinates;
$\vect{b}_i^n=(1,2,\ldots,1,2,0,0,1,2,\ldots,1,2)$ 
with the two $0$'s occurring in the $(2i-1)$-th and $2i$-th coordinates;
$\vect{c}_i^n=(0,0,\ldots,0,0,2,1,0,0,\ldots,0,0)$ 
with the $2$ and $1$ occurring in the $(2i-1)$-th and the $2i$-th coordinates.
Next we define a $2n$-ary operation $f_n \colon A^{2n} \to A$ 
for each $n\ge 2$ as follows:
\[
f_n(\vect{a})=
\begin{cases}
0 & \text{if $\vect{a}=\vect{a}_i^n$ ($1\le i\le n$),}\\
1 & \text{if $\vect{a}=\vect{b}_i^n$ ($1\le i\le n$),}\\
2 & \text{if $\vect{a}=\vect{c}_i^n$ ($1\le i\le n$),}\\
u & \text{if $\vect{a}=\bar u$ ($3\le u\le r$),}\\
r+1 & \text{otherwise.}
\end{cases}
\]

We claim that $f_n \not\fequiv[\cl C] f_m$ whenever $n \neq m$.
Suppose on the contrary that $f_n \fequiv[\cl C] f_m$
for some $m < n$. Then
there exists $\vect{h} \in (\cl C^{(2n)})^{2m}$ such that 
$f_n = f_m \circ \vect{h}$. 
For each element $v$ in the common range $\{0,1,\ldots,r,r+1\}$ of
$f_m$ and $f_n$, $\vect{h}$ maps the inverse image $f_n^{-1}(v)$
of $v$ under $f_n$ into the inverse image $f_m^{-1}(v)$
of $v$ under $f_m$; for, if $\vect{x}\in f_n^{-1}(v)$, then
$v=f_n(\vect{x})=f_m\bigl(\vect{h}(\vect{x})\bigr)$, implying that
$\vect{h}(\vect{x})\in f_m^{-1}(v)$.
Thus, in particular, $\vect{h}(\bar u)=\bar u$ for all $3\le u\le r$, and
\[
\vect{h}(\vect{a}_i^n)\in\{\vect{a}_1^m,\ldots,\vect{a}_m^m\},\quad
\vect{h}(\vect{b}_i^n)\in\{\vect{b}_1^m,\ldots,\vect{b}_m^m\},\quad
\vect{h}(\vect{c}_i^n)\in\{\vect{c}_1^m,\ldots,\vect{c}_m^m\}  
\]
for all $i$ ($1\le i\le n$).
Since $m < n$, there exist $1\le p<q\le n$ and $1\le s\le m$ such that
$\vect{h}(\vect{a}_p^n)=\vect{h}(\vect{a}_q^n)=\vect{a}_s^m$.
We have 
$(\vect{a}_p^n,\vect{b}_p^n,\bar3,\ldots,\bar r)\in\rho^{2n}$, 
as in each coordinate the first or second component of the tuple is $0$.
Therefore, since $\vect{h}$ preserves $\rho$, we get that
\[
\bigl(\vect{a}_s^m,\vect{h}(\vect{b}_p^n),
\bar3,\ldots,\bar r\bigr)
=
\bigl(\vect{h}(\vect{a}_p^n),\vect{h}(\vect{b}_p^n),
\vect{h}(\bar3),\ldots,\vect{h}(\bar r)\bigr)
\in\rho^{2m}.
\]
If $j\not=s$ then 
$(\vect{a}_s^m,\vect{b}_j^m,\bar3,\ldots,\bar r)\notin\rho^{2m}$,
because the $2s$-th coordinate of the tuple is $(1,2,3,\ldots,r)\notin\rho$.
This forces $\vect{h}(\vect{b}_p^n)=\vect{b}_s^m$.
The same argument with $\vect{a}_q^n$ in place of $\vect{a}_p^n$
shows that $\vect{h}(\vect{b}_q^n)=\vect{b}_s^m$. 
Similarly, 
since $(\vect{c}_p^n,\vect{b}_p^n,\bar3,\ldots,\bar r)\in\rho^{2n}$
and $\vect{h}$ preserves $\rho$, we get that
\[
\bigl(\vect{h}(\vect{c}_p^n),\vect{b}_s^m,
\bar3,\ldots,\bar r\bigr)
=
\bigl(\vect{h}(\vect{c}_p^n),\vect{h}(\vect{b}_p^n),
\vect{h}(\bar3),\ldots,\vect{h}(\bar r)\bigr)
\in\rho^{2m}.
\]
Again, if $j\not=s$ then 
$(\vect{c}_j^m,\vect{b}_s^m,\bar3,\ldots,\bar r)\notin \rho^{2m}$,
because the $2j$-th coordinate of the tuple is $(1,2,3,\ldots,r)\notin\rho$.
Thus $\vect{h}(\vect{c}_p^n)=\vect{c}_s^m$.
The same argument with $\vect{c}_q^n$ in place of $\vect{c}_p^n$
yields that $\vect{h}(\vect{c}_q^n)=\vect{c}_s^m$. 
Now we see that
$(\vect{a}_p^n,\vect{c}_q^n,\bar3,\ldots,\bar r)\in\rho^{2n}$,
since in each coordinate the first or second component is $0$, but
\[
\bigl(\vect{h}(\vect{a}_p^n),\vect{h}(\vect{c}_q^n),
\vect{h}(\bar3),\ldots,\vect{h}(\bar r)\bigr)
=
(\vect{a}_s^m,\vect{c}_s^m,
\bar3,\ldots,\bar r)
\notin\rho^{2m},
\]
because the $(2s-1)$-th coordinate is $(1,2,3,\ldots,r)\notin\rho$.
This contradiction shows that 
$f_n \not\fequiv[\cl C] f_m$
if $m<n$, and hence proves that
$\cl C\notin\FF_A$. 
\end{proof}

\section{$h$-regular relations}
\label{sec:h-reg}

Let $A$ be a finite set with $k$ elements ($k\ge3$).
In this section our goal is to find all maximal clones $\Pol \lambda_T$ 
on $A$ determined by $h$-regular relations $\lambda_T$ 
(see Section~\ref{sec:prelim} for the definition)
that are members of $\FF_A$. 
Recall that the arity of an $h$-regular relation $\lambda_T$ is $h$
with $3\le h\le k$.
The only $h$-regular relation with $h=k$ is $\lambda_T$ where
$T$ is the singleton consisting of the equality relation on $A$, 
and then $\Pol \lambda_T$ is S\l{}upecki's clone on $A$. 

We will show that $\Pol \lambda_T \notin \FF_A$ unless $\Pol \lambda_T$ is S\l{}upecki's clone (Theorem~\ref{thm-hreg}). Moreover,
we will find an interval in the clone lattice that includes 
S{\l}upecki's clone and 
is contained in $\FF_A$ (Theorem~\ref{thm-slup}).

As a preparation for stating the latter result we introduce some notation.
For $2 \leq i \leq k$, $\cl{B}_i$ will denote the subclone of 
$\cl{O}_A$ that consists of all essentially at most unary 
operations and all operations whose range contains at most $i$ elements. 
Thus, $\cl{B}_{k-1}$ is S{\l}upecki's clone and $\cl{B}_{k-2}$
is the clone introduced in Remark~\ref{notconverse}.
For $i=0$, $\cl B_0$ will stand for
the clone of all essentially at most unary operations, and
for $i=1$, $\cl{B}_1$ denotes Burle's clone defined preceding 
Corollary~\ref{burle}.
Furthermore, $T_A$ will denote the full transformation monoid $\cl{O}_A^{(1)}$
on $A$, and $T_A^-$ its submonoid consisting of the identity function and all
nonpermutations. 
For any submonoid $M$ of $T_A$ containing $T_A^-$ and for any $i$ ($1\le i<k$)
we will use $\cl B_i(M)$ to denote the clone that
arises from $\cl{B}_i$ by omitting all operations 
depending on at most one variable which are outside the clone
$\langle M \rangle$. 

It is well known (see~\cite{Slupecki} and~\cite{Burle}) that
the subclones of $\cl{O}_A$ containing $T_A$ are exactly the
clones in the $(k + 1)$-element chain
\[
\langle T_A \rangle = \cl{B}_0 \subset \cl{B}_1 \subset \cl{B}_2 \subset \dots \subset \cl{B}_{k-1} \subset \cl{B}_k = \cl{O}_A,
\]
which is often referred to as the S{\l}upecki--Burle chain.
Szabó (unpublished, \cite{Szendrei}) extended this result and showed that 
the proper subclones of $\cl{O}_A$ containing $T_A^-$ are exactly the
clones $\cl B_i(M)$ where $0\le i< k$ and
$M$ is a submonoid of $T_A$ containing $T_A^-$.

\begin{theorem}
\label{thm-slup}
If $\cl C$ is a clone on a $k$-element set $A$ $(k\ge3)$
such that $T_A^-\subseteq\cl C$, then
$\cl C\in\FF_A$ if and only if
$\cl B_{k-1}(T_A^-)\subseteq\cl C$.
\end{theorem}

\begin{proof}
Let $\cl{N}=\cl B_{k-1}(T_A^-)$, which is the subclone of ${\mathcal O}_A$
that consists of all projections and all nonsurjective operations.
Assume first that $\cl{N}\subseteq\cl C$.
We want to show that $\cl C\in\FF_A$.
By Proposition~\ref{basic_props2}~(ii) 
it suffices to prove that $\cl{N} \in \FF_A$.
We will start with the following claim.

\begin{quote}
\textit{Claim.}
If $f$ and $g$ are operations on $A$ that are
not essentially unary and satisfy
$\range f = \range g$, then 
$f \fequiv[\cl{N}] g$.
\end{quote}

\textit{Proof of Claim.}
Let $\range f = \range g = S$, and let $f$ be $n$-ary and $g$ be $m$-ary.
Since $f$ and $g$ are not essentially unary, $\card{S}\ge2$.
If $\card{S} \geq 3$, then it follows from Jablonski\u{\i}'s Lemma 
(Lemma~\ref{lm-jablonskii})
that there is a transversal $B = \{\vect{b}_1, \dotsc, \vect{b}_{\card{S}}\}$ for $\ker f$ such that $B \subseteq C_1 \times C_2 \times \dots \times C_n$ for some proper subsets $C_i \subset A$. This condition clearly holds also in the case $\card{S} = 2$. 
The assumption $\range{f}=\range{g}$ combined with the choice of $B$ 
ensures that
for each $\vect{a}\in A^m$ there exists a unique $\vect{b}_j\in B$ such that
$g(\vect{a}) = f(\vect{b}_j)$. Therefore we get a well-defined function
$\vect{h} \colon A^m \to A^n$ by setting $\vect{h}(\vect{a}) = \vect{b}_j$ 
whenever $g(\vect{a}) = f(\vect{b}_j)$. It is clear from this definition that 
$g = f \circ \vect{h}$.
Since $B \subseteq C_1 \times C_2 \times \dots \times C_n$, we see that
the components $h_i$ of $\vect{h} = (h_1, \dotsc, h_n)$ are non-surjective, 
and hence they are members of $\cl{N}$.
Thus, $g \subf[\cl{N}] f$.
The same argument with the roles of $f,g$ switched shows also that
$f \subf[\cl{N}] g$.
\quad $\diamond$

It follows from the Claim above that every operation $f$ on $A$ that
is not essentially unary, is $\cl{N}$-equivalent to a
binary operation.
It is easy to see that 
for any clone $\cl{K}$, every
essentially unary operation is $\cl K$-equivalent
to a unary operation.
Therefore we get from Proposition~\ref{basic_props2}~(i) that
$\cl{N} \in \FF_A$. 
Proposition~\ref{basic_props2}~(ii) thus implies that
$\cl C\in\FF_A$ whenever $\cl{N}\subseteq\cl C$.

For the converse assume that $\cl{N}\not\subseteq\cl C$.
Since $T_A^-\subseteq\cl C$, Szab\'o's theorem implies that
$\cl C$ is a subclone of $\cl{B}_{k-2}$.
Therefore, by Proposition~\ref{basic_props2}~(ii), $\cl C\notin\FF_A$ will
follow if we show that $\cl{B}_{k-2}\notin\FF_A$.
For $k=3$ the clone $\cl{B}_{k-2}$ is Burle's clone,
so in this case $\cl{B}_{k-2}\notin\FF_A$ follows from Corollary~\ref{burle}.

{}From now on let $k>3$, and assume without loss of generality that
$A=\{0,1,\ldots,k-1\}$.
For each $n \geq 2$ and $1\le i\le n$ let 
$\vect{u}_i^n$ denote the $n$-tuple whose $i$-th coordinate is
$1$ and all other coordinates are $k-1$.
Now we define an $n$-ary operation $f_n$ on $A$ by
\[
f_n(\vect{a}) =
\begin{cases}
l & \text{if $\vect{a}=\bar{l}$ with $1\le l\le k-2$,}\\
k-1 & \text{if $\vect{a}=\vect{u}_i^n$ ($1\le i\le n$),}\\
0 & \text{otherwise.}
\end{cases}
\]
It is clear that $f_n$ depends on all of its variables,
because it is invariant under all permutations of its variables, and
is not constant. 
Our claim $\cl{B}_{k-2}\notin\FF_A$ will follow, if we
show that $f_m \not\fequiv[\cl{B}_{k-2}] f_n$ whenever $m \neq n$.

Assume that, on the contrary, $f_m \subf[\cl{B}_{k-2}] f_n$ for some $n<m$. Then there exists $\vect{h} = (h_1, \dotsc, h_n) \in (\cl{B}_{k-2}^{(m)})^n$ such that $f_m = f_n \circ \vect{h}$. This implies that
$\vect{h}$ maps each kernel class 
$f_m^{-1}(l)$ ($l\in A$) of 
$f_m$ to the corresponding kernel class 
$f_n^{-1}(l)$ 
of $f_n$. 
Applying this for 
$l \in \{1, \dotsc, k-2\}$ we obtain that 
$\vect{h}(\bar l) = \bar l$, so
the range of each $h_i$ 
contains the elements $1,2,\ldots,k-2$.
Applying the same property of $\vect{h}$ for $l=k-1$ we get that
the range of $\vect{h}$ must also contain an $n$-tuple 
of the form $\vect{u}_s^n$  for some $1\le s\le n$.
For each $i\not=s$ the $i$-th coordinate of $\vect{u}_s^n$ is $k-1$,
therefore for all such $i$ all elements $1,2,\ldots,k-2,k-1$
must be in the range of $h_i$. 
This implies that each $h_i$ ($i \neq s$) is 
essentially unary, because the only members of 
$\cl{B}_{k-2}$ with ranges containing at least $k-1$ 
elements are essentially unary.
On the other hand, it is not the case that $h_s$, too, is 
essentially unary, because $n < m$ and $f_m$ depends on all of 
its variables. 
Thus $h_s$ has essential arity $\ge2$.
The facts established so far about the ranges of the $h_i$'s 
imply that
the range of $h_s$ is $\{1, \dotsc, k - 2\}$. 
Furthermore, the other $h_i$'s ($1\le i\le n$, $i\not=s$) 
are of the form $h_i(\vect{x}) = h'_i(x_{\sigma(i)})$ 
for some $\sigma \colon \{1, \dotsc, n\} \to \{1, \dotsc, m\}$ and some unary 
operations $h'_i$ that fix the elements $1, \dots, k-2$. 
Choose and fix $t\not=s$ ($1\le t\le n$) arbitrarily, and let $p=\sigma(t)$;
hence $h_t(\vect{x})=h'_t(x_p)$.
Since $\vect{h}$ maps
$f_m^{-1}(k-1)$ to $f_n^{-1}(k-1)$,
we get that $\vect{h}(\vect{u}_p^m)=\vect{u}_j^n$ for some $j$ ($1\le j\le n$).
Thus $h_i(\vect{u}_p^m)=1$ if $i=j$ and $h_i(\vect{u}_p^m)=k-1$ if $i\not=j$.
As the range of $h_s$ does not contain $k-1$, it must be the case
that $j=s$.
Hence $h_t(\vect{u}_p^m)=k-1$.
Since the $p$-th coordinate of $\vect{u}_p^m$ is $1$, 
we get that $h_t(\vect{u}_p^m)=h'_t(1)$, and hence 
$h'_t(1)=k-1$. 
This contradicts the fact established earlier that 
$h'_t$ fixes $1$.
The proof of Theorem~\ref{thm-slup} is complete.
\end{proof}

Now we turn to the second main result of this section which shows that
if $\lambda_T$ is an $h$-regular relation of arity $h<k$, then
the maximal clone $\Pol\lambda_T$ is not a member of $\FF_A$.
We will use the notation $\hh = \{1, \dotsc, h\}$ throughout 
the rest of the section. 

The following property of the operations in $\Pol \lambda_T$ will be useful (see, e.g.,~\cite[Lemma 7.3]{RS}).

\begin{lemma}
\label{lm-hreg}
Let $T = \{\theta_1, \dotsc, \theta_r\}$ be an $h$-regular family of
equivalence relations on $A$, 
let $\theta = \bigcap_{i = 1}^r \theta_i$,
and let $g$ be an $m$-ary operation in $\Pol \lambda_T$. If the range of $g$ contains a transversal for the blocks of each $\theta_i$ $(1\le i\le r)$, then
\begin{enumerate}
\item[{\rm(1)}]
for each $i$ $(1 \leq i \leq r)$ there exist $p$ $(1 \leq p \leq m)$ and $q$ $(1 \leq q \leq r)$ such that for all $\vect{a}, \vect{b} \in A^m$,
\[
g(\vect{a}) \, \theta_i \, g(\vect{b})
\qquad \text{whenever} \qquad
a_p \, \theta_q \, b_p;
\]
\end{enumerate}
consequently,
\begin{enumerate}
\item[{\rm(2)}]
$g$ preserves $\theta$, and 
\item[{\rm(3)}]
the operation $g^\theta$ on $A/\theta$ depends on at most $r$ variables.
\end{enumerate}
\end{lemma}

\begin{theorem}
\label{thm-hreg}
If $\lambda_T$ is an $h$-regular relation such that $h<k$, then
$\Pol \lambda_T\notin\FF_A$.
\end{theorem}

\begin{proof}
Let $T = \{\theta_1, \dotsc, \theta_r\}$ be an $h$-regular family of equivalence relations on $A$, let $\theta = \bigcap_{i=1}^r \theta_i$, and let $\cl{C} = \Pol \lambda_T$. First we will consider the case when $r \geq 2$. Since $T$ is $h$-regular, there exists a surjective function $\phi \colon A \to \hh^r$ such that each $\theta_i$ is the inverse image under $\phi$ of the kernel of the $i$-th projection map $\pi_i \colon \hh^r \to \hh$. The diagonal $\Delta = \{\bar{u} : u \in \hh\}$ of $\hh^r$ is a common transversal for the kernel classes of $\pi_i$ for each $i$. Therefore by choosing $t_u \in A$ for each $u \in \hh$ such that $\phi(t_u) = \bar{u}$ we get an $h$-element subset $\{t_u : u \in \hh\}$ of $A$ that is a common transversal for the blocks of each $\theta_i \in T$. In particular, $t_1, \dotsc, t_h$ are pairwise non-equivalent modulo $\theta$. The number of blocks of $\theta$ is $h^r > h+2$ (since $r \geq 2$ and $h \geq 3$). Hence we can extend $t_1, \dotsc, t_h$ to a transversal $o, e, t_1, \dotsc, t_h, t_{h+1}, \dotsc, t_s$ of $\theta$ ($s = h^r - 2$).

For $n \geq 2$ define an $n$-ary operation $f_n$ on $A$ as follows:
\[
f_n(a_1, \dotsc, a_n) =
\begin{cases}
a_1 &
\text{if $a_1 \, \theta \, a_2 \, \theta \dotsb \theta \, a_n$ but $(a_1, e) \notin \theta$,} \\
e &
\text{if $\card{\{i : a_i \, \theta \, e\}} = n - 1$,} \\
o &
\text{otherwise.}
\end{cases}
\]
We will show that if $f_m \subf[\cl{C}] f_n$ then $m \leq nr$. Hence, if $f_m \fequiv[\cl{C}] f_n$ then $n/r \leq m \le nr$. This will imply that no two operations in the infinite sequence $f_{n_\ell}$, $\ell = 1, 2, \dotsc$, with $n_\ell = r^\ell + r^{\ell - 1} + \dotsb + r + 1$ are in the same $\fequiv[\cl{C}]$-class, and therefore $\cl{C}\notin\FF_A$.
 
Assume that $f_m \subf[\cl{C}] f_n$. Hence there exists $\vect{g} = (g_1, \dotsc, g_n) \in (\cl{C}^{(m)})^n$ such that $f_m = f_n \circ \vect{g}$. Substituting $\bar{t}_u = (t_u, \dotsc, t_u) \in A^m$ ($1 \leq u \leq h$) into this equality we get that 
$t_u = f_m(\bar{t}_u) = f_n \bigl( \vect{g}(\bar{t}_u) \bigr)$. Since $t_u \neq o$ and $t_u \neq e$, the definition of $f_n$ implies that 
\[
t_u = g_1(\bar{t}_u) \, \theta \, g_2(\bar{t}_u) \, \theta \dotsb \theta \, g_n(\bar{t}_u).
\]
Thus, it follows from the choice of $t_1, \dotsc, t_h$ that the range of each $g_j$ $(1 \leq j \leq n)$ contains a transversal for the blocks of all equivalence relations $\theta_i$. Therefore by Lemma~\ref{lm-hreg} each $g_j$ preserves $\theta$ and the operation $g_j^\theta$ on $A/\theta$ depends on at most $r$ variables. It is easy to see from their definitions that the operations $f_m$ and $f_n$ also preserve $\theta$. Hence for the operations $f_m^\theta$ and $f_n^\theta$ on $A/\theta$ we get that $f_m^\theta = f_n^\theta \circ \vect{g}^\theta$. Since each $g_j^\theta$ depends on at most $r$ variables, we conclude that $f_m^\theta$ depends on at most $nr$ variables. But the definition of $f_m$ shows that
$f_m^\theta$ depends on all of its $m$ variables, because it is symmetric in all of its variables (that is, every operation obtained from $f_m^\theta$ by permuting variables is $f_m^\theta$ itself) and is not constant. This implies that $m \leq nr$, completing the proof of the theorem in the case when $r \geq 2$.

Now let $r = 1$, that is, $T = \{\theta\}$ where $\theta$ has $h \geq 3$ blocks, but $\theta$ is not the equality relation. We may assume without loss of generality that $A = \{0, 1, \dotsc, k - 1\}$, $\hh = \{1, 2, \dotsc, h\}$ is a transversal for the blocks of $\theta$, and $0 \, \theta \, 1$. For $n \geq 2$ define an $n$-ary operation $f_n$ on $A$ as follows:
\[
f_n(a_1, \dotsc, a_n) =
\begin{cases}
a_1 &
\text{if $a_1 \, \theta \, a_2 \, \theta \dotsb \theta \, a_n$ but $(a_1, 1) \notin \theta$,} \\
1 &
\text{if $\card{\{i : a_i \, \theta \, 1\}} = n - 1$,}\\
0 &
\text{otherwise.}
\end{cases}
\]
We want to show that if $f_m \subf[\cl{C}] f_n$ then $m \leq n$. Hence, if $f_m \fequiv[\cl{C}] f_n$ then $m = n$. This will imply that no two operations in the infinite sequence $f_n$, $n = 2, 3, \dotsc$, are in the same $\fequiv[\cl{C}]$-class, and hence $\cl{C}\notin\FF_A$.

Assume that $f_m \subf[\cl{C}] f_n$. Hence there exists $\vect{g} = (g_1, \dotsc, g_n) \in (\cl{C}^{(m)})^n$ such that $f_m = f_n \circ \vect{g}$. Substituting $\bar{u} = (u, \dotsc, u) \in A^m$ ($2 \leq u \leq h$) into this equality we get that $u = f_m(\bar{u}) = f_n \bigl( \vect{g}(\bar{u}) \bigr)$. Since $u \neq 0$ and $u \neq 1$, the definition of $f_n$ implies that
\[
u = g_1(\bar{u}) \, \theta \, g_2(\bar{u}) \, \theta \dotsb \theta \, g_n(\bar{u}).
\]
Thus, the range of each $g_j$ ($1 \leq j \leq n$) contains an element from every $\theta$-block $2/\theta, \dotsc, h/\theta$ (i.e., from every $\theta$-block other than $1/\theta$).

Now let $\vect{v}_i$ denote the $m$-tuple whose $i$-th coordinate is $2$ and all other coordinates are $1$. Substituting the tuple $\vect{v}_i$ into $f_m = f_n \circ \vect{g}$ we get that $1 = f_m(\vect{v}_i) = f_n \bigl( \vect{g}(\vect{v}_i) \bigr)$. The definition of $f_n$ yields that 
\begin{enumerate}
\item[$(*)$]
for each $i$ ($1 \leq i \leq m$), exactly $n - 1$ of the $n$ elements 
\[
\text{$g_1(\vect{v}_i)$, $g_2(\vect{v}_i)$, \ldots, $g_n(\vect{v}_i)$}
\]
are in the $\theta$-block $1/\theta$. 
\end{enumerate}
This implies that at least $n - 1$ of the operations $g_1, \dotsc, g_n$ have the property that their ranges contain transversals for the blocks of $\theta$. We want to argue that all operations $g_1, \dotsc, g_n$ have this property.

Assume not, and let, say, $g_1$ be the unique operation among $g_1, \dotsc, g_n$ whose range fails to contain a transversal for the blocks of $\theta$. Since the range of $g_1$ contains an element from each one of the $\theta$-blocks other than $1/\theta$, the range of $g_1$ must be disjoint from $1/\theta$. 
Now $(*)$ implies that $g_j(\vect{v}_i) \, \theta \, 1$ for all $j > 1$ and all $i$ ($2 \leq j \leq n$, $1 \leq i \leq m$). In particular, for $j = 2$, this shows that the range of $g_2$ contains a transversal for the blocks of $\theta$, so by Lemma~\ref{lm-hreg} (case $r = 1$) there must exist a $p$ ($1 \leq p \leq m$) such that for arbitrary arguments $\vect{a}, \vect{b} \in A^m$
\[
g_2(\vect{a}) \, \theta \, g_2(\vect{b})
\qquad \text{whenever} \qquad
a_p \, \theta \, b_p.
\]
However, this fails for $\vect{a} = \vect{v}_p$ and $\vect{b} = \bar{2}$; indeed, the $p$-th coordinates of $\vect{v}_p$ and $\bar{2}$ are both $2$, but as we established earlier, $g_2(\vect{v}_p) \, \theta \, 1$, $g_2(\bar{2}) \, \theta \, 2$, and $(1, 2) \notin \theta$. This contradiction proves that all operations $g_1, \dotsc, g_n$ have the property that their ranges contain transversals for the blocks of $\theta$.

Now we can finish the proof the same way as before. It follows from Lemma~\ref{lm-hreg} that each $g_j$ preserves $\theta$ and the operation $g_j^\theta$ on $A/\theta$ depends on at most one variable. It is easy to see from their definitions that the operations $f_m$ and $f_n$ also preserve $\theta$. Hence for the operations $f_m^\theta$ and $f_n^\theta$ on $A/\theta$ we get that $f_m^\theta = f_n^\theta \circ \vect{g}^\theta$. This implies that $f_m^\theta$ depends on at most $n$ variables. But the definition of $f_m$ shows that $f_m^\theta$ depends on all of its $m$ variables. Thus $m \leq n$. This completes the proof of Theorem~\ref{thm-hreg}
\end{proof}

\section{Intersections of maximal clones}
\label{sec:intersecs}

Theorems~\ref{mainthm-eqrels}, \ref{thm-centralk-1}, \ref{thm-central-r},
\ref{thm-slup}, \ref{thm-hreg} from previous sections of this paper, 
combined with earlier results stated in
Theorem~\ref{order_affine} and Corollary~\ref{discr-max} 
completely describe which maximal clones on a finite set $A$
belong to the filter $\FF_A$.
This description is summarized below
(see also Table~\ref{table-maxcl}).

\begin{theorem}
\label{thm-max-summary}
Let $A$ be a finite set with $k$ elements $(k\ge3)$.
For a maximal clone
$\cl M$ on $A$ we have $\cl M\in\FF_A$ if and only if
$\cl M$ is one of the following clones:
\begin{itemize}
\item
$\cl M=\Pol\gamma$ for a prime permutation $\gamma$ on $A$,
\item
$\cl M=\Pol\epsilon$ for 
a nontrivial equivalence relation $\epsilon$ on $A$,
\item
$\cl M=\Pol B$ for a nonempty proper subset $B$ of $A$,
\item
$\cl M=\Pol \sigma_c$ for some $c\in A$ where
$\sigma_c$ is the $(k-1)$-ary central relation with central element $c$, 
\item
$\cl M$ is S{\l}upecki's clone.
\end{itemize}
\end{theorem}

In this section we determine for each pair of maximal clones in $\FF_A$
whether or not their intersection is in $\FF_A$.
The results can be summarized as follows.

\begin{theorem}
\label{thm-intersecs}
Let $A$ be a finite set with $k$ elements $(k\ge3)$, and
let $\cl M$ and $\cl N$ be distinct maximal clones in $\FF_A$.
\begin{enumerate}[\indent\rm(1)]
\item
If $\cl N$ is S{\l}upecki's clone, then $\cl M\cap\cl N\notin\FF_A$.
\item
If $\cl N=\Pol\sigma_c$ for some $c\in A$, then
$\cl M\cap\cl N\in\FF_A$ if and only if $\cl M=\Pol\{c\}$.
\item
If $\cl N=\Pol\epsilon$ for a nontrivial equivalence relation 
$\epsilon$ on $A$ and
$\cl M=\Pol\rho$ where $\rho$ is
a prime permutation, a nonempty proper subset, or a nontrivial equivalence
relation on $A$, then $\cl M\cap\cl N\in\FF_A$ unless
\begin{itemize}
\item
$\rho=\gamma$ is a prime permutation such that $\gamma\notin\cl N$, or
\item
$\rho$ is an equivalence relation 
incomparable to $\epsilon$.
\end{itemize} 
\item
If $\cl M=\Pol\rho$ and $\cl N=\Pol\tau$ where $\rho,\tau$ are
prime permutations or nonempty proper subsets of $A$, then
$\cl M\cap\cl N\in\FF_A$.
\end{enumerate}
\end{theorem}

Since every clone in $\FF_A$ other than $\cl{O}_A$ is below a maximal
clone in $\FF_A$, the ordered set $\FF_A\setminus\{\cl{O}_A\}$ 
can be decomposed into a union of up-closed sets of the form
\[
\FF_A(\cl M):=\{\cl{C}:\cl C\subseteq\cl M\}
\]
for each maximal clone $\cl M$ in $\FF_A$.
Statement (1) of Theorem~\ref{thm-intersecs} shows that 
for S{\l}upecki's clone $\cl N$, the set 
$\FF_A(\cl N)$ is disjoint from all other $\FF_A(\cl M)$'s.
Similarly, statement (2) shows that 
for each $\cl N=\Pol\sigma_c$, the set 
$\FF_A(\cl N)$ is almost disjoint from all other $\FF_A(\cl M)$'s.
In contrast, by statements (3) and (4) (or by the more general
Theorem~\ref{mainthm-eqrels}) there are large overlaps between the
sets $\FF_A(\cl M)$ for the remaining three types of maximal clones.
Thus, Theorem~\ref{thm-intersecs} can be viewed as a structure theorem
for the order filter $\FF_A$, stating that $\FF_A$ consists of three
almost independent parts: (i) the clones contained in S{\l}upecki's clone,
(ii) the clones contained in $\Pol\sigma_c$ for some $c\in A$, and
(iii) the clones that lie below at least one maximal clone of
one  of the remaining three types (i.e., a maximal clone 
determined by a prime permutation,
a subset, or an equivalence relation); see Figure~2.
%
%
\begin{center}%
\begin{figure}%
\setlength{\unitlength}{1mm}%
\begin{picture}(110,60)%
%
\put(60,15){\circle*{1.7}}
\put(47,12){\tiny$\Pol(E,\Aut E,\cl{P}^+(A))$}
\put(46,8.5){\tiny\text{($E$ a chain of equiv rels)}}
\bezier{200}(60,15)(35,30)(21,40)
\bezier{200}(60,15)(72,25)(78,39.6)
\bezier{120}(21,39.6)(26,39.6)(45,39.6)
\bezier{150}(45,39.6)(53.5,20)(62,39.6)
\bezier{75}(62,39.6)(70,39.6)(78,39.6)
%
\thicklines
\put(17,34){\oval(12,4)[tl]}
\put(19,34){\oval(12,4)[tr]}
\bezier{30}(17,36)(18,36.5)(19,38)
\bezier{30}(19,36)(18,36.5)(17,38)
\put(17,40){\line(0,-1){2}}
\put(19,40){\line(0,-1){2}}
\put(11,34){\line(0,-1){1.5}}
\put(25,34){\line(0,-1){1.5}}
\bezier{100}(11,30)(11,27)(16,19)
\bezier{100}(25,30)(25,27)(20,19)
\put(10,40){\oval(14,4)[t]}
\put(3,40){\line(0,-1){19.5}}
\put(55,40){\oval(72,4)[t]}
\put(91,40){\line(0,-1){26}}
\put(100,40){\oval(14,4)[t]}
\put(93,40){\line(0,-1){25.5}}
\put(107,40){\line(0,-1){24}}
\bezier{100}(3,20.5)(12,19)(16,19)
\put(16,19){\line(1,0){4}}
\bezier{50}(20,19)(24,19)(28,13.5)
\bezier{200}(28,13.5)(34,5)(38,2.5)
\bezier{200}(38,2.5)(44,0)(50,0)
\bezier{200}(50,0)(60,0)(75,8)
\bezier{200}(75,8)(85,13)(91,14)
\bezier{200}(93,14.5)(102,16)(107,16)
%
%
\thinlines
\put(13,34){\circle*{1.7}}
\put(15.7,34){$\ldots$}
\put(23,34){\circle*{1.7}}
\put(13,34){\line(4,3){8}}
\put(13,34){\line(-4,3){8}}
\put(23,34){\line(4,3){4.2}}
\bezier{10}(29.4,38.8)(30,39.25)(31,40)
\put(23,34){\line(-4,3){8}}
\put(4,31){\tiny $\Pol(\sigma_1,\!\{\!1\!\})$}
\put(19,31){\tiny $\Pol(\sigma_k,\!\{\!k\!\})$}%
%
\thinlines
\put(100,24){\circle*{1.7}}
\put(94.5,21){\tiny ${\mathcal B}_{k\!-\!1}(T_A^-)$}
\bezier{200}(100,24)(92,32)(100,40)
\bezier{200}(100,24)(108,32)(100,40)%
%
%
\thinlines
\put(100,10){$\FF_A$}
\put(53,60){\circle*{1.7}}
\put(55,60){${\mathcal O}_A$}
\put(5,40){\circle*{1.7}}
\put(-1,37.7){\tiny\text{$\Pol\sigma_1$}}
\bezier{300}(5,40)(29,50)(53,60)
\put(7.7,40){$\ldots$}
\put(15,40){\circle*{1.7}}
\put(9,37.7){\tiny\text{$\Pol\sigma_{k}$}}
\bezier{250}(15,40)(34,50)(53,60)
\put(-1,48){\tiny\text{$(k\!-\!1)$-ary}}
\put(-1,45.5){\tiny central}
\put(-1,43){\tiny rels}
\put(21,40){\circle*{1.7}}
\bezier{250}(21,40)(37,50)(53,60)
\put(23.7,40){$\ldots$}
\put(31,40){\circle*{1.7}}
\bezier{250}(31,40)(42,50)(53,60)
\put(35,40){\circle*{1.7}}
\bezier{250}(35,40)(44,50)(53,60)
\put(37.7,40){$\ldots$}
\put(45,40){\circle*{1.7}}
\bezier{250}(45,40)(49,50)(53,60)
\put(27.3,37.5){\tiny\text{$\Pol\{\!k\!\}$}}
\put(29,43){\tiny sub}
\put(36,43){\tiny -}
\put(40,43){\tiny sets}
\put(51,40){\circle*{1.7}}
\bezier{250}(51,40)(52,50)(53,60)
\put(56.7,40){$\ldots$}
\put(67,40){\circle*{1.7}}
\bezier{250}(67,40)(60,50)(53,60)
\put(55,43){\tiny perms}
\put(73,40){\circle*{1.7}}
\bezier{250}(73,40)(63,50)(53,60)
\put(78.7,40){$\ldots$}
\put(89,40){\circle*{1.7}}
\bezier{250}(89,40)(71,50)(53,60)
\put(70,45.5){\tiny equiv.}
\put(73,43){\tiny rels}
\put(100,40){\circle*{1.7}}
\bezier{300}(100,40)(76.5,50)(53,60)
\put(101,39){\tiny ${\mathcal B}_{k\!-\!1}$}
\end{picture}%
\caption{The structure of $\FF_A$ for $A=\{1,2,\ldots,k\}$ ($k\ge3$)}%
\end{figure}%
\end{center}%
%

For the proof of Theorem~\ref{thm-intersecs} we have to verify
that almost all intersections $\cl M\cap\cl N$ of maximal clones
$\cl M,\cl N\in\FF_A$ fail to be in $\FF_A$ if
$\cl N$ is S{\l}upecki's clone
or a maximal clone determined by a $(k-1)$-ary central relation. 
This will be done in 
Lemmas~\ref{lm-centralk-Bk-1}--\ref{lm-eqrel-Bk-1} 
and Lemmas~\ref{lm-centralk-1subset}--\ref{lm-centralk-1eqrel} below.

We will assume throughout that $A$ is a finite set with $k$ elements, 
and will use the notation $\cl{B}_{k-1}$ and $\cl{B}_{k-2}$
from Section~\ref{sec:h-reg} for S{\l}upecki's clone
and its lower cover in the S{\l}upecki--Burle chain.

\begin{lemma}
\label{lm-centralk-Bk-1}
If $c\in A$, then 
$\Pol\sigma_c\cap\cl{B}_{k-1}\notin\FF_A$.
\end{lemma}

\begin{proof}
Let $\cl{C} = \Pol \sigma_c \cap \cl{B}_{k-1}$, and assume 
without loss of generality that $A=\{0,1,\ldots,k-1\}$ and $c=0$.
To simplify notation we will write $\sigma$ for $\sigma_0$.
For each $n \geq 1$ define an $n$-ary operation $f_n$ on $A$ as follows:
\[
f_n(\vect{a})=\begin{cases}
              u & \text{if $\vect{a}=\bar{u}$, $1\le u\le k-1$,}\\ 
              0 & \text{otherwise.} 
              \end{cases}
\]
Clearly, $f_n$ depends on all of its variables, 
because it is invariant under all permutations of its variables, and 
is not constant. 

We claim that $f_n \not\fequiv[\cl{C}] f_m$ whenever $n \neq m$, 
and hence $\cl{C}\notin\FF_A$. 
For, suppose on the contrary that
$f_n \fequiv[\cl{C}] f_m$ for some $n < m$.
Then there exists 
$\vect{h} = (h_1, \dotsc, h_n) \in (\cl{C}^{(m)})^n$ such that 
$f_m = f_n \circ \vect{h}$. 
Thus $\vect{h}$ maps $f_m^{-1}(u)$ into $f_n^{-1}(u)$ for each
$u\in A$. 
Since for $1\le u\le k-1$ the set
$f_m^{-1}(u)$ (resp.\ $f_n^{-1}(u)$) contains the $m$-tuple ($n$-tuple)
$\bar{u}$ only, we get that
$\vect{h}(\bar{u}) = \bar{u}$ holds for all
$u$ with $1\le u\le k-1$. 
We will distinguish two cases according to whether or not 
$\vect{h}(\bar{0}) = \bar{0}$.

Assume first that $\vect{h}(\bar{0}) = \bar{0}$. 
Then each $h_i$ ($1 \leq i \leq n$) is surjective and, 
being a member of $\cl{B}_{k-1}$, $h_i$ is thus 
essentially unary.
Therefore the equality $f_m = f_n \circ \vect{h}$ implies that
$f_m$ depends on at most $n\ (<m)$ variables.
This is impossible, since  we established
earlier that $f_m$ depends on all $m$ 
of its variables.

Assume now that 
$\vect{h}(\bar{0}) = \vect{b}=(b_1, \dotsc, b_n) \neq \bar{0}$. 
Then $b_i=b \neq 0$ for some $i$ ($1\le i\le n$).
Then
$(\bar{1},\ \ldots,\ \overline{b-1},\ \bar{0},\ \overline{b+1},\ \ldots,\ 
\overline{k-1})\in\sigma^m$,
but 
\begin{align*}
\bigl(\vect{h}(\bar{1}),\ \ldots,\ \vect{h}(\overline{b-1}),\ 
&\vect{h}(\bar{0}),\ 
\vect{h}(\overline{b+1}),\ \ldots,\ \vect{h}(\overline{k-1})\bigr)\\
&{}=\bigl(\bar{1},\ \ldots,\ \overline{b-1},\ \vect{b},\ 
\overline{b+1},\ \ldots,\ \overline{k-1}\bigr)\notin\sigma^n,
\end{align*}
since the $i$-th coordinate of the tuple is
$(1,\ldots,b-1,b,b+1,\ldots,k-1)\notin\sigma$.
This is again impossible, since our assumption that
$\vect{h}\in (\cl{C}^{(m)})^n$ requires that $\vect{h}$ preserve
$\sigma$.
\end{proof}

\begin{lemma}
\label{lm-perm-Bk-1}
If $\gamma$ is a prime permutation on $A$, then
$\Pol\gamma\cap\cl{B}_{k-1}\notin\FF_A$.
\end{lemma}

\begin{proof} 
Let $\cl{C}=\Pol\gamma\cap\cl{B}_{k-1}$ and $k=|A|$.
Our goal is to prove that $\cl{C}\subseteq\cl{B}_{k-2}$.
Since $\cl{B}_{k-2}\notin\FF_A$ by Theorem~\ref{thm-slup}, this will
imply our claim that $\cl{C}\notin\FF_A$.

By assumption, $\gamma$ is a prime permutation. 
Therefore $\gamma$ has no fixed points, and
every cycle of $\gamma$ has the same prime length $p$.
So $k=mp$ for some integer $m\ge1$.
First we will show that the range of every operation in
$\Pol\gamma$ is closed under $\gamma$.
Indeed,
let $f$ be an $n$-ary operation in $\Pol\gamma$, and let 
$a\in\range{f}$, i.e., $a=f(a_1,\ldots,a_n)$ for some
$a_1,\ldots,a_n\in A$. Then 
$\gamma(a)=\gamma\bigl(f(a_1,\ldots,a_n)\bigr)
=f\bigl(\gamma(a_1),\ldots,\gamma(a_n)\bigr)$ holds because
$f\in\Pol\gamma$, hence $\gamma(a)\in\range{f}$.
This proves that if $f\in\Pol\gamma$, then 
$\range{f}$ is closed under $\gamma$.
It follows that $\range{f}$ is closed under all powers of $\gamma$, including
$\gamma^{-1}=\gamma^{p-1}$. 
Hence $A\setminus\range{f}$ is also closed under $\gamma$. 
This
implies that if $\range{f}\not=A$, then $|\range{f}|\le|A|-p=k-p$.

Now we are ready to prove that $\cl{C}\subseteq\cl{B}_{k-2}$.
Let $f\in\cl{C}$. If $f$ is essentially at most unary, then
$f\in\cl{B}_0\subseteq \cl{B}_{k-2}$. 
So, suppose that $f$ is not essentially at most unary. 
Then $f\in\cl{B}_{k-1}$ implies that $\range{f}\not=A$, and hence 
$f\in\Pol\gamma$ implies,
by our discussion in the preceding paragraph, that
$|\range{f}|\le k-p$. For $k=p$ 
this shows that such an $f$ cannot exist, while for $k=mp\ge 2p$ 
it shows that $f\in\cl{B}_{k-p}\subseteq\cl{B}_{k-2}$.
In either case, this completes the proof that
$\cl{C}\subseteq\cl{B}_{k-2}$,
and finishes the proof of the lemma.
\end{proof}

\begin{lemma}
\label{lm-subs-Bk-1}
If $B$ is a nonempty proper subset of $A$, then
$\Pol B\cap\cl{B}_{k-1}\notin\FF_A$.
\end{lemma}

\begin{proof}
Let $\cl{C}=\Pol B\cap\cl{B}_{k-1}$.
We may assume without loss of generality that $A=\{0,1,\ldots,k-1\}$
and $B = \{0, 1, \dotsc, r - 1\}$ ($1\le r<k$).
For the proof of $\cl{C}\notin\FF_A$
we will use a description of 
S{\l}upecki's clone $\cl{B}_{k-1}$ via relations. 
As we mentioned at the beginning of Section~\ref{sec:h-reg},
$\cl{B}_{k-1}=\Pol\lambda_T$
where $\lambda_T$ is the $h$-regular relation associated to
the singleton $T=\{\zero_A\}$ 
consisting of the equality relation on $A$.
It is clear from the definition of $h$-regular relations in 
Section~\ref{sec:prelim} that
the relation $\lambda_{\{\zero_A\}}$ is nothing else
than the $k$-ary relation
\[
\iota_k=\{(a_1,\ldots ,a_k)\in A^k: a_1,\ldots ,a_k 
\text{ are not pairwise distinct}\}.
\]
Hence $\cl{B}_{k-1}=\Pol\iota_k$.

Now we turn to the proof of $\cl{C}\notin\FF_A$.
First we will consider the case when
$r=1$, and hence $B=\{0\}$.
It is straightforward to verify that
\begin{align*}
\sigma_0&{}=\{(a_1,\ldots ,a_{k-1})\in A^{k-1}:
                 (a_1,\ldots ,a_{k-1},0)\in\iota_k\}\\
       &{}=\{(a_1,\ldots ,a_{k-1})\in A^{k-1}: \\
       &\qquad\qquad\text{there exists $a_k\in B$
                 such that $(a_1,\ldots ,a_{k-1},a_k)\in\iota_k$}\}.
\end{align*}
This shows that 
every operation on $A$ that preserves $\iota_k$ and $B$, also preserves
$\sigma_0$.
Therefore
we get that $\cl C\subseteq\Pol\sigma_0\cap\cl{B}_{k-1}$.
By Lemma~\ref{lm-centralk-Bk-1}, $\Pol\sigma_0\cap\cl{B}_{k-1}\notin\FF_A$,
so it follows that $\cl{C}\notin\FF_A$.

{}From now on we will assume that $r\ge2$.
For each $a$ ($0\leq a \leq k - 1$) let $\vect{e}_a^k$
denote the $k$-tuple whose $j$-th coordinate is 
$(a + j - 1) \mmod{k}$ for each $j$.
Furthermore,
for each $a$ ($0\leq a \leq k - 1$)  and $n>k$ 
let $\vect{e}_a^n$ denote the $n$-tuple that is the concatenation of
$\vect{e}_a^k$ with the constant $(n-k)$-tuple
repeating the last coordinate of $\vect{e}_a^k$.
Thus, 
\begin{align*}
\vect{e}_0^n&{}=(0,1,\ldots,k-1,k-1,\ldots,k-1),
\quad\text{and}\\ 
\vect{e}_a^n&{}=(a,a+1,\ldots,k-1,0,\ldots,a-1,a-1,\ldots,a-1)
\quad\text{for $1\le a\le k-1$}. 
\end{align*}
Note that none of the $n$-tuples $\vect{e}_a^n$ is a member of $B^n$,
since all elements of $A$ occur among the coordinates of $\vect{e}_a^n$.

For each $n \geq k$ we define an $n$-ary operation $f_n$ on $A$ as follows:
\[
f_n(a_1, \dotsc, a_n) =
\begin{cases}
a_1 & \text{if $(a_1, \dotsc, a_n) \in B^n$,} \\
a_1 & \text{if $a_1 = a_2 = \dotsb = a_n$, $r\le a_1\le k-1$,} \\
1   & \text{if $a_1 = 1$ and $(a_1, \dotsc, a_n) \notin 
      B\cup\{\vect{e}_1^n$\},} \\
1   & \text{if $(a_1, \dotsc, a_n) = \vect{e}_a^n$ for some $a \neq 1$,} \\
0   & \text{otherwise.}
\end{cases}
\]
Note that $f_n$ depends on all of its variables, which can be seen as follows:
for any $i$ with $1\le i\le n$ the $n$-tuple $\overline{k-1}$ and 
the $n$-tuple $\vect{u}_i$ obtained from $\overline{k-1}$ by changing the $i$-th
coordinate to $0$ satisfy
$f_n(\overline{k-1})=k-1\not=0=f_n(\vect{u}_i)$. 

We claim that $f_n \not\fequiv[\cl{C}] f_m$ whenever $n \neq m$, and hence
$\cl{C}\notin\FF_A$. 
For, suppose on the contrary that
$f_n \fequiv[\cl{C}] f_m$ for some $n < m$.
Then there exists $\vect{h} = (h_1, \dotsc, h_n)
\in (\cl{C}^{(m)})^n$ such that $f_m = f_n \circ \vect{h}$.
Thus $\vect{h}$ maps each set $f_m^{-1}(u)$ ($u\in A$) 
into the set $f_n^{-1}(u)$.
Since for $r\le a\le k-1$ the set $f_m^{-1}(a)$ (resp.\ $f_n^{-1}(a)$)
contains the $m$-tuple ($n$-tuple) $\bar{a}$ only, we get that
$\vect{h}(\bar{a})=\bar{a}$ holds for all $r\le a\le k-1$.
In particular, 
$h_1(\bar{a})=a$ for all $r\le a\le k-1$.

Now let $\vect{a}=(a_1,\dotsc, a_m) \in B^m$. Since $\vect{h}$ preserves $B$, 
we have that 
$\vect{h}(\vect{a})=\bigl(h_1(\vect{a}),\ldots,h_n(\vect{a})\bigr)\in B^n$.
So, by applying the definitions of $f_m$ and $f_n$ for tuples in $B$ we get 
that
\[
a_1=f_m(\vect{a})=f_n\bigl(\vect{h}(\vect{a})\bigr)
=f_n\bigl((h_1(\vect{a}),\ldots,h_n(\vect{a}))\bigr)=
h_1(\vect{a}),
\]
that is, $h_1$ restricted to $B$ is projection onto the first variable.
Combining this with the property of $h_1$ established in the preceding
paragraph we get that $h_1$ is
surjective, and hence, being a member of $\cl{B}_{k-1}$, 
it is essentially unary.
The fact that $h_1$ restricted to the set $B$ of size $\ge2$
depends on its first variable forces that it is the first variable that
$h_1$ depends on.
Since $h_1(\bar{a})=a$ for all $a$
(whether $a\in B$ or $r\le a\le k-1$), we conclude that
$h_1$ is projection onto the first variable.

Next we want to determine $\vect{h}({\vect{e}_a^m})$ for each $a\in A$.
Since $h_1$ is projection onto the first variable and 
$\vect{e}_a^m$ has first coordinate $a$, we get that 
$\vect{h}({\vect{e}_a^m})$ also has first coordinate $a$.
On the other hand, 
\[
f_n\bigl(\vect{h}(\vect{e}_a^m)\bigr)=f_m(\vect{e}_a^m)=
\begin{cases}
1 & \text{if $a\not=1$,}\\
0 & \text{if $a=1$.}
\end{cases}
\]
If $a\not=1$, then $\vect{h}(\vect{e}_a^m)$ is an $n$-tuple with first
coordinate $a\not=1$ whose $f_n$-image is $1$. It follows from the definition
of $f_n$ that the only such $n$-tuple is $\vect{e}_a^n$, so
$\vect{h}(\vect{e}_a^m)=\vect{e}_a^n$.
If $a=1$, then $\vect{h}(\vect{e}_1^m)$ is an $n$-tuple with first
coordinate $1$ whose $f_n$-image is $0$. Again, the definition
of $f_n$ shows that the only such $n$-tuple is $\vect{e}_1^n$, so
$\vect{h}(\vect{e}_1^m)=\vect{e}_1^n$.
This proves that $\vect{h}(\vect{e}_a^m)=\vect{e}_a^n$ for all $a\in A$.

Since for each $i$ ($1\le i\le n$) the $i$-th coordinates of the tuples
$\vect{e}_a^n$ ($a\in A$) exhaust $A$, we obtain that each $h_i$ is 
surjective. But, $h_i\in\cl{B}_{k-1}$ for each $i$, therefore
each $h_i$ is essentially unary.
Hence $f_m = f_n \circ \vect{h}$ yields
that $f_m$ depends on at most $n\ (<m)$ variables, contradicting 
the fact that $f_m$ depends on all $m$ of its variables.
\end{proof}

\begin{lemma}
\label{lm-eqrel-Bk-1}
If $\epsilon$ is a nontrivial equivalence relation on $A$, then
$\Pol\epsilon\cap\cl{B}_{k-1}\notin\FF_A$.
\end{lemma}

\begin{proof}
Let $\cl{C}=\Pol\epsilon\cap\cl{B}_{k-1}$.
We may assume without loss of generality that 
$A=\{0,1,\ldots,k-1\}$, and 
the equivalence classes of $\epsilon$
are $\{0,1,\ldots,n_1\}$, $\{n_1+1,\ldots,n_2\}$, $\ldots$, 
$\{n_{r-1}+1,\ldots ,k-1\}$ for some $r\ge2$ and some 
$0=n_0+1<1\le n_1 < n_2 < \cdots < n_{r-1}<n_r=k-1$.

For $0\le a \le k-1$ and $n\ge 1$ let 
$\vect{e}_a^{kn}$ denote the $kn$-tuple that is a concatenation
of $k$ constant $n$-tuples such that the $(jn+1)$-th coordinate
of $\vect{e}_a^{kn}$ is $(a+j) \mmod k$ 
for each $j$ ($0\le j\le k-1$);
equivalently, 
\[
\vect{e}_a^{kn}=(\bar{a},\overline{a+1},\ldots,\overline{k-1},
\bar{0},\ldots,\overline{a-1})
\quad(0\le a\le k-1)
\]
where each constant tuple $\bar{b}$ ($0\le b\le k-1$)
has length $n$.
Two properties of these tuples will be important:
\begin{itemize}
\item
$(\vect{e}_a^{kn},\vect{e}_{b}^{kn})\notin\epsilon^{kn}$
if $a\not=b$, and
\item
$\vect{e}_a^{kn}/\epsilon^{kn}$ has an element other than 
$\vect{e}_a^{kn}$ for each $a$.
\end{itemize}
The first property can be verified by 
observing that if $(a,b)\notin\epsilon$, then
the first coordinates of $\vect{e}_a^{kn}$ and $\vect{e}_b^{kn}$
are not $\epsilon$-related, while if $(a,b)\in\epsilon$, say
$n_i+1\le a<b\le n_{i+1}$, then 
$\bigl(a+(n_{i+1}-b+1),b+(n_{i+1}-b+1)\bigr)\notin\epsilon$,
so that for $j=n_{i+1}-b+1$ the $(nj+1)$-th coordinates
of $\vect{e}_a^{kn}$ and $\vect{e}_b^{kn}$
are not $\epsilon$-related.
The second property is true because the assumption $(0,1)\in\epsilon$
ensures that if for any $\ell$ ($1\le \ell\le n$)
we replace the $\ell$-th occurrence of $0$ in $\vect{e}_a^{kn}$  
by $1$, we get a $kn$-tuple $(\vect{e}_a^{kn})^{[\ell]}$ which is 
$\epsilon^{kn}$-related, but not equal to $\vect{e}_a^{kn}$.

For $n \ge 1$ we now define a $kn$-ary operation $f_n$ 
on $A$ as follows:
\[
f_n(\vect{a})=
\begin{cases}
a & \text{if $\vect{a}=\vect{e}_a^{kn}$ for some $a$},\\
(a+1)\mmod k & \text{if $(\vect{a},\vect{e}_a^{kn})\in\epsilon^{kn}$,\ \ 
                   $\vect{a}\not=\vect{e}_a^{kn}$ for some $a$,}\\
0 & \text{otherwise.}
\end{cases}
\]
The properties of $\vect{e}_a^{kn}$ established in the preceding paragraph
make sure that $f_n$ is well-defined, and that 
$f_n[\vect{e}_a^{kn}/\epsilon^{kn}]$
is a $2$-element set for each $a$.
Moreover, it follows also that $f_n$ depends on all
of its variables, because
for each $j$ $(1\le j\le kn)$ there exist $a$ and
$\ell$ such that the $kn$-tuples $\vect{e}_a^{kn}$
and $(\vect{e}_a^{kn})^{[\ell]}$ differ in their $j$-th coordinates
only, and $f_n(\vect{e}_a^{kn})=a\not=(a+1) \mmod k =
f_n((\vect{e}_a^{kn})^{[\ell]})$.

We claim that $f_n\not\fequiv[\cl{C}]f_m$ whenever $n \not=  m$, and hence
$\cl{C}\notin\FF_A$. 
For, let $n < m$, and suppose
on the contrary that there exists 
$\vect{h} = (h_1,\ldots ,h_n)\in(\cl{C}^{(m)})^n$ 
such that $f_m = f_n \circ \vect{h}$.
It follows from the definition of $f_n$ that 
for each block $B$ of $\epsilon^{kn}$, 
$f_n[B]$ is the $2$-element set consisting of
$a$ and $(a+1) \mmod k$, if 
$B=\vect{e}_a^{kn}/\epsilon^{kn}$ ($0\le a\le k-1$), and 
$f_n[B]$ is the singleton $\{0\}$ otherwise.
We want to use this fact to prove that $\vect{h}(\vect{e}_a^{km})=
\vect{e}_a^{kn}$ holds for each $a$ ($0\le a\le k-1$).
Indeed, since $\vect{h}$ preserves $\epsilon$, 
therefore $\vect{h}$ maps
$\vect{e}_a^{km}/\epsilon^{km}$ into a single $\epsilon^{kn}$-block 
$B$ in $A^{kn}$.
As 
\[
\{a,(a+1)\mmod k\}=f_m[\vect{e}_a^{km}/\epsilon^{km}]
=f_n\bigl[\vect{h}[\vect{e}_a^{km}/\epsilon^{km}]\bigr]
\subseteq f_n[B],
\]
we get that $B=\vect{e}_a^{kn}/\epsilon^{kn}$.
Since 
$a=f_m(\vect{e}_a^{km})=f_n\bigl(\vect{h}(\vect{e}_a^{km})\bigr)$, and
$\vect{a}=\vect{e}_a^{kn}$ is the only element $\vect{a}\in B$ for which
$f_n(\vect{a})=a$,
we conclude that
$\vect{h}(\vect{e}_a^{km})=\vect{e}_a^{kn}$, as claimed.

Since for each $i$ ($1\le i\le n$), all elements of $A$ occur in the 
$i$-th coordinate of some $\vect{e}_a^{kn}$, 
the equalities $\vect{h}(\vect{e}_a^{km})=\vect{e}_a^{kn}$ ($a\in A$) 
imply that each $h_i$ is surjective.
As each $h_i$ is a member of $\cl{B}_{k-1}$, we get that
each $h_i$ is essentially unary.
Hence $f_m = f_n \circ \vect{h}$ yields
that $f_m$ depends on at most $n\ (<m)$ variables, contradicting 
the fact established earlier 
that $f_m$ depends on all $m$ of its variables.
\end{proof}

Next we will consider intersections of $\Pol\sigma_c$ with
other maximal clones in $\FF_A$.
We will start with two auxiliary lemmas.

\begin{lemma}
\label{lm-centralk-1-gen}
Let $A=\{0,1,\ldots,k-1\}$, and let $\sigma_0$ be the $(k-1)$-ary central
relation on $A$ with central element $0$.
A subclone $\cl{C}$ of
$\Pol\sigma_0$ fails to belong to $\FF_A$ if
for some integers $n_0\ge3$ and $l\in\{0,1\}$, 
there exist $n$-tuples $\vect{c}_i^n$ $(1\le i\le n-l)$
for each $n\ge n_0$ such that the following two conditions are satisfied:
\begin{enumerate}[\indent\rm(1)]
\item
For all $n\ge n_0$ and $1\le i\le j\le n-l$ we have
\[
(\vect{c}_i^n,\vect{c}_j^n,\bar 3,\ldots,\overline{k-1})
\in(\sigma_0)^{n}
\quad\iff\quad
j-i\le 1.
\]
\item
For all $m,n\ge n_0$, we have $\vect{h}(\vect{c}_1^m)=\vect{c}_1^n$
whenever
$\vect{h}\in(\cl{C}^{(m)})^n$ is such that
\begin{enumerate}[\rm(i)]
\item
$\vect{h}(\vect{c}_1^m)\in\{\vect{c}_i^n:1\le i<n-l\}$,
\item
$\vect{h}(\vect{c}_{m-l}^m)=\vect{c}_{n-l}^n$, and
\item
$\vect{h}(\bar b)=\bar b$ for all $3\le b\le k-1$.
\end{enumerate}
\end{enumerate}
\end{lemma}

\begin{proof}
Let $\cl{C}$ be a subclone of $\Pol\sigma_0$, and assume that 
conditions (1) and (2) are satisfied for some
$n$-tuples $\vect{c}_i^n$ ($n\ge n_0$, $1\le i\le n-l)$.
For each $n\ge  n_0$ we define
an $n$-ary operation $f_n$ on $A$ as follows:
\[
f_n(\vect{a})=
\begin{cases}
1 & \text{if $\vect{a}=\vect{c}_i^{n}$ ($1\le i< n-l$),}\\
2 & \text{if $\vect{a}=\vect{c}_{n-l}^{n}$,}\\
b & \text{if $\vect{a}=\bar b$ ($3\le b\le k-1$),}\\
0 & \text{otherwise.}\\
\end{cases}
\]
We will prove $\cl{C}\notin\FF_A$ by showing that
$f_n \not\fequiv[\cl{C}] f_m$ whenever $n \neq m$. 

Suppose that, on the contrary, there exist
$m<n$ such that $f_n \fequiv[\cl{C}] f_m$. Hence there exists
$\vect{h}\in (\cl{C}^{(m)})^{n}$ 
such that $f_m = f_n \circ \vect{h}$. 
Thus $\vect{h}$ maps each set $f_m^{-1}(b)$ ($b\in A$) 
into the set $f_n^{-1}(b)$.
Applying this to $3\le b\le k-1$ and to $b=2$ we get that
$\vect{h}(\bar b)=\bar b$ for all $3\le b\le k-1$ and
$\vect{h}(\vect{c}_{m-l}^{m})=\vect{c}_{n-l}^{n}$.
The same property for $b=1$ shows that 
\begin{equation}
\label{eq-i}
\vect{h}(\vect{c}_j^{m})
\in\{\vect{c}_i^{n}:1\le i<n-l\}
\quad
\text{for all $j$ ($1\le j<m-l$).}
\end{equation}
In particular, $\vect{h}(\vect{c}_1^{m})
\in\{\vect{c}_i^{n}:1\le i<n-l\}$. 
Thus $\vect{h}\in (\cl{C}^{(m)})^{n}$ satisfies 
all three requirements (i)--(iii) in (2). 
Therefore we can apply condition (2) to conclude that
\begin{equation}
\label{eq-(2)}
\vect{h}(\vect{c}_1^{m})=\vect{c}_1^{n}.
\end{equation}
By condition (1) we have 
$(\vect{c}_j^m,\vect{c}_{j+1}^m,\bar 3,\ldots,\overline{k-1})
\in(\sigma_0)^{m}$ for all $j$ ($1\le j<m-l$).
Since $\vect{h}\in (\cl{C}^{(m)})^{n}$, and therefore 
$\vect{h}$ preserves $\sigma_0$,
the $\vect{h}$-images of these tuples are in $(\sigma_0)^n$. 
Since $\vect{h}$ satisfies (iii),
this means that
\begin{equation}
\label{eq-ii}
\bigl(\vect{h}(\vect{c}_{j}^{m}),\vect{h}(\vect{c}_{j+1}^{m}),
            \bar 3,\ldots,\overline{k-1}\bigr)
\in(\sigma_0)^{n}
\quad
\text{for all $j$ ($1\le j< m-l$).} 
\end{equation}
Using (\ref{eq-i})
we obtain that the $(m-l)$-element sequence of $\vect{h}$-images 
\[
\vect{c}_1^{n}=\vect{h}(\vect{c}_1^{m}),\ \vect{h}(\vect{c}_2^{m}),\ 
\ldots,\ \vect{h}(\vect{c}_{j}^{m}),\ \vect{h}(\vect{c}_{j+1}^{m}),\ 
\ldots,\ \vect{h}(\vect{c}_{m-l-1}^{m}),\ \vect{h}(\vect{c}_{m-l}^{m})=
\vect{c}_{n-l}^{n}
\] 
has its first $m-l-1$ members in the set
$\{\vect{c}_i^{n}:1\le i<n-l\}$.
(The equalities for the first and last members follow from
(\ref{eq-(2)}) and the
fact that $\vect{h}$ satisfies (ii).)
Thus $\vect{h}(\vect{c}_{j}^{m})=\vect{c}_{s_j}^{n}$ for each $j$
($1\le j\le m-l$) so that $s_1=1$, $1\le s_2,\ldots,s_{m-l-1}<n-l$, and
$s_{m-l}=n-l$.
Combining this with (\ref{eq-ii}) we get that
\begin{equation*}
\bigl(\vect{c}_{s_{j}}^{n},\vect{c}_{s_{j+1}}^{n},
            \bar 3,\ldots,\overline{k-1}\bigr)
\in(\sigma_0)^{n}
\quad
\text{for all $j$ ($1\le j\le m-l-1$).}
\end{equation*}
Since $\sigma_0$ is totally symmetric, we also have that
\begin{equation*}
\bigl(\vect{c}_{s_{j+1}}^{n},\vect{c}_{s_{j}}^{n},
            \bar 3,\ldots,\overline{k-1}\bigr)
\in(\sigma_0)^{n}
\quad
\text{for all $j$ ($1\le j\le m-l-1$).}
\end{equation*} 
Thus it follows from condition (1) that $|s_{j+1}-s_j|\le 1$ for all $j$
($1\le j\le m-l-1$). 
Therefore 
$n-l-1=|s_{m-l}-s_1|\le\sum_{j=1}^{m-l-1}|s_{j+1}-s_j|\le m-l-1$,
which contradicts our assumption that $m<n$.
This completes the proof of the lemma.
\end{proof}

\begin{lemma}
\label{lm-centralk-1-aux}
Let $A=\{0,1,\ldots,k-1\}$, let $\sigma_0$ be the $(k-1)$-ary central
relation on $A$ with central element $0$, and for  
$n \geq 4$ and $2 \leq i \leq n-1$ let
\[
\vect{e}_i^{n}=(1,\ldots,1,0,2,0,1,\ldots,1)\in A^{n}
\]
be the $n$-tuple where the sole $2$ is in the $i$-th coordinate.
For all $2\le i\le j\le n-1$ we have
\[
(\vect{e}_i^{n},\vect{e}_j^{n},\bar 3,\ldots,\overline{k-1})
\in(\sigma_0)^{n}
\quad\iff\quad
j-i\le 1.
\]
\end{lemma}

\begin{proof}
Let $j\ge i$. 
If $j=i$ or $j=i+1$, then in each coordinate, the $(k-1)$-tuple
$(\vect{e}_i^{n},\vect{e}_j^{n},\bar 3,\ldots,\overline{k-1})$
is of the form $(0,0,\ldots)$, $(1,1,\ldots)$, $(2,2,\ldots)$,
$(0,1,\ldots)$, $(2,0,\ldots)$, $(0,2,\ldots)$, or $(1,0.\ldots)$.
Thus $(\vect{e}_i^{n},\vect{e}_j^{n},\bar 3,\ldots,\overline{k-1})
\in(\sigma_0)^{n}$.
If $j>i+1$, then in the $i$-th coordinate of the $(k-1)$-tuple
$(\vect{e}_i^{n},\vect{e}_j^{n},\bar 3,\ldots,\overline{k-1})$
we have $(2,1,3,\ldots,k-1)\notin\sigma_0$, hence 
$(\vect{e}_i^{n},\vect{e}_j^{n},\bar 3,\ldots,\overline{k-1})
\notin(\sigma_0)^{n}$.
\end{proof}

\begin{lemma}
\label{lm-centralk-1subset}
If $c\in A$ and $B$ is a nonempty proper subset of $A$ such that 
$B \neq \{c\}$, then  $\Pol B\cap\Pol\sigma_c\notin\FF_A$.
\end{lemma}

\begin{proof}
Let $\cl{C} = \Pol B\cap\Pol\sigma_c$.
We may assume without loss of generality that $A=\{0,1,\ldots,k-1\}$
and $c=0$. In view of Lemma~\ref{lm-centralk-1-gen}, our claim
that $\cl{C}\notin\FF_A$ will follow if we
exhibit tuples $\vect{c}_i^n$ that satisfy conditions (1) and (2).
We will distinguish two cases according to whether 
$c=0$ is a member of $B$ or not.

{\it Case 1:} $0 \in B$. In this case $\card{B} \geq 2$. 
Assume without loss of generality that $0,1\in B$ and 
$2 \notin B$. 
For $n \geq 4$, let 
$\vect{c}_1^{n} = (0, 0, 1, 1, \dotsc, 1) \in A^{n}$, and let
$\vect{c}_i^n=\vect{e}_i^{n}$ ($2\le i\le n-1$) be the tuples from 
Lemma~\ref{lm-centralk-1-aux}.
We claim that (1)--(2) of Lemma~\ref{lm-centralk-1-gen}
hold true (with $n_0=4$ and $l=1$).
For $j\ge i\ge2$, condition (1) follows from Lemma~\ref{lm-centralk-1-aux}.
So, let $i=1$. Then it is straightforward to check that
$\bigl(\vect{c}_1^{n},\vect{c}_{2}^{n},
            \bar 3,\ldots,\overline{k-1}\bigr)
\in(\sigma_0)^n$, while if $j>2$, then
$\bigl(\vect{c}_1^{n},\vect{c}_{j}^{n},
            \bar 3,\ldots,\overline{k-1}\bigr)
\notin(\sigma_0)^n$, because
in the $j$-th coordinate we have
$(1,2,3,\ldots,k-1)\notin\sigma_0$.
This proves that condition (1) holds.
To establish condition (2) assume that
$\vect{h}\in(\cl{C}^{(m)})^n$ ($m,n\ge4$)
satisfies requirements (i)--(iii) in condition (2);
in fact, in this case
it will be enough to assume that $\vect{h}$ satisfies (i), that is,
$\vect{h}(\vect{c}_1^m)\in\{\vect{c}_i^n:1\le i<n-1\}$.
Here $\vect{c}_1^{m}\in B^{m}$, 
$\vect{c}_1^{n}\in B^{n}$, and 
$\vect{c}_2^{n},\ldots,\vect{c}_{n-2}^{n}\notin B^{n}$,
because $B$ contains $0,1$ and
does not contain $2$.
Since $\vect{h}$ preserves $B$, we must have that
$\vect{h}(\vect{c}_1^{m})=\vect{c}_1^{n}$.

{\it Case 2:}
$0 \notin B$. 
Assume without loss of generality that $2 \in B$. 
For $n \geq 4$, consider the following $n$-tuples: 
$\vect{c}_1^n =\bar 2$, $\vect{c}_2^n = (2, 0, 0, \dotsc, 0)$, 
and $\vect{c}_i^n = \vect{e}_{i - 1}^n$ 
from Lemma~\ref{lm-centralk-1-aux} for $3 \leq i \leq n$.
We want to show that conditions (1) and (2) of
Lemma~\ref{lm-centralk-1-gen} are satisfied (with $n_0=4$ and $l=0$).
For $j\ge i\ge3$ condition (1) follows from Lemma~\ref{lm-centralk-1-aux}.
For $i=2$ we have
$\bigl(\vect{c}_2^{n},\vect{c}_{3}^{n},
            \bar 3,\ldots,\overline{k-1}\bigr)
\in(\sigma_0)^n$, since in each coordinate
one of $\vect{c}_2^{n}$, $\vect{c}_{3}^{n}$ is $0$.
On the other hand, if $j>3$, then
$\bigl(\vect{c}_2^{n},\vect{c}_{j}^{n},
            \bar 3,\ldots,\overline{k-1}\bigr)
\notin(\sigma_0)^n$, because in the first coordinate we have
$(2,1,3,\ldots,k-1)\notin\sigma_0$.
For $i=1$, 
$\bigl(\vect{c}_1^{n},\vect{c}_{2}^{n},
            \bar 3,\ldots,\overline{k-1}\bigr)
\in(\sigma_0)^n$, since in each coordinate
this $(k-1)$-tuple has the form $(2,2,\ldots)$ or $(2,0,\ldots)$.
However, for $j>2$ we have 
$\bigl(\vect{c}_1^{n},\vect{c}_{j}^{n},
            \bar 3,\ldots,\overline{k-1}\bigr)
\notin(\sigma_0)^n$, because in every coordinate where
$\vect{c}_{j}^{n}$ is $1$ we have
$(2,1,3,\ldots,k-1)\notin\sigma_0$.
This proves condition (1).
As before,
to verify condition (2) let 
$\vect{h}\in(\cl{C}^{(m)})^n$ ($m,n\ge4$)
be such that 
$\vect{h}(\vect{c}_1^m)\in\{\vect{c}_i^n:1\le i<n\}$.
Here $\vect{c}_1^{m}\in B^{m}$, 
$\vect{c}_1^{n}\in B^{n}$, and 
$\vect{c}_2^{n},\ldots,\vect{c}_{n-1}^{n}\notin B^{n}$,
because $2\in B$ and $0\notin B$.
Since $\vect{h}$ preserves $B$, it must be the case that
$\vect{h}(\vect{c}_1^{m})=\vect{c}_1^{n}$.
\end{proof}

\begin{lemma}
\label{lm-centralk-1central}
If $c$ and $d$ are distinct elements of $A$, then 
$\Pol\sigma_c\cap\Pol\sigma_d\notin\FF_A$.
\end{lemma}

\begin{proof}
Let $\cl{C} = \Pol\sigma_c\cap\Pol\sigma_d$.
We may assume without loss of generality that $A=\{0,1,\ldots,k-1\}$,
$c=0$, and $d=2$. We will again use Lemma~\ref{lm-centralk-1-gen} 
to show that $\cl{C}\notin\FF_A$.
In fact, we will show that the tuples $\vect{c}_i^n$ ($1\le i\le n$)
exhibited for Case~2 of the proof of Lemma~\ref{lm-centralk-1subset}
satisfy conditions (1) and (2) of Lemma~\ref{lm-centralk-1-gen} 
(with $n_0=4$ and $l=0$)
for this clone $\cl{C}$ as well.
Since (1) is independent of the choice
of the subclone $\cl{C}$ of $\Pol\sigma_0$, 
there is nothing more to do to prove (1).
It remains to show that condition (2) is satisfied.
Let 
$\vect{h}\in(\cl{C}^{(m)})^n$ ($m,n\ge4$), and assume that
$\vect{h}$ satisfies requirements (i)--(iii) in condition (2), that is,
$\vect{h}(\vect{c}_1^m)\in\{\vect{c}_i^n:1\le i<n\}$,
$\vect{h}(\vect{c}_m^m)=\vect{c}_n^n$, and
$\vect{h}(\bar b)=\bar b$ for all $3\le b\le k-1$.
Since $\cl{C}$ is a subclone of $\Pol\sigma_2$,
$\vect{h}$ preserves $\sigma_2$. In particular,
the $\vect{h}$-image of the $(k-1)$-tuple
$(\vect{c}_1^m,\vect{c}_m^m,\bar3,\ldots,\overline{k-1})\in(\sigma_2)^m$
is a tuple in $(\sigma_2)^n$; that is, 
\[
\bigl(\vect{h}(\vect{c}_1^m),\vect{c}_n^n,\bar 3,\ldots,\overline{k-1}\bigr)
\in(\sigma_2)^n.
\]
By assumption, $\vect{h}(\vect{c}_1^m)\in\{\vect{c}_i^n:1\le i<n\}$;
on the other hand, for $i=2$ we have
$(\vect{c}_2^n,\vect{c}_n^n,\bar 3,\ldots,\overline{k-1})
\notin(\sigma_2)^n$, because
the second coordinate is $(0,1,3,\ldots,k-1)\notin\sigma_2$,
while for $3\le i<n$ we have 
$(\vect{c}_i^n,\vect{c}_n^n,\bar 3,\ldots,\overline{k-1})
\notin(\sigma_2)^n$, because
the last coordinate is $(1,0,3,\ldots,k-1)\notin\sigma_2$.
Thus it must be that $\vect{h}(\vect{c}_1^m)=\vect{c}_1^n$,
as required.
\end{proof}

\begin{lemma}
\label{lm-centralk-1perm}
If $\gamma$ is a nonidentity permutation of $A$ and $c\in A$, then
$\Pol\gamma\cap\Pol\sigma_c\notin\FF_A$.
\end{lemma}

\begin{proof}
Let $\cl{C} = \Pol\gamma\cap\Pol\sigma_c$, and
assume without loss of generality that $A=\{0,1,\ldots,k-1\}$ and
$c=0$. 
Let $d=\gamma(0)$, and let $B=\{a\in A:\gamma(a)=a\}$ 
be the set of fixed points of $\gamma$.
It is easy to verify that 
\[
\sigma_d=\bigl\{\bigl(\gamma(a_1),\ldots,\gamma(a_{k-1})\bigr):
(a_1,\ldots,a_{k-1})\in\sigma_0\bigr\}.
\]
Thus it follows that every operation that preserves $\sigma_0$ and $\gamma$
also preserves $B$ and $\sigma_d$.
Hence $\cl{C}\subseteq\Pol\sigma_0\cap\Pol B\cap\Pol\sigma_d$.
In case $d\not=0$ we get from Lemma~\ref{lm-centralk-1central}
and Proposition~\ref{basic_props2}~(ii)
that $\cl{C}\notin\FF_A$.
If $d=0$, then $0\in B$. Moreover, since 
$\gamma$ is not the identity permutation, $B$ is a proper subset of
$A$.
Therefore Lemma~\ref{lm-centralk-1subset}, combined again with
Proposition~\ref{basic_props2}~(ii), yields that 
$\cl{C}\notin\FF_A$ unless $B=\{0\}$.

So, it remains to consider the case when
$B=\{0\}$, that is, $0$ is the unique fixed point of $\gamma$.
Assume from now on that $\gamma$ satisfies this condition.
To prove that $\cl{C}\notin\FF_A$ holds in this case as well,
we will use Lemma~\ref{lm-centralk-1-gen}, that is,
we will exhibit tuples $\vect{c}_i^n$ that satisfy
conditions (1) and (2) of Lemma~\ref{lm-centralk-1-gen}.

First let $k \geq 4$. Since $0$ is the only fixed point of $\gamma$,
we may assume without loss of generality that $\gamma(2)=3$.
Now let $\vect{c}_i^n$ ($n\ge4$, $1\le i\le n$) be 
the tuples defined in Case~2
of the proof of Lemma~\ref{lm-centralk-1subset}.
We know from that proof that condition (1) is satisfied.
To prove that (2) is satisfied with our current choice of clone $\cl{C}$,
consider any
$\vect{h}\in(\cl{C}^{(m)})^n$ ($m,n\ge4$) that satisfies
$\vect{h}(\bar 3)=\bar 3$, a fragment of requirement (iii) in (2).
Since $(\vect{c}_1^m,\bar3)=(\bar2,\bar3)\in\gamma^m$ and $\vect{h}$ preserves
$\gamma$, we get that 
$\bigl(\vect{h}(\vect{c}_1^m),\bar3\bigr)\in\gamma^n$.
As $\gamma$ is a permutation, it follows that
$\vect{h}(\vect{c}_1^m)=\bar2=\vect{c}_1^n$.

Finally, let $k=3$. Since $0$ is the only fixed point of $\gamma$,
$\gamma$ is the transposition $(1\ 2)$.
For each $n\ge7$ define $n$-tuples $\vect{c}_i^n$ ($1\le i\le n$)
as follows: 
$\vect{c}_1^n=(2,2,\ldots,2,0,1,0)$,
$\vect{c}_2^n=(2,0,\ldots,0,0,0,0)$,
and for $3\le i\le n$, 
$\vect{c}_i^n=\vect{e}_{i-1}^n$
are the tuples from Lemma~\ref{lm-centralk-1-aux}.
Since these tuples, with the exception of $\vect{c}_1^n$, are the same
as those in Case~2 of the proof of Lemma~\ref{lm-centralk-1subset},
we know from that proof that condition (1) holds whenever $j\ge i\ge 2$.
For $i=1$, clearly, 
$(\vect{c}_1^n,\vect{c}_2^n)\in(\sigma_0)^n$,
because in each coordinate the pair is $(2,2)$ or contains a $0$.
However, if $j>2$, then 
$(\vect{c}_1^n,\vect{c}_j^n)\notin (\sigma_0)^n$,
because we have $(2,1)\notin\sigma_0$ either in 
the fourth coordinate (if $j=3$), or in
the first coordinate (if $j>3$).
This proves that (1) holds.
To prove that (2) also holds,
let
$\vect{h}\in(\cl{C}^{(m)})^n$ ($m,n\ge7$) satisfy
requirement (ii) from condition (2), that is,
$\vect{h}(\vect{c}_m^m)=\vect{c}_n^n$.
Since $(\vect{c}_1^m,\vect{c}_m^m)\in\gamma^m$ and $\vect{h}$ preserves
$\gamma$, we get that 
$\bigl(\vect{h}(\vect{c}_1^m),\vect{c}_n^n\bigr)\in\gamma^n$.
As $\gamma$ is a permutation, it follows that
$\vect{h}(\vect{c}_1^m)=\vect{c}_1^n$, completing the proof.
\end{proof}

\begin{lemma}
\label{lm-centralk-1eqrel}
If $\epsilon$ is a nontrivial equivalence relation on $A$ and $c\in A$, then
$\Pol\epsilon\cap\Pol\sigma_c\notin\FF_A$.
\end{lemma}

\begin{proof}
Let $\cl{C} = \Pol(\epsilon)\cap\Pol\sigma_c$, and 
assume without loss of generality that 
$A=\{0,1,\ldots,k\}$, $c=0$, and $2$ is an element of $A$ 
such that $(0,2)\notin\epsilon$, but at least one of the 
$\epsilon$-classes $0/\epsilon$, $2/\epsilon$ is not a singleton.
For $n\ge4$ let $\vect{e}_i^n$ ($2\le i\le n-1$) be the $n$-tuples from
Lemma~\ref{lm-centralk-1-aux}.
The following two properties of these tuples will be important:
\begin{itemize}
\item
$(\vect{e}_i^n,\vect{e}_j^n)\notin\epsilon^n$ if $i\not=j$;
\item
$\vect{e}_i^n/\epsilon$ has an element other than $\vect{e}_i^n$ for each $i$.
\end{itemize}
To verify the first property we may assume that $i<j$, since $\epsilon^n$
is a symmetric relation. If $j=i+1$, then 
$(\vect{e}_i^n,\vect{e}_j^n)\notin\epsilon^n$, because in the $i$-th coordinate
$(2,0)\notin\epsilon$; if $j>i+1$, then
$(\vect{e}_i^n,\vect{e}_j^n)\notin\epsilon^n$, because in the $j$-th 
and $(j+1)$-th coordinates we have the pairs $(1,2)$, $(1,0)$, which cannot
simultaneouly be in $\epsilon$, or else we would get $(2,0)\in\epsilon$.
The second property follows from the assumption that
at least one of the 
$\epsilon$-classes $0/\epsilon$, $2/\epsilon$ is not a singleton.

For $n\ge4$ we now define an $n$-ary operation $f_n$ on $A$
as follows:
\[
f_n(\vect{a})=
\begin{cases}
1 & \text{if $\vect{a}=\vect{e}_2^n$,}\\
2 & \text{if $\vect{a}\in\vect{e}_2^n/\epsilon^n$ but 
             $\vect{a}\not=\vect{e}_2^n$,}\\
1 & \text{if $\vect{a}=\vect{e}_i^n$ $(2<i<n-1)$,}\\
0 & \text{if $\vect{a}\in\vect{e}_i^n/\epsilon^n$ but 
             $\vect{a}\not=\vect{e}_i^n$ $(2<i<n-1)$,}\\
2 & \text{if $\vect{a}=\vect{e}_{n-1}^n$,}\\
0 & \text{if $\vect{a}\in\vect{e}_{n-1}^n/\epsilon^n$ but 
             $\vect{a}\not=\vect{e}_{n-1}^n$,}\\
b & \text{if $\vect{a}=\bar b$ $(3\le b\le k-1)$,}\\
0 & \text{otherwise.}
\end{cases}
\]
The properties of $\vect{e}_i^n$ established in the preceding paragraph
make sure that $f_n$ is well-defined, and that $f_n[\vect{e}_i^n/\epsilon^n]$
is a $2$-element set for each $2\le i\le n-1$.
Our aim is to prove that $f_m \not\fequiv[\cl{C}] f_n$ whenever $m \neq n$, 
which will show that $\cl{C}\notin\FF_A$. 

Suppose that, on the contrary, there exist $m<n$ such that 
$f_m\fequiv[\cl{C}] f_n$.
Hence there exists $\vect{h} \in (\cl{C}^{(m)})^{n}$ such that 
$f_m = f_n \circ \vect{h}$, that is,
$\vect{h}$ maps each set $f_m^{-1}(b)$ ($b\in A$) 
into the set $f_n^{-1}(b)$.
Applying this to $3\le b\le k-1$ we get that
$\vect{h}(\bar b)=\bar b$ for all $3\le b\le k-1$.
Since $\vect{h}$ preserves $\epsilon$, it maps each $\epsilon^m$-class
into an $\epsilon^n$-class.
In particular, let $B_i$ denote the $\epsilon^n$-class containing
$\vect{h}[\vect{e}_i^m/\epsilon^m]$ ($2\le i\le m-1$).
Then
\begin{align*}
\{1,2\}=f_m[\vect{e}_2^m/\epsilon^m]
&{}=
f_n\bigl[\vect{h}[\vect{e}_2^m/\epsilon^m]\bigr]\subseteq
f_n[B_2],\\
\{1,0\}=f_m[\vect{e}_i^m/\epsilon^m]
&{}=
f_n\bigl[\vect{h}[\vect{e}_i^m/\epsilon^m]\bigr]\subseteq
f_n[B_i]
\quad \text{for $2<i<m-1$},\\
\{2,0\}=f_m[\vect{e}_{m-1}^m/\epsilon^m]
&{}=
f_n\bigl[\vect{h}[\vect{e}_{m-1}^m/\epsilon^m]\bigr]\subseteq
f_n[B_{m-1}].
\end{align*}
However, 
it follows from the definition of $f_n$ that for each $\epsilon^n$-class
$B$,
\[
f_n[B]=
\begin{cases}
\{1,2\} & \text{if $B=\vect{e}_2^n/\epsilon^n$,}\\
\{1,0\} & \text{if $B=\vect{e}_i^n/\epsilon^n$ $(2<i<n-1)$,}\\
\{2,0\} & \text{if $B=\vect{e}_{n-1}^n/\epsilon^n$,}\\
C\subseteq\{0,3,\ldots,k-1\} & \text{otherwise}.
\end{cases}
\]
Therefore $B_2=\vect{e}_2^n/\epsilon^n$,
$B_{m-1}=\vect{e}_{n-1}^n/\epsilon^n$, and  for each $2<i<m-1$,
$B_i=\vect{e}_{s_i}^n/\epsilon^n$ for some $s_i$ with $2<s_i<n-1$.
Since $1=f_m(\vect{e}_2^m)=
f_n\bigl(\vect{h}(\vect{e}_2^m)\bigr)$,  $\vect{h}(\vect{e}_2^m)\in B_2$,
and the only element $\vect{a}\in B_2$ with $f_n(\vect{a})=1$ is
$\vect{a}=\vect{e}_2^n$, we get that $\vect{h}(\vect{e}_2^m)=\vect{e}_2^n$.
We conclude similarly that $\vect{h}(\vect{e}_i^m)=\vect{e}_{s_i}^n$
for all $2<i<m-1$, and $\vect{h}(\vect{e}_{m-1}^m)=\vect{e}_{n-1}^n$.
By introducing the notation $s_2=2$ and $s_{m-1}=n-1$ we can write these
results more compactly as follows:
\[
\vect{h}(\vect{e}_i^m)=\vect{e}_{s_i}^n
\ (2\le i\le m-1)
\ \text{where $2=s_2<s_3,.\dots,s_{m-2}<s_{m-1}=n-1$.}
\]

Now we can finish the proof the same way 
as in Lemma~\ref{lm-centralk-1-gen}.
We know from Lemma~\ref{lm-centralk-1-aux} that
$(\vect{e}_i^m,\vect{e}_{i+1}^m,\bar 3,\ldots,\overline{k-1})
\in(\sigma_0)^{m}$ for all $i$ ($2\le i\le m-2$).
Since $\vect{h}$ preserves $\sigma_0$,
the $\vect{h}$-images of these tuples are in $(\sigma_0)^n$; that is,
\[
(\vect{e}_{s_i}^n,\vect{e}_{s_{i+1}}^n,\bar 3,\ldots,\overline{k-1})
\in(\sigma_0)^{n}
\quad
\text{for all $i$ ($2\le i\le m-2$).}
\]
Since $\sigma_0$ is totally symmetric, the tuples obtained by interchanging 
$\vect{e}_{s_i}^n$ and $\vect{e}_{s_{i+1}}^n$ are also members of
$(\sigma_0)^{n}$. 
Thus we get from Lemma~\ref{lm-centralk-1-aux}
that $|s_i-s_{i+1}|\le1$ for all $2\le i\le m-2$.
This implies that 
$(n-1)-2=|s_{n-1}-s_2|\le\sum_{i=2}^{m-2}|s_{j+1}-s_j|\le m-3$,
which contradicts our assumption that $m<n$.
This completes the proof of the lemma.
\end{proof}

\begin{proof}[Proof of Theorem~\ref{thm-intersecs}]
Statement (1) follows from Lemmas~\ref{lm-centralk-Bk-1}--\ref{lm-eqrel-Bk-1}.
In Statement (2) the necessity is a consequence of 
Lemmas~\ref{lm-centralk-Bk-1} and 
\ref{lm-centralk-1subset}--\ref{lm-centralk-1eqrel}, while the sufficiency
was established in Theorem~\ref{thm-centralk-1}.
Statements (3) and (4) are special cases of Theorem~\ref{mainthm-eqrels}
and Theorem~\ref{discr}.
\end{proof}


\subsection*{Acknowledgments}
Part of this work was done while the first author was visiting the University of Colorado at Boulder.


\end{document}